\documentclass[aihp]{imsart}

\RequirePackage{amsthm,amsmath,amsfonts,amssymb}
\RequirePackage[numbers,sort&compress]{natbib}
\RequirePackage[colorlinks,citecolor=blue,urlcolor=blue]{hyperref}
\RequirePackage{graphicx}

\startlocaldefs
\theoremstyle{plain}

\newtheorem{theorem}{Theorem}[section]
\newtheorem{lemma}[theorem]{Lemma}
\theoremstyle{remark}
\newtheorem{definition}[theorem]{Definition}
\newtheorem{hypothesis}{Hypothesis}[section]

\newtheorem{proposition}{Proposition}[section]

\newtheorem{remark}{Remark}[section]

\def\RR{\mathbb{R}}
\def\PP{\mathbb{P}}
\def\EE{\mathbb{E}}
\def\ep{\varepsilon}

\def \eref#1{\hbox{(\ref{#1})}}

\endlocaldefs

\begin{document}

\begin{frontmatter}
\title{Multi-Scale McKean-Vlasov SDEs: Moderate Deviation Principle in Different Regimes}
%\title{A sample article title with some additional note\thanksref{t1}}
\runtitle{MDP for Multi-Scale MVSDEs}
%\thankstext{T1}{A sample additional note to the title.}

\begin{aug}
\author[A]{\inits{F.}\fnms{Wei}~\snm{Hong}\ead[label=e1]{weihong@jsnu.edu.cn}},
\author[B]{\inits{S.}\fnms{Ge}~\snm{Li}\ead[label=e2]{20231007@hbue.edu.cn}\orcid{0000-0003-2740-5107}}
\and
\author[A]{\inits{T.}\fnms{Shihu}~\snm{Li}\ead[label=e3]{shihuli@jsnu.edu.cn}\orcid{0000-0003-1401-4406}}
%%%%%%%%%%%%%%%%%%%%%%%%%%%%%%%%%%%%%%%%%%%%%%
%% Addresses                                %%
%%%%%%%%%%%%%%%%%%%%%%%%%%%%%%%%%%%%%%%%%%%%%%
\address[A]{School of Mathematics and Statistics, Jiangsu Normal University, Xuzhou, 221116, China\printead[presep={,\ }]{e1,e3}}

\address[B]{School of Statistics and Mathematics, Hubei University of Economics, Wuhan, 430205, China\printead[presep={,\ }]{e2}}
\end{aug}

\begin{abstract}
The main aim of this paper is to study the moderate deviation principle for McKean-Vlasov stochastic differential equations   with multiple scales. Specifically, we are interested in the asymptotic estimates of the deviation processes $\frac{X^{\delta}-\bar{X}}{\lambda(\delta)}$ as $\delta\to 0$ in different regimes (i.e.~$\varepsilon=o(\delta)$ and $\varepsilon=O(\delta)$), where $\delta$ stands for the intensity of the noise and $\varepsilon:=\varepsilon(\delta)$ stands for the time scale separation. The rate functions in two regimes are different, in particular,  we show that it  is strongly affected by the noise of the fast component in  latter regime, which is essentially different from the former one and the case of large deviations (cf.~\cite{HLLS}). As a by-product, the explicit representation formulas of the rate functions in all of regimes are also given.
The main techniques are  based on the weak convergence approach and the functional occupation measure approach.
\end{abstract}

%\begin{abstract}[language=french]
%In French. Please do not only use unsupervised automatic translators. If you feel that you cannot provide the French abstract, please just provide the English one and the Editors will provide the translation.
%\end{abstract}

\begin{keyword}[class=MSC]
\kwd[Primary ]{60H10}
\kwd[; secondary ]{60F10}
\end{keyword}

\begin{keyword}
\kwd{McKean-Vlasov SDEs}
\kwd{moderate deviation principle}
\kwd{multi-scale}
\kwd{weak convergence approach}
\kwd{occupation measure}
\end{keyword}

\end{frontmatter}

\tableofcontents
\section{Introduction}

Large deviation principle (LDP)
mainly characterizes the exponential decay of the probability
distributions with respect to certain kinds of extreme or remote tail events, which is one of the important topics in probability theory and has been intensively studied in the literature. On the other hand, the moderate deviation principle (MDP)  bridge the gap between central limit theorem (CLT) and LDP, and in such a study, one is concerned with probabilities of deviations of a smaller order than those captured by LDP.
%The moderate deviation principle (MDP) is one of the important topics in probability theory and has been widely applied in
%many fields such as thermodynamics, statistics, information theory, and engineering.
%We refer interested readers to the classical monographs \cite{bd,fw} for the theory andimportant applications.
There are numerous results concerning the MDP for stochastic systems in the literature. For instance, the MDP for processes with independent increments has been established by De Acosta \cite{De92}, Chen \cite{Chen91} and  Ledoux \cite{Led92}.  The MDP for stochastic differential equations (SDEs) driven by a Poisson random measure was studied by Budhiraja et al. in \cite{BDG16}, and see also \cite{CL04,DXZZ,fw,WZZ15,Wu95} and references therein for more results on MDP within different settings. In the recent work \cite{BW17}, Budhiraja and Wu also studied the MDP for a class of weakly interacting particle systems by employing the well-known weak convergence approach,  see also \cite{DLR,DRW} for the related studies for interacting particle systems.

In the past several decades, large and moderate deviation principle for multi-scale stochastic dynamical systems have also been extensively studied by many experts. In a classical paper \cite{Ki92}, Kifer proved the large deviations bounds for the multi-scale dynamical systems with the fast term given by a Markov process. Liptser \cite{Li96} studied the following multi-scale SDEs,
\begin{equation}\left\{\begin{array}{l}\label{E00}
\displaystyle
d X^{\varepsilon}_t = b(X^{\varepsilon}_t, Y^{\varepsilon}_t)dt+\sqrt{\varepsilon}\sigma(X^{\varepsilon}_t,Y^{\varepsilon}_t)d W^{1}_t, \\
\displaystyle d Y^{\varepsilon}_t =\frac{1}{\varepsilon}f( Y^{\varepsilon}_t)dt+\frac{1}{\sqrt{\varepsilon}}g( Y^{\varepsilon}_t)d W^{2}_t,\\
\displaystyle X^{\varepsilon}_0=x,~Y^{\varepsilon}_0=y,
\end{array}\right.
\end{equation}
where $W^{1}_t$ and $W^{2}_t$ are mutually independent Brownian motions. The author  considered the LDP for the joint distribution of the slow process $X^{\varepsilon}$ and of the empirical process associated with the fast variable $Y^{\varepsilon}$. Recently, the results of \cite{Li96} was significantly extended by Puhalskii \cite{Pu16} to the fully coupled case in the sense that all of the coefficients may depend on both the variables $X^{\varepsilon},Y^{\varepsilon}$ and the diffusion terms may be correlated.
For the study of MDP in the multi-scale case, Guillin \cite{Gu01} established the MDP of inhomogeneous functionals where the ``fast" perturbation is a continuous time ergodic Markov chain. Later, in the work \cite{Gu03}, Guillin  also considered the following multi-scale SDEs
\begin{equation}\label{E01}
d X^{\varepsilon}_t = b(X^{\varepsilon}_t, Y_{t/\varepsilon})dt+\sqrt{\varepsilon}\sigma(X^{\varepsilon}_t,Y_{t/\varepsilon})d W_t,
\end{equation}
where the random environment ($Y_t$) is an exponentially ergodic Markov
process independent of the Wiener process ($W_t$), and proved the
MDP for the averaging principle of $X^{\varepsilon}$. Recently, Gasteratos et al.~\cite{gss} combined the weak convergence approach and the occupation measure approach introduced in \cite{ds} to studied the MDP for a system of
stochastic reaction-diffusion equations with a time-scale separation in slow and fast components and small noise in the slow component. Bezemek and Spiliopoulos \cite{bs1,bs2} also established the LDP and MDP of the empirical laws of multi-scale interacting particle system.
We refer the interested reader to \cite{ds,Gu05,HLL2,HLLS,LS99,ms,SWXY,ve99,ve00} and references therein for more recent results on this topic. We also mention that the multi-scale systems are very common in many fields such as  material sciences, fluids
dynamics, climate dynamics, etc, the reader can refer to \cite{An00,C09,C1,EE,MTV} for more practical background.

In the present work, we are interested in the  asymptotic behavior, in particular the MDP,  of the multi-scale McKean-Vlasov SDEs which can be used to describe stochastic systems whose evolution is influenced by both the microscopic location and the macroscopic distribution of particles, i.e., the coefficients depend  not only on the solution pointwisely but also on its time marginal law. Such kind of models typically arise as the limits of mean-field interaction particle systems, where the coefficients depend on the empirical measure
of the system. This macroscopic behavior is referred to as the {\it propagation
of chaos} by Kac \cite{K56}, which has been widely studied in the literature, see e.g. \cite{HRZ,JW,M96,M1,S91,W1}.

Notice that the multi-scale McKean-Vlasov SDEs can  be also derived by taking the mean-field limits of the multi-scale interaction particle system. For example,
Gomes and Pavliotis \cite{GP} studied the following multi-scale interacting diffusions
\begin{equation}\label{ree1}
d X_t^{\varepsilon,i} =-\nabla V^\varepsilon(X_t^{\varepsilon,i})dt-\frac{\theta}{N}\sum_{j=1}^N\nabla F(X_t^{\varepsilon,i}-X_t^{\varepsilon,j})dt+\sqrt{2\beta^{-1}}d W^i_t,
\end{equation}
where $V^\varepsilon(x):=V(x,x/\sqrt{\varepsilon}):\mathbb{R}\rightarrow\mathbb{R}$ is the confining potential with a fast fluctuating,
$F:\mathbb{R}\rightarrow\mathbb{R}$ is the interaction potential, $\theta,\beta$ are some constants. Let $Y_t^\varepsilon:=X_t^\varepsilon/\sqrt{\varepsilon}$, the mean field limit ($N\rightarrow\infty$) of (\ref{ree1}) is the following multi-scale McKean-Vlasov SDEs
\begin{equation*}
\left\{\begin{array}{l}
\displaystyle
d X^{\varepsilon}_t = \Big\{-\partial_x V(X_t^{\varepsilon},Y_t^{\varepsilon})-{\frac{1}{\sqrt{\varepsilon}} } \partial_y V(X_t^{\varepsilon},Y_t^{\varepsilon})-\theta \mathbb{E}\big[\nabla F(x-X_t^{\varepsilon})\big]|_{x=X_t^{\varepsilon}}\Big\}dt+\sqrt{2\beta^{-1}}d W_t, \\
d Y^{\varepsilon}_t = \Big\{-\frac{1}{\varepsilon}\partial_y V(X_t^{\varepsilon},Y_t^{\varepsilon})-\frac{1}{\sqrt{\varepsilon}}\partial_x V(X_t^{\varepsilon},Y_t^{\varepsilon})-\frac{\theta}{\sqrt{\varepsilon}} \mathbb{E}\big[\nabla F(x-X_t^{\varepsilon})\big]|_{x=X_t^{\varepsilon}}\Big\}dt+\sqrt{\frac{2\beta^{-1}}{\varepsilon}}d W_t.
\end{array}\right.
\end{equation*}
The authors in \cite{GP} proved that, although the mean field limit ($N\rightarrow\infty$) and homogenization limit ($\varepsilon\rightarrow0$) commute for finite time,  they do not commute in the long time limit in general. Delgadino at al.~in \cite{DGP} studied the following multi-scale interacting particle system
\begin{equation}\label{ree2}
d X_t^{\varepsilon,i} =-\frac{1}{\varepsilon} \nabla V(\varepsilon^{-1}X_t^{\varepsilon,i})dt-\frac{1}{N}\sum_{j\neq i}^N\frac{1}{\varepsilon}\nabla F(\varepsilon^{-1}(X_t^{\varepsilon,i}-X_t^{\varepsilon,j}))dt+\sqrt{2\beta^{-1}}d W^i_t,
\end{equation}
where $F:\mathbb{R}^{d}\rightarrow\mathbb{R}$ and $V:\mathbb{R}^{d}\rightarrow\mathbb{R}$ are smooth and 1-periodic interaction and confining potentials.
They proved that the mean field and homogenization limits  do not commute if the mean field system constrained to the torus undergoes a
phase transition, i.e. if multiple steady states exist.

We place ourselves in the setting of a multi-scale system as in \cite{HLLS,RSX1}  and consider the following multi-scale McKean-Vlasov stochastic dynamical system
\begin{equation}\left\{\begin{array}{l}\label{E2}
\displaystyle
d X^{\delta}_t = b(X^{\delta}_t, \mathcal{L}_{X^{\delta}_t}, Y^{\delta}_t)dt+\sqrt{\delta}\sigma(X^{\delta}_t, \mathcal{L}_{X^{\delta}_t})d W^{1}_t, \\
\displaystyle d Y^{\delta}_t =\frac{1}{\varepsilon}f(X^{\delta}_t, \mathcal{L}_{X^{\delta}_t}, Y^{\delta}_t)dt+\frac{1}{\sqrt{\varepsilon}}g( X^{\delta}_t, \mathcal{L}_{X^{\delta}_t}, Y^{\delta}_t)d W^{2}_t,\\
\displaystyle X^{\delta}_0=x,~Y^{\delta}_0=y,
\end{array}\right.
\end{equation}
where $\mathcal{L}_{X^{\delta}_t}$ is the law of $X^{\delta}_t$,
$\delta$ describes the intensity of the noise and $\varepsilon:=\varepsilon(\delta)$ describes the ratio of the time scale between the slow component $X^{\delta}$ and fast component $Y^{\delta}$. Here we define an $\RR^{d_1+ d_2}$-valued standard Brownian motion $W$  on a complete filtration probability space $(\Omega, \mathcal{F}, \{\mathcal{F}_{t}\}_{t\geq0}, \mathbb{P})$
 such that we can choose the projection operators $P_1:\RR^{d_1+ d_2}\to \RR^{d_1}$, $P_2:\RR^{d_1+ d_2}\to \RR^{d_2}$, and
\begin{equation*}\label{bm}
W_t^1:=P_1W_t,~W_t^2:=P_2W_t
\end{equation*}
are independent
$d_1$ and $d_2$ dimensional standard Brownian motions  respectively.

{Note that the system \eref{E2} does not rely on the law of the fast variable $\mathcal{L}_{Y^{\delta}_t}$  and the slow diffusion coefficient $\sigma$ does not depend on $Y^{\delta}_t$, mainly due to some {technical restrictions in the proofs}, for example, we need to utilize some results and estimates from \cite{RSX1}. Recently, the authors in \cite{LWX} considered the diffusion approximation problem of multi-scale McKean-Vlasov SDEs, which depends on $\mathcal{L}_{Y^{\delta}_t}$. However, it should be pointed out that the system in \cite{LWX} is not fully-coupled and cannot cover the model considered in this work.  Specifically, the major parts in fast component in \cite{LWX} do not depend on the slow process $X^{\varepsilon}_t$, which differ from the results of multi-scale SDEs in the distribution independent case. Additionally,
there is no order $\frac{1}{\sqrt{\varepsilon}}$ drift term in the dynamics of $X^{\delta}$ in \eref{E2}. Intuitively, if we add a singular (or homogenization) term, e.g., $\frac{1}{\sqrt{\varepsilon}}K(X^{\delta}_t,\mathcal{L}_{X^{\delta}_t},Y^{\delta}_t)$ in system (\ref{E2}), then the following term will appear in the formulation of the deviation process $Z_t^{\delta}$ (see \eref{z01} blow):
\begin{equation}\label{es1+}
\frac{1}{\lambda(\delta)\sqrt{\varepsilon}}K(X^{\delta}_t,\mathcal{L}_{X^{\delta}_t},Y^{\delta}_t).
\end{equation}
However, as $\delta\to 0$ (hence, $\varepsilon\to 0 $), the term (\ref{es1+}) will blow up, thus we could not involve such a homogenization term in the MDP topic.
In addition, Morse and Spiliopoulos \cite{ms} considered the MDP for the slow-fast system with additional homogenization term $\frac{\sqrt{\delta}}{\sqrt{\varepsilon}}K(X^{\delta}_t,Y^{\delta}_t)$.
In this case, one should not expect in general to have a strong convergence rate, or $L^2$ convergence at all, if one added the such term in the drift of the slow component and/or if one allowed for the diffusion
coefficient of the slow component to depend on the state of the fast component (see  \cite[Remark 4.4]{blp}).  However, as evidenced in \cite{ms} a weak convergence rate is sufficient in the setting without dependence on the law of the process to derive the MDP, which deserves further investigations for McKean-Vlasov case in the future work. 

For the study of asymptotic behavior for {multi-scale McKean-Vlasov SDEs}, R\"{o}ckner et al.~\cite{RSX1} first studied the averaging principle of system \eref{E2} for fixed $\delta=1$, more precisely,
they proved that the slow component in \eref{E2} converges strongly to the solution of the averaged equation, as $\varepsilon\rightarrow0$, with optimal convergence rate $1/2$,
which could be seen as the classical functional law of large numbers. The interested readers can see also \cite{bs3,HLL4,HLS,LWX,LX23,SXW,XLLM} for the averaging principle results for multi-scale McKean-Vlasov SDEs within different settings. Very recently, based on the results of \cite{RSX1}, the first and third named authors of this work with their collaborators in \cite{HLLS} further studied the CLT and LDP for system \eref{E2}.

Thus in order to bridge the gap between CLT scale and LDP scale as mentioned before, the main aim of this paper is to investigate the moderate deviations of $X^{\delta}_t$ in \eref{E2} from the averaged equation $\bar{X}_t$, as $\delta\rightarrow0$ (hence $\varepsilon\rightarrow0$). That is, consider the asymptotic behavior of the trajectory
\begin{equation}\label{z01}
Z_t^{\delta}:=\frac{X^{\delta}_t-\bar{X}_t}{\lambda(\delta)},
\end{equation}
where
\begin{equation*}\label{ave}
\frac{d\bar{X}_t}{dt}=\bar{b}(\bar{X}_t,\mathcal{L}_{\bar{X}_t}),~\bar{X}_0=x\in\RR^n,
\end{equation*}
the coefficient $\bar{b}$ is defined by
\begin{equation*}
\bar{b}(x,\mu)=\int_{\RR^m}b(x,\mu,y)\nu^{x,\mu}(dy),
\end{equation*}
and $\nu^{x,\mu}$  denotes the unique invariant measure associated with the generator of fast variable $\mathbf{L}^{2}_{x,\mu}$ (cf. \eref{inf1}). Here $\lambda(\delta)$ is a deviation scale that strongly influences the asymptotic behavior of $Z_t^{\delta}$, which satisfies
\begin{equation}\label{e10}
\lambda(\delta)\rightarrow0,\quad \delta/\lambda^2(\delta)\rightarrow0,\quad \mbox{as}\quad \delta\rightarrow0.
\end{equation}
It should be  pointed out that depending on the order that $\delta$ and $\varepsilon$ converge to zero, there are three different regimes of interaction, i.e.
\begin{equation}\label{regime}
\lim_{\delta\rightarrow0}\varepsilon/\delta=\left\{
  \begin{array}{ll}
    0, & \hbox{Regime 1;} \\
    \gamma\in(0,\infty), & \hbox{Regime 2;} \\
    \infty, & \hbox{Regime 3.}
  \end{array}
\right.
\end{equation}
In \cite{HLLS}, based on the time discretization technique and the weak convergence approach, we only proved that the LDP holds for system (\ref{E2}) in the case of Regime 1.
However, the main results of this article will show that $Z^{\delta}$ satisfies the LDP (i.e.~$X^{\delta}$ satisfies the MDP) in $C([0,T];\mathbb{R}^n)$  both for Regimes 1 and 2. As a by-product, the explicit representation formulas for the rate functions in Regimes 1 and 2
are  also given (see Theorems \ref{t3} and \ref{thj1} for the details).

We want to mention that large and moderate deviation principles for McKean-Vlasov SDEs have been extensively investigated in the single-scale case in recent years. For example, Herrmann et al. \cite{hp} considered the small noise limit of a self-stabilizing diffusion, and they first proved the LDP in path space with the uniform topology for the McKean-Vlasov SDEs with small Gaussian perturbation. Dos Reis et al. \cite{dst} obtained Freidlin-Wentzell's LDP results in both uniform and H\"{o}lder topologies for a more general type of McKean-Vlasov SDEs via assuming that the coefficients satisfy some extra time H\"{o}lder continuity conditions. Suo and Yuan \cite{sy} study the CLT and MDP for McKean-Vlasov SDEs under Lipschitz condition. In recent work \cite{liuw}, the authors proved the Freidlin-Wentzell's LDP and MDP for a class of McKean-Vlasov SDEs with small L\'{e}vy jumps.

Nevertheless,  there is no result concerning the  MDP of Freidlin-Wentzell type for multi-scale McKean-Vlasov SDEs in the literature.  Our strategies here are mainly based on the  weak convergence approach developed by Budhiraja et
al. in~\cite{bd3,BDM,de} and the occupation measure approach developed by Dupuis et
al. in~\cite{ds,gss}.
Compared with the related results for single-scale McKean-Vlasov SDEs (cf.~\cite{hp,dst,liuw,sy}), here we need to deal with the perturbations of fast variables and characterize the limits of the control equations (cf.~(\ref{esZ})) associated with $Z^{\delta}$ at different regimes, which is inherently difficult and quite different from the single-scale case. Moreover, it is interesting to see that the rate functions in two regimes are quite different (see Theorems \ref{t3} and \ref{thj1}). In particular, we show that at the second regime, the rate function depends not only on the diffusion coefficient $\sigma$ in (\ref{E2}), but also on the diffusion coefficient $g$ of fast component and the order $\gamma$ that rates the {ratio} of $\delta$ and $\varepsilon$ converging to zero.

\vspace{1mm}
In the following, we would like to give some detailed explanations of difficulties for different regimes and our strategies in the proof. In Regime 1, {since the homogenization occurs more quickly than the vanishing noise, we can adopt the weak convergence criterion (cf.~(i) and (ii) in Hypothesis \ref{h2} below) to give a direct way to prove the MDP for multi-scale McKean-Vlasov SDEs (\ref{E2}). Different from the LDP case (cf.~\cite{HLLS}), it is nontrivial to characterize the convergence limit of the control problem (\ref{esZ}) due to the MDP scale $\lambda(\delta)$. To this end, we will introduce an auxiliary Poisson equation depending on parameter measures (cf.~(\ref{PE})) and the corresponding regularities of its solutions
on Wasserstein space. Another difficulty lies in establishing some crucial estimates of the controlled processes \eref{e9} (cf. Lemmas \ref{PMY}), which is accurate in the sense of the MDP scale (\ref{e10}), and  that it is enough for our use in proving the convergence of solutions to the control problem (\ref{esZ}) and also has some independent interest.}

However, it should be pointed out that the above proof strategy is not applicable to {the case of Regime 2, where the } homogenization and the vanishing
noise occur at rates of the same order as shown in \cite{HLLS} (cf. Remark \ref{rem+1} for more details).
To solve this problem, we adopt the functional occupation measure approach and introduce
the notion of ``viable pair'' in the McKean-Vlasov structure, which is originally introduced in \cite{de} and then significantly developed in the recent works \cite{gss,hu2}, and is a pair of a trajectory and measure that captures both the
limit averaging dynamics of the {controlled processes \eref{e9}} and the invariant measure of the controlled fast
process {$ Y^{\delta,h^\delta}$} in \eref{e9}. We remark that it is a correct way
to study this problem because the  behavior of fast component will not converge pathwisely to
anything, but its occupation measure will converge to a limiting measure. We will show that the occupation measures $\{\mathbf{P}^{\delta,\Delta}\}_{\delta>0}$ defined by  (\ref{occupation}) are tight in a space of probability measures (thus there exists a weakly convergent subsequence),
and then, by proving the Laplace principle upper and lower bounds in terms of the variational representation formulas for functionals of a Brownian motion, the corresponding MDP is obtained. It should be pointed out that in order to prove the Laplace principle upper bound, we will replace the control by a feedback form and establish the nearly optimal control to achieve the aim, which is quite different from the proof of the Laplace principle lower bounds.
Finally, it is worth noting that the above two methods does not work for the case of Regime 3 anymore, such scale relationship makes the analysis very challenging and it is also an open issue even in the {distribution independent} case.

{To the best of our knowledge, this is the first result studying the MDP of multi-scale McKean-Vlasov SDEs in the literature. Both the form of the moderate deviations rate function and the method of proof differ between the two regimes. One of major contributions of this work is to see how the occupation measure and viable pair framework work on the context of McKean-Vlasov SDEs.  Furthermore, compared with previous works in the context of MDP for standard (i.e.~distribution independent) multi-scale SDE  (e.g.~\cite{gss,sm}), {there are also some results of independent interest in our paper.} For example,  we also drop a crucial condition, which is imposed in \cite{gss,sm}, of the fast component being
\begin{equation}\label{e4+}
f(x,y):=-\rho y+\zeta(x,y),
\end{equation}
where $\rho$ is a positive constant. In fact, the authors in \cite{gss,sm} essentially required (\ref{e4+}) to transform the fast equation to be a mild solution form (cf.~the proof in Lemma 3.1 of \cite{sm} and in Lemma 4.1 of \cite{gss}), in order to get the estimates of the controlled process associated with the fast component. In our paper, we give a different but more general proof (see Lemmas 3.1 and 3.2), which does not depend on the structure (\ref{e4+}), and we believe that it is also useful in the infinite-dimensional SPDE case. Moreover, in the case of Regime 1, we directly apply the powerful weak convergence criterion, which greatly simplifies the whole proof and allows us to remove the uniform ellipticity of $\sigma\sigma^*$. In the existing works, such a uniform ellipticity is strongly required in the study of MDP for multi-scale system in order to transform the control to the feedback form (see e.g.~\cite{gss}). As a cost, the slow diffusion coefficient $\sigma$ in our model cannot depend on $Y^{\delta}_t$.}

The rest of the paper is organized as follows. In Section 2, we  introduce some notations and definitions and state the main results. In Section 3, we first give some a priori estimates which are crucial in the weak convergence analysis.  Sections 4 and 5 are devoted to  the proof of the MDP for  Regimes 1 and 2 respectively.  We conclude with Section 6 where we put the appendix.

Note that throughout this paper $C$ and $C_T$  denote positive constants which may change from line to line, where
the subscript $T$  is used to emphasize that the constant depends on certain parameter.

\section{Definitions and main results}\label{sec.prelim}

\subsection{Definitions and main assumptions}
First we introduce some notations that will be used frequently throughout the present work. Denote by $|\cdot|$ the Euclidean vector norm, $\langle\cdot, \cdot\rangle$ the Euclidean inner product and $\|\cdot\|$ the matrix norm or the operator norm if there is no confusion possible. The tensor product $\mathbb{R}^{l_1}\otimes\mathbb{R}^{l_2}$ represents the space of all ${l_1}\times {l_2}$-dimensional matrix for $l_1,l_2\in \mathbb{N}_{+}$.
Define $\mathcal{P}$ the set of all probability measures on $(\RR^n, \mathcal{B}(\RR^n))$ and $\mathcal{P}_k$ ($k\geq1$) by
$$
\mathcal{P}_k:=\Big\{\mu\in \mathcal{P}: \mu(|\cdot|^k):=\int_{\RR^n}|x|^k\mu(dx)<\infty\Big\},
$$
 then $\mathcal{P}_k$ is a Polish space under the $L^k$-Wasserstein distance
$$
\mathbb{W}_k(\mu_1,\mu_2):=\inf_{\pi\in \mathcal{C}_{\mu_1,\mu_2}}\left[\int_{\RR^n\times \RR^n}|x-y|^k\pi(dx,dy)\right]^{1/k}, \quad \mu_1,\mu_2\in\mathcal{P}_k,
$$
where $\mathcal{C}_{\mu_1,\mu_2}$ is the set of all couplings for $\mu_1$ and $\mu_2$.

%The Cameron-Martin space $\mathcal{H}_0$ associated with $\{W_t,t\in[0,T]\}$ is given as follows
%\begin{eqnarray*}\label{CM1}
%\mathcal{H}_0:=
%\!\!\!\!\!\!\!\!&&\Bigg\{h:[0,T]\to\RR^{d_1+ d_2}\Big|~\text{there exists}~\dot{h}\in L^2([0,T];\RR^{d_1+d_2})
%\nonumber \\
%\!\!\!\!\!\!\!\!&&~~~~~\text{such that}~ h_t=\int_0^t\dot{h}_s d s,t\in[0,T] \Bigg\},
%\end{eqnarray*}
%where the $\dot{h}$ denotes the weak derivative of $h$, then $\mathcal{H}_0$ is a Hilbert space with the scalar product
%$$\langle h^1, h^2\rangle_0 :=\int_0^T\langle\dot{ h}^1_t,\dot{ h}^2_t\rangle d t .$$

Let us define the following sets which are frequently used in the theory of LDP,
$$\mathcal{A}:=\left\lbrace  h:  h\  \text{is  $\RR^{d_1+ d_2}$-valued
 $\mathcal{F}_t$-predictable process and}\
  \int_0^T|{ h}_s|^2 d s<\infty,~\mathbb{P}\text{-a.s.}\right\rbrace, $$
$$S_M:=\left\lbrace  h\in L^2([0,T];\RR^{d_1+ d_2}):
\int_0^T|{ h}_s|^2  ds\leq M
 \right\rbrace,$$
and
$$\mathcal{A}_M:=\Big\{ h\in\mathcal{A}:  h(\omega)\in S_M,~\mathbb{P}\text{-a.s.}\Big\}.$$

We now collect  the definitions of LDP and Laplace principle from \cite{ds}. Let $\{X^\delta\}_{\delta>0}$ denote a family of
random variables defined on a probability space
$(\Omega,\mathcal{F},\mathbb{P})$ taking values in a Polish
space $\mathcal{E}$. Moreover, let {$\lambda(\delta)$} be a positive real-valued function going to 0 as $\delta\to 0$. The rate function is defined as follows.

\begin{definition}(Rate function) A function $I: \mathcal{E}\to [0,+\infty]$ is called
a rate function if $I$ is lower semicontinuous. Moreover, a rate function $I$
is called a {\it good rate function} if  the level set $\{x\in \mathcal{E}: I(x)\le
K\}$ is compact for each constant $K<\infty$.
\end{definition}

\begin{definition}(Large deviation principle) The random variable family
 $\{X^\delta\}_{\delta>0}$ is said to satisfy
 the large deviation principle on $\mathcal{E}$ with speed { $\lambda(\delta)$} and rate function
 $I$ if  the following lower and upper bound conditions hold.

(i) (LDP lower bound) For any open set $G\subset \mathcal{E}$,
$$\liminf_{\delta\to 0}
 {\lambda(\delta)} \log \mathbb{P}(X^{\delta}\in G)\geq -\inf_{x\in G}I(x).$$

(ii) (LDP upper bound) For any closed set $F\subset \mathcal{E}$,
$$ \limsup_{\delta\to 0}
   {\lambda(\delta)}\log \mathbb{P}(X^{\delta}\in F)\leq
  -\inf_{x\in F} I(x).
$$
\end{definition}

\begin{definition}(Laplace principle)
Let $\{X^\delta\}_{\delta>0}$   be a family of random variables taking values in a Polish space $\mathcal{E}$ and let $I$ be a rate function. We say that $\{X^\delta\}_{\delta>0}$  satisfies the Laplace principle with rate function $I$  if for every bounded and continuous function $\Lambda: \mathcal{E}\rightarrow \mathbb{R}$
\begin{equation}\label{req1}
 \lim_{\delta\rightarrow0}-{\lambda(\delta)}\log\mathbb{E}\Big[\exp\Big\{-{\lambda^{-1}(\delta)}\Lambda(X^\delta)\Big\}\Big]=\inf_{x\in \mathcal{E}}\Big[I(x)+\Lambda(x)\Big].
\end{equation}
\end{definition}

In the present work, we assume that the maps
\begin{eqnarray*}
&&b: \RR^n\times\mathcal{P}_2\times\RR^m \rightarrow \RR^{n};\\
&& \sigma: \RR^n\times \mathcal{P}_2\rightarrow \RR^{n}\otimes\RR^{d_1};\\
&&f:\RR^n\times\mathcal{P}_2\times\RR^m\rightarrow \RR^{m};\\
&&g:\RR^n\times\mathcal{P}_2\times\RR^m\rightarrow \RR^{m}\otimes\RR^{d_2}
\end{eqnarray*}
satisfy the following conditions.

\vspace{2mm}
\begin{enumerate}

\item [$({\mathbf{A}}{\mathbf{1}})$]\label{A1} Suppose that there exist constants $C,\kappa>0$ such that for all $x,x_1,x_2\in\RR^n, \mu,\mu_1,\mu_2\in \mathcal{P}_2, y,y_1,y_2\in\RR^m$,
\begin{eqnarray}
\!\!\!\!\!\!\!\!&&|b(x_1, \mu_1, y_1)-b(x_2, \mu_2, y_2)|+\|\sigma(x_1,\mu_1)-\sigma(x_2,\mu_2)\|\nonumber\\
\!\!\!\!\!\!\!\!&&\leq C\big[|x_1-x_2|+|y_1-y_2|+\mathbb{W}_2(\mu_1, \mu_2)\big] , \label{A11}
\end{eqnarray}
\begin{eqnarray}
\!\!\!\!\!\!\!\!&&|f(x_1,\mu_1, y_1)-f(x_2, \mu_2,y_2)|+\|g(x_1, \mu_1,y_1)-g(x_2, \mu_2, y_2)\|\nonumber\\
\!\!\!\!\!\!\!\!&&\leq C\big[|x_1-x_2|+|y_1-y_2|+\mathbb{W}_2(\mu_1, \mu_2)\big].\label{A21}
\end{eqnarray}
Moreover,
\begin{equation}
2\langle f(x,\mu, y_1)-f(x, \mu,y_2), y_1-y_2\rangle\!+3\|g(x, \mu,y_1)-g(x,\mu, y_2)\|^2\!\leq -\kappa|y_1-y_2|^2.\label{sm}
\end{equation}

\item [$({\mathbf{A}}{\mathbf{2}})$]\label{A2}
Suppose that $b\in C_b^{2,(1,1),2}(\RR^n\times \mathcal{P}_2\times\RR^m;\RR^n)$, $\sigma\in C_b^{2,(1,1)}(\RR^n\times \mathcal{P}_2;\RR^{n}\otimes\RR^{d_1})$, $f\in C^{2,(1,1),2}_b(\RR^n\times \mathcal{P}_2\times\RR^m;\RR^m)$ and $g\in C^{2,(1,1),2}_b(\RR^n\times \mathcal{P}_2\times\RR^m;\RR^{m}\otimes\RR^{d_2})$, where the notations $C_b^{2,(1,1)},C_b^{2,(1,1),2}$ are defined in Definitions \ref{de1} and \ref{de2} respectively.  Moreover, there exist constants $C>0$  and $\gamma_1\in (0,1]$ such that for all $y_1,y_2\in\RR^m$,
\begin{eqnarray*}
\sup_{x\in\RR^n,\mu\in\mathcal{P}_2}\|\partial_{\mu} F(x, \mu, y_1)-\partial_{\mu} F(x, \mu, y_2)\|_{L^2(\mu)}\leq C|y_1-y_2|^{\gamma_1} , \label{A40}
\end{eqnarray*}
\begin{eqnarray*}
\sup_{x\in\RR^n,\mu\in\mathcal{P}_2}\|\partial^2_{xx} F(x, \mu, y_1)-\partial^2_{xx} F(x, \mu, y_2)\|\leq C|y_1-y_2|^{\gamma_1} , \label{A41}
\end{eqnarray*}
\begin{eqnarray*}
\sup_{x\in\RR^n,\mu\in\mathcal{P}_2}\|\partial^2_{xy} F(x, \mu, y_1)-\partial^2_{xy} F(x, \mu, y_2)\|\leq C|y_1-y_2|^{\gamma_1} , \label{A42}
\end{eqnarray*}
\begin{eqnarray*}
\sup_{x\in\RR^n,\mu\in\mathcal{P}_2}\|\partial^2_{yy} F(x, \mu, y_1)-\partial^2_{yy} F(x, \mu, y_2)\|\leq C|y_1-y_2|^{\gamma_1} , \label{A431}
\end{eqnarray*}
\begin{eqnarray*}
\sup_{x\in\RR^n,\mu\in\mathcal{P}_2}\|\partial_{\mu}\partial_{y} F(x, \mu, y_1)-\partial_{z}\partial_{\mu} F(x, \mu, y_2)\|_{L^2(\mu)}\leq C|y_1-y_2|^{\gamma_1} ,
\end{eqnarray*}
\begin{eqnarray*}
\sup_{x\in\RR^n,\mu\in\mathcal{P}_2}\|\partial_{z}\partial_{\mu} F(x, \mu, y_1)-\partial_{z}\partial_{\mu} F(x, \mu, y_2)\|_{L^2(\mu)}\leq C|y_1-y_2|^{\gamma_1} , \label{A44}
\end{eqnarray*}
%\begin{eqnarray*}
%\sup_{x\in\RR^n,\mu\in\mathcal{P}_2,y\in\RR^m}\|\partial_{\mu}\partial_y F(x,\mu,y)\|_{L^2(\mu)}\leq C , \label{A45}
%\end{eqnarray*}
where $F$ represents $b, f$ and $g$ respectively {and the notations $L^2(\mu)$ and the $\partial_{\mu}$ are defined in Appendix \ref{appendix1}.}

\item [$({\mathbf{A}}{\mathbf{3}})$]\label{A3}
Suppose that $\sigma$, $f$ and $g$ satisfy
\begin{equation*}\label{32}
\sup_{x\in\RR^n}\|\sigma(x,\delta_0)\|<\infty, ~\sup_{x\in\RR^n,\mu\in\mathcal{P}_2}|f(x,\mu,0)|<\infty, ~
\sup_{x\in\RR^n,\mu\in\mathcal{P}_2,y\in\RR^m}\|g(x,\mu,y)\|<\infty.
\end{equation*}

\item [$({\mathbf{A}}{\mathbf{4}})$]\label{A4}
There exist $c_1,c_2>0$ such that for all $x\in\RR^n$ and $\mu\in\mathcal{P}_2$, $\xi\in\mathbb{R}^n$,
$$c_1|\xi|^2\leq\langle\sigma\sigma^*(x,\mu)\xi,\xi\rangle\leq c_2|\xi|^2.$$
\end{enumerate}

We give some comments for the conditions above.
\begin{remark}\label{re1}
(i) Conditions \eref{A11} and (\ref{A21}) is used to guarantee the existence and uniqueness of strong solutions to system \eref{E2}.
 Conditions  \eref{sm} and $({\mathbf{A}}{\mathbf{3}})$ imply that for any $\beta\in (0,\kappa)$, there exists $C_{\beta}>0$  such that for any $x\in\RR^n, y\in\RR^m$, $\mu\in\mathcal{P}_2$,
\begin{eqnarray}
 2\langle f(x,\mu,y), y\rangle+3\|g(x,\mu, y)\|^2\leq -\beta|y|^2+C_{\beta}.\label{RE3}
\end{eqnarray}
{In addition, the dissipative condition \eref{sm} is also used to guarantee the existence and uniqueness of invariant measures to the frozen equation (\ref{FEQ2}).}
%If $g$ satisfies the sublinear growth w.r.t.~variable $y$, i.e. there is $\kappa\in[0,1)$ such that
%\begin{eqnarray*}
%\|g(x, \mu,y)\|\!\leq C(1+|x|+(\mu(|\cdot|^2))^{\frac{1}{2}}+|y|^\kappa),
%\end{eqnarray*}
%then (\ref{sm}) could  be modified into a more natural form as follows
%\begin{equation*}
%2\langle f(x,\mu, y_1)-f(x, \mu,y_2), y_1-y_2\rangle\!+\|g(x, \mu,y_1)-g(x,\mu, y_2)\|^2\!\leq -\gamma|y_1-y_2|^2.
%\end{equation*}

{(ii) Condition $({\mathbf{A}}{\mathbf{2}})$ is mainly used to derive some regularities of solutions to Poisson equation (\ref{PE}) (cf.~\cite[Proposition 4.1]{RSX1}), which plays a crucial role in characterizing the rate function of the MDP to system \eref{E2}.
Note that \cite{HLLS} imposed some stronger regularity conditions in the study of the central limit theorem, which is mainly used to show the convergence of  the quadratic variational process
 (cf.~Lemma 3.4 in \cite{HLLS}),
whereas in the proof of  MDP we do not follow the argument of  \cite{HLLS} and thus do not need e.g.~the condition (2.10) in \cite{HLLS}. From the same reason, we only need a weaker condition (\ref{sm}) compared with \cite{HLLS}.}

(iii) Condition $({\mathbf{A}}{\mathbf{3}})$ is a key assumption  to estimate $\mathbb{E}\Big[\sup_{t\in [0, T]}|Y_{t}^{\delta,h^\delta}|^{2}\Big]$ of controlled process $Y^{\delta,h^\delta}$ (see details in Lemma \ref{PMY} below).
Condition $({\mathbf{A}}{\mathbf{4}})$ is essentially required for the use of the ergodic theorem (see (\ref{eser}) below) and for the regularity of $Q_2$ and hence $\bar{h}$ (see Subsection \ref{app2} in Appendix).

\end{remark}

\subsection{Weak convergence analysis}
In this subsection, we intend to review the weak convergence approach  that is a powerful technique to study the LDP for various stochastic dynamical systems and has been systematically developed in \cite{bd,de}. The cores of this approach are that the LDP is equivalent to the Laplace principle (cf.~\cite{ds}), and the use of a variational representation of exponential functionals of Brownian motiones (cf.~\cite{bd2,bd3}). Thus instead of proving the LDP  we can prove the Laplace principle.

{Let $\mathcal{E}$ be the  space of all continuous functions on $\mathbb{R}^{d}$.} In the present work, we aim to prove that $\{Z^\delta\}_{\delta>0}$ defined by (\ref{z01}) satisfy the Laplace principle with speed $\frac{\delta}{\lambda^2(\delta)}$, i.e. for every bounded and continuous function $\Lambda: \mathcal{E}\rightarrow \mathbb{R}$
\begin{equation}\label{req1}
 \lim_{\delta\rightarrow0}-\frac{\delta}{\lambda^2(\delta)}\log\mathbb{E}\left[\exp\left\{-\frac{\lambda^2(\delta)}{\delta}\Lambda(Z^\delta)\right\}\right]=\inf_{x\in \mathcal{E}}\Big[I(x)+\Lambda(x)\Big].
\end{equation}
The derivation of the Laplace principle  is based on a variational representation for functionals of Brownian motions that allows to rewrite the prelimit on the left hand side of (\ref{req1}). Let  $F(\cdot)$ a bounded and measurable real-valued function defined on $C([0,T];\mathbb{R}^{d})$. By \cite{bd2,bd3}, we have
\begin{equation*}\label{req2}
-\log\mathbb{E}\Big[\exp\Big\{-F(W)\Big\}\Big]=\inf_{h\in\mathcal{A}}\mathbb{E}\left[\int_0^T|h_s|^2ds+F\left(W+\int_0^\cdot h_sds\right)\right],
\end{equation*}
where $W$ is a standard $d$-dimensional Brownian motion.

In the current setting, let $W=(W^{1},W^{2})$ and $d=d_1+d_2$. {Under the assumption $(\mathbf{A1})$, since both the coefficients of slow and fast components are globally Lipschitz continuous, the strong well-posedness result has been proved in Theorem 2.2 in
\cite{RSX1}.} Then by a  decoupled argument (cf.~Section 4.1 in \cite{HLLS}), we know that there exists a measurable map $\mathcal{G}^\delta: C([0,T]; \RR^{d_1+d_2})\rightarrow C([0,T]; \RR^n)$ such that we have the representation
$$X^{\delta}=\mathcal{G}^{\delta}\big(\sqrt{\delta}W_{\cdot}\big).$$
Then set
\begin{equation}\label{e11}
\Upsilon^{\delta}(\cdot):=\frac{\mathcal{G}^{\delta}(\cdot)-\bar{X}}{\lambda(\delta)}.
\end{equation}
By the properties of the map $\mathcal{G}^{\delta}$, we can get that $\Upsilon^{\delta}$ is a measurable map from $C([0,T]; \RR^{d_1+d_2})$ to $C([0,T]; \RR^n)$ such that
$$Z^{\delta}=\Upsilon^{\delta}(\sqrt{\delta}W_{\cdot}).$$
For any $ h^\delta\in \mathcal{A}_M$, let us define
$$Z^{\delta, h^\delta}:=\Upsilon^\delta\Big(\sqrt{\delta}W_{\cdot}
+\lambda(\delta)\int_0^{\cdot}h^\delta_sd s\Big),$$
then it is the solution of the following stochastic control problem
\begin{equation}\label{esZ}
\left\{ \begin{aligned}
dZ_t^{\delta,h^\delta}=&\frac{1}{\lambda(\delta)}\Big[b\big(\lambda(\delta)Z_t^{\delta,h^\delta} +\bar{X}_t ,\mathcal{L}_{X^{\delta}_t},Y^{\delta,h^\delta}_t\big)-\bar{b}\big(\lambda(\delta)Z_t^{\delta,h^\delta} +\bar{X}_t ,\mathcal{L}_{X^{\delta}_t}\big)\Big]dt\\
&
+\frac{1}{\lambda(\delta)}\Big[\bar{b}\big(\lambda(\delta)Z_t^{\delta,h^\delta} +\bar{X}_t ,\mathcal{L}_{X^{\delta}_t}\big)-\bar{b}(\bar{X}_t,\mathcal{L}_{X^{\delta}_t})\Big]dt\\
&+\frac{1}{\lambda(\delta)}\Big[\bar{b}(\bar{X}_t,\mathcal{L}_{X^{\delta}_t})-\bar{b}(\bar{X}_t,\mathcal{L}_{\bar{X}_t})\Big]dt\\
&+\frac{\sqrt{\delta}}{\lambda(\delta)}\sigma\big(\lambda(\delta)Z_t^{\delta,h^\delta} +\bar{X}_t ,\mathcal{L}_{X^{\delta}_t}\big)dW_t^1\\
&+\sigma\big(\lambda(\delta)Z_t^{\delta,h^\delta} +\bar{X}_t ,\mathcal{L}_{X^{\delta}_t}\big){h}^{1,\delta}_t dt,\\
Z_0^{\delta,h^\delta}=0,&
\end{aligned} \right.
\end{equation}
where the controls $h^{1,\delta}_t:=P_1h^{\delta}_t,h^{2,\delta}_t:=P_2h^{\delta}_t$,      the controlled processes $(X^{\delta,h^\delta},Y^{\delta,h^\delta})$ solve
\begin{equation}\label{e9}
\left\{ \begin{aligned}
dX^{\delta,h^\delta}_t=&b(X^{\delta,h^\delta}_t,\mathcal{L}_{X^{\delta}_t},Y^{\delta,h^\delta}_t)dt
+\lambda(\delta)\sigma(X^{\delta,h^\delta}_t,\mathcal{L}_{X^{\delta}_t})h^{1,\delta}_t dt\\
&+\sqrt{\delta}\sigma(X^{\delta,h^\delta}_t,\mathcal{L}_{X^{\delta}_t})dW_t^1,\\
dY^{\delta,h^\delta}_t=&\frac{1}{\varepsilon}f(X^{\delta,h^\delta}_t,\mathcal{L}_{X^{\delta}_t},Y^{\delta,h^\delta}_t)dt
+\frac{\lambda(\delta)}{\sqrt{\delta\varepsilon}}g(X^{\delta,h^\delta}_t,\mathcal{L}_{X^{\delta}_t},Y^{\delta,h^\delta}_t)h^{2,\delta}_t dt\\
&+\frac{1}{\sqrt{\varepsilon}}g(X^{\delta,h^\delta}_t,\mathcal{L}_{X^{\delta}_t},Y^{\delta,h^\delta}_t)dW_t^2,\\
X^{\delta,h^\delta}_0=&x, Y^{\delta,h^\delta}_0=y.
\end{aligned} \right.
\end{equation}

After setting $F(W)= \frac{\lambda^2(\delta)}{\delta}\Lambda(Z^{\delta})$   and rescaling
the controls by $\frac{\sqrt{\delta}}{\lambda(\delta)}$ we get the representation
\begin{equation}\label{j4}
-\frac{\delta}{\lambda^2(\delta)}\log\mathbb{E}\left[\exp\left\{-\frac{\lambda^2(\delta)}{\delta}\Lambda(Z^\delta)\right\}\right]=\inf_{h\in \mathcal{A}}\mathbb{E}\left[\frac{1}{2}\int_0^T|h_s|^2ds+\Lambda(Z^{\delta, h})\right],
\end{equation}
where  $Z^{\delta, h}$  is defined by  (\ref{esZ}) with $h$ replacing $h^{\delta}$. Therefore, it is crucial to  study the  limiting behaviors of the controlled deviation process  $Z^{\delta, h^\delta}$ in the weak convergence analysis.

We now recall the following  sufficient condition for the Laplace principle to hold for $Z^{\delta}$ (cf.~Theorem 4.6 in \cite{liuw}). For a specific proof we also refer the reader to \cite[Theorem 9.9]{bd} or \cite[Theorem 3.2]{msz}.

For any $\delta>0$, suppose that $\Upsilon^\delta: C([0,T]; \RR^{d_1+d_2})\rightarrow
\mathcal{E}$ is given by (\ref{e11}).

\begin{hypothesis}\label{h2}
 There exists a measurable map $\Upsilon^0: C([0,T];
\mathbb{R}^{d_1+d_2})\rightarrow \mathcal{E}$ for which the following two conditions hold.

(i) Let $\{h^\delta\}_{\delta>0}\subset S_M$ for some $M<\infty$ such that $h^\delta$ converges weakly to element $h$ in $S_M$ as $\delta\to0$, then
$\Upsilon^0\big(\int_0^{\cdot}{h}^\delta_s ds\big)$ converges to $\Upsilon^0\big(\int_0^{\cdot}{h}_s ds\big)$ in $\mathcal{E}$.

(ii) Let $\{h^\delta\}_{\delta>0}\subset \mathcal{A}_M$ for
some $M<\infty$. For any $\varepsilon_0>0$, we have
$$\lim_{\delta\to 0}\mathbb{P}\Big(d\Big(\Upsilon^\delta\Big(\sqrt{\delta}W_{\cdot}
+\lambda(\delta)\int_0^{\cdot}h^\delta_sd s\Big),\Upsilon^0\big(\int_0^{\cdot}{h}^\delta_s ds\big)\Big)>\varepsilon_0\Big)=0, $$
where $d(\cdot,\cdot)$ denotes the metric in $\mathcal{E}$.
\end{hypothesis}

Then we have the following result.
\begin{lemma}\label{app1}  Suppose that Hypothesis \ref{h2}
holds, then $\{Z^\delta\}_{\delta>0}$ satisfies the LDP in $\mathcal{E}$ with speed $\delta/\lambda(\delta)^2$ and  good
rate function $I$ given by
\begin{equation}\label{rf}
I(f):=\inf_{\left\{ h \in L^2([0,T];\RR^{d_1+ d_2}):\  f=\Upsilon^0(\int_0^\cdot
{h}_t dt)\right\}}\left\lbrace\frac{1}{2}
\int_0^T|{h}_t|^2 d t \right\rbrace,
\end{equation}
where infimum over an empty set is taken as $+\infty$.
\end{lemma}

\subsection{Main results}
\vspace{0.1cm}

We first introduce the following Poisson equation depending on parameters $(x,\mu)\in \RR^n\times \mathcal{P}_2$,
\begin{equation}
-\mathbf{L}^{2}_{x,\mu}\Phi(x,\mu,y)=b(x,\mu,y)-\bar{b}(x,\mu),\label{PE}
\end{equation}
where $\Phi(x,\mu,y):=(\Phi_1(x,\mu,y),\ldots, \Phi_n(x,\mu,y))$,
$$\mathbf{L}^{2}_{x,\mu}\Phi(x,\mu,y):=(\mathfrak{L}^{2}_{x,\mu}\Phi_1(x,\mu,y),\ldots, \mathfrak{L}^{2}_{x,\mu}\Phi_n(x,\mu,y))$$
and for any $k=1,\ldots,n,$
\begin{eqnarray}\label{inf1}
\mathfrak{L}^{2}_{x,\mu}\Phi_k(x,\mu,y):=\langle f(x,\mu,y), \partial_y \Phi_k(x,\mu,y)\rangle+\frac{1}{2}\text{Tr}[g g^{*}(x,\mu,y)\partial^2_{yy} \Phi_k(x,\mu,y)].
\end{eqnarray}
According to \cite[Proposition 4.1]{RSX1}, under the conditions $(\mathbf{A1})$ and $(\mathbf{A2})$, (\ref{PE}) admits a solution $\Phi(x,\mu,y)$ satisfying $\Phi(\cdot,\mu,\cdot)\in C^{2,2}(\RR^n\times\RR^m;\RR^n)$ and $\Phi(x,\cdot,y)\in C^{(1,1)}(\mathcal{P}_2; \RR^n)$.

% and
%\begin{equation}\label{14}
%\overline{(\partial_y\Phi_g)(\partial_y\Phi_g)^*}(x,\mu):=\int_{\RR^m}(\partial_y\Phi_g(x,\mu,y))(\partial_y\Phi_g(x,\mu,y))^*\nu^{x,\mu}(dy).
%\end{equation}

In view of  the theory of the averaging principle (cf.~\cite{RSX1}), let $\delta\to 0$ (hence $\varepsilon\to0$) in (\ref{E2}), we can get the following differential equation
\begin{equation}\label{e7}
\frac{d\bar{X}_t}{dt}=\bar{b}(\bar{X}_t,\mathcal{L}_{\bar{X}_t}),~\bar{X}_0=x_0\in\RR^n,
\end{equation}
where the coefficient $\bar{b}$ is defined by
\begin{equation}\label{nu1}
\bar{b}(x,\mu)=\int_{\RR^m}b(x,\mu,y)\nu^{x,\mu}(dy),
\end{equation}
and $\nu^{x,\mu}$  denotes the unique invariant measure associated with the operator $\mathbf{L}^{2}_{x,\mu}$ {(as stated in Remark \ref{re1} (i), $(\mathbf{A1})$ guarantees the existence and uniqueness of invariant measures to the frozen equation (\ref{FEQ2}), see \cite[Theorem 4.3.9]{LR1}).
It is clear that (\ref{e7}) admits a unique solution (cf.~Lemma \ref{C1} below),} always denoted by $\bar{X}$ throughout this work, which satisfies $\bar{X}\in C([0,T];\RR^n)$. In addition, we mention that the solution $\bar{X}$ of (\ref{e7}) is a deterministic path and its distribution $\mathcal{L}_{\bar{X}_t}=\delta_{\bar{X}_t}$, where $\delta_{\bar{X}_t}$ is the Dirac measure of $\bar{X}_t$.

In this paper, we intend to show that the deviation process
\begin{eqnarray*}
Z_t^{\delta}:=\frac{X^{\delta}_t-\bar{X}_t}{\lambda(\delta)}
=\!\!\!\!\!\!\!\!&&~~\int_0^t\frac{1}{\lambda(\delta)}\Big[b(\lambda(\delta)Z_s^{\delta} +\bar{X}_s ,\mathcal{L}_{X^{\delta}_s},Y^{\delta}_s)-\bar{b}(\bar{X}_s,\mathcal{L}_{\bar{X}_s})\Big]ds
\\
\!\!\!\!\!\!\!\!&&+\int_0^t\frac{\sqrt{\delta}}{\lambda(\delta)}\sigma(\lambda(\delta)Z_s^{\delta} +\bar{X}_s ,\mathcal{L}_{X^{\delta}_s})dW_s^1
\end{eqnarray*}
satisfies the LDP (equivalently, $X^{\delta}$ satisfies the MDP).

\vspace{2mm}
Now we present our main results of this work considering the MDP for \eref{E2} in different regimes.

\vspace{2mm}
\textbf{Regime 1: $\lim_{\delta\rightarrow0}\varepsilon/\delta=0$.} Let us define a map $\Upsilon^0: C([0,T]; \RR^{d_1+d_2})\rightarrow C([0,T]; \RR^n)$  by
\begin{equation*}\label{g1}
\Upsilon^0(\phi):=\left\{ \begin{aligned}
&Z^{h},~~\text{if}~\phi=\int_0^{\cdot}h_sds~~\text{for some}~h\in L^2([0,T];\RR^{d_1+ d_2}),\\
&0,~~~~\text{otherwise}.
\end{aligned} \right.
\end{equation*}
Here $Z^{h}$ is the solution of the skeleton equation
\begin{eqnarray}\label{ske1}
dZ_t^h= \!\!\!\!\!\!\!\!&&~~~ \partial_x\bar{b}(\bar{X}_t,\mathcal{L}_{\bar{X}_t})\cdot Z_t^hdt+\sigma(\bar{X}_t,\mathcal{L}_{\bar{X}_t}){h}^1_tdt,~Z_0^h=0,
\end{eqnarray}
where $h^1_t:=P_1 h_t$.

\vspace{2mm}
The first result of MDP is the following.
\begin{theorem}\label{t3}
Suppose that the conditions $(\mathbf{A1})$-$(\mathbf{A3})$ hold.
Then for any initial values $x\in\RR^n,y\in\RR^m$, $\{Z^{\delta}\}_{\delta>0}$
satisfies the LDP in $C([0,T]; \RR^n)$ in Regime 1 of (\ref{regime}) with speed $\delta/\lambda(\delta)^2$ and
good rate function $I$ given by
\begin{equation}\label{rf1}
I(\varphi):=\inf_{\left\{h\in L^2([0,T];\RR^{d_1+ d_2}):\  \varphi=Z^h\right\}}\left\lbrace\frac{1}{2}
\int_0^T|{h}^1_t|^2 d t \right\rbrace,
\end{equation}
with the convention that the infimum over the empty set is $\infty$.

Furthermore, if the condition $(\mathbf{A4})$ holds, let
%$$Q_i(x,\mu)=\int_{\mathcal{Y}}\sigma\sigma^T(x,\mu)\nu^{x,\mu}(dy),$$
$$Q_1(x,\mu):=\sigma(x,\mu)\sigma^*(x,\mu),$$
then we obtain an explicit representation of the rate function
{\begin{equation}\label{thjj2}
I(\varphi)=
   \left\{
  \begin{array}{ll}
\displaystyle \frac{1}{2}\int_0^T|Q_1^{-1/2}(\bar{X}_s,\mathcal{L}_{\bar{X}_s})(\dot{\varphi}_s-\partial_x \bar{b}(\bar{X}_s,\mathcal{L}_{\bar{X}_s})\varphi_s)|^2ds,~ &\mbox{$\varphi(0)=0$, $\varphi$ is absolutely continuous},\\
\displaystyle +\infty, \quad &\mbox{otherwise}.
\end{array}
\right.
\end{equation}}
\end{theorem}
\begin{remark}
(i) {In order to prove Theorem \ref{t3}, it suffices to check the conditions (i) and (ii) in Hypothesis \ref{h2} hold, which greatly simplified the whole proof instead of using the method of functional occupation measures in Regime 2.  It is worth noting that, it is unnecessary to impose the uniform ellipticity of $\sigma\sigma^*$, whereas in Regime 2, the uniform ellipticity is strongly required in order to transform the control to the  feedback form. To the best of our knowledge, the aforementioned points seem also new in the study of MDP for the standard multi-scale system (see e.g.~\cite{gss}).}

(ii) If the uniform ellipticity condition $(\mathbf{A4})$ holds, the representation of the rate function {(\ref{thjj2})} easily follows from the definition of the skeleton equation (\ref{ske1}).

\end{remark}

{Below, we will provide a remark to explain why the weak convergence criterion is not applicable to the Regime 2.
\begin{remark}\label{rem+1}
The main difficulty of using the weak convergence method lies in characterizing the limit in distribution of the controlled process
$Z_t^{\delta,h^\delta}$ in \eref{esZ}, due to the appearance of the term
\begin{equation}\label{e1}
\frac{\sqrt{\varepsilon}}{\sqrt{\delta}}\int^t_0\partial_y \Phi(X_s^{\delta,h^\delta} ,\mathcal{L}_{X^{\delta}_s},Y^{\delta,h^\delta}_s)\cdot g(X^{\delta,h^\delta}_s,\mathcal{L}_{X^{\delta}_s},Y^{\delta,h^\delta}_s)h^{2,\delta}_s ds.
\end{equation}

 In the case of Regime 1, it is clear that (\ref{e1}) tends to $0$, as $\delta\to 0$, as long as the terms $\partial_y \Phi$ and $g$ are bounded and $h^{2,\delta}$ is $L^2$-bounded. Nevertheless, it is quite involved to prove the convergence of (\ref{e1}) in Regime 2 due to the joint perturbations of  controlled fast process $Y^{\delta,h^\delta}$ and  control function $h^{2,\delta}$.

\end{remark}
}

\vspace{2mm}
\textbf{Regime 2: $\lim_{\delta\rightarrow0}\varepsilon/\delta=\gamma$.}
Before introducing our second result of MDP, we need some additional notations and definitions. We use $\mathcal{Y}:=\mathbb{R}^m$ to emphasize the state space of the fast component, and  $\mathcal{Z}_1:=\mathbb{R}^{d_1}, \mathcal{Z}_2:=\mathbb{R}^{d_2}$ to emphasize the spaces in which the controls take values.   Let $A_1,A_2,A_3,A_4$ be Borel subsets of $\mathcal{Z}_1,\mathcal{Z}_2,\mathcal{Y}, [0,T]$ respectively. Let  $\Delta:=\Delta(\delta)$ be a time-scale separation such that
\begin{equation}\label{delta}
  \Delta(\delta) \rightarrow0,~ \frac{\lambda^2(\delta)}{\Delta} \rightarrow0,~\text{as}~ \delta \rightarrow0.
\end{equation}

Due to the involved controls in multi-scale systems, it is convenient to introduce the following occupation measure:
\begin{equation}\label{occupation}
\mathbf{P}^{\delta,\Delta}(A_1\times A_2 \times A_3\times A_4):=\int_{A_4}\frac{1}{\Delta}\int_t^{t+\Delta}\mathbf{1}_{A_1}(h^{1,\delta}_s)\mathbf{1}_{A_2}(h^{2,\delta}_s){\mathbf{1}_{A_3}(Y^{\delta, h^\delta}_s)}dsdt,
\end{equation}
where the controlled process $(X^{\delta, h^\delta},Y^{\delta, h^\delta})$ is defined by (\ref{e9}).
Throughout this work, we adopt the convention  that the control $h_t=h^{\delta}_t=0$ if $t>T$.
We also remark that
 for any bounded continuous functions $\xi$, we have
\begin{align}\label{p1}
 &\int_{\mathcal{Z}_1\times\mathcal{Z}_2\times\mathcal{Y}\times[0,T]}\xi(X^{\delta, h^\delta}_t,\mathcal{L}_{X^{\delta}_t},y,h_1,h_2) \mathbf{P}^{\delta,\Delta}(dh_1dh_2dydt)\nonumber\\
 &=\int_0^T\frac{1}{\Delta}\int_t^{t+\Delta}\xi(X^{\delta, h^\delta}_t,\mathcal{L}_{X^{\delta}_t},Y^{\delta, h^\delta}_s,h^{1,\delta}_s,h^{2,\delta}_s)dsdt.
\end{align}

%Let
%\begin{equation*}
%\gamma_i=\left\{
%  \begin{array}{ll}
%    0, & \hbox{if i=1;} \\
%    \gamma, & \hbox{if i=2,}
%  \end{array}
%\right.
%\end{equation*}
Define a map $\Theta:\mathbb{R}^n\times \mathbb{R}^n\times \mathcal{P}_2\times\mathcal{Y}\times\mathcal{Z}_1\times\mathcal{Z}_2\rightarrow\mathbb{R}^n$  by
\begin{equation}\label{theta}
 \Theta(\varphi,x,\mu,y,h^1,h^2):=\partial_x \bar{b}(x,\mu)\varphi+\sigma(x,\mu)h^1+\sqrt{\gamma}\partial_y\Phi_g(x,\mu,y)h^2,
\end{equation}
where $\Phi$ is a solution of Poisson equation (\ref{PE}), and we denote
\begin{eqnarray}
\partial_y\Phi_g(x,\mu,y)u:=\partial_y\Phi(x,\mu,y)\cdot g(x,\mu,y)u,\quad u\in\RR^{d_2}. \label{partial Phig}
\end{eqnarray}

We now introduce the notion of {\it viable pair}, which is from \cite{gss,hu2} but with slight modification to the McKean-Vlasov structure.
\begin{definition}\label{defj1}(Viable pair)
Let $\Theta:\mathbb{R}^n\times \mathbb{R}^n\times \mathcal{P}_2\times\mathcal{Y}\times\mathcal{Z}_1\times\mathcal{Z}_2\rightarrow\mathbb{R}^n$ and $\bar{X}$ solve (\ref{e7}).   A pair $(\varphi,\mathbf{P})\in C([0,T];\mathbb{R}^n)\times \mathcal{P}(\mathcal{Z}_1\times\mathcal{Z}_2\times\mathcal{Y}\times[0,T])$  is called viable w.r.t.~$(\Theta,\nu^{\bar{X},\mathcal{L}_{\bar{X}}})$, where $\nu$ is  defined in (\ref{nu1}),  and write $(\varphi,\mathbf{P})\in\mathcal{V}_{(\Theta,\nu^{\bar{X},\mathcal{L}_{\bar{X}}})}$, if the following statements hold:

\vspace{1mm}

\vspace{1mm}
{\rm(i)} The measure $\mathbf{P}$ has finite second moments, i.e.
$$\int_{\mathcal{Z}_1\times\mathcal{Z}_2\times\mathcal{Y}\times [0,T]}\Big[|h^1|^2+|h^2|^2+|y|^2\Big]\mathbf{P}(dh^1dh^2dyds)<\infty.$$

\vspace{1mm}
{\rm(ii)} For all $t\in[0,T]$,
\begin{equation}\label{j1}
 \varphi_t=\int_{\mathcal{Z}_1\times\mathcal{Z}_2\times\mathcal{Y}\times [0,t]}\Theta(\varphi_s,\bar{X}_s,\mathcal{L}_{\bar{X}_s},y,h^1,h^2)\mathbf{P}(dh^1dh^2dyds).
\end{equation}

\vspace{1mm}
{\rm(iii)} For all $A_1\times A_2\times A_3\times A_4\in \mathcal{B}(\mathcal{Z}_1\times\mathcal{Z}_2\times\mathcal{Y}\times[0,T])$,
\begin{equation}\label{j2}
\mathbf{P}(A_1\times A_2\times A_3\times A_4)=\int_{A_4}\int_{A_3}\eta(A_1\times A_2|y,t)\nu^{\bar{X}_t,\mathcal{L}_{\bar{X}_t}}(dy)dt,
\end{equation}
where $\eta$ is  a stochastic kernel (cf.~\cite[Appendix B.2]{bd}) given $\mathcal{Y}\times[0,T]$. This  implies that the
last marginal of $\mathbf{P}$ is Lebesgue measure on $[0,T]$ and in particular
 for all $t\in[0,T]$,
\begin{equation}\label{j3}
\mathbf{P}(\mathcal{Z}_1\times\mathcal{Z}_2\times\mathcal{Y}\times[0,t])=t.
\end{equation}
%This implies that the last marginal of $\mathbf{P}$ is Lebesgue measure on $[0,T]$ and in particular
%\begin{equation}\label{j3}
%\mathbf{P}(\mathcal{Z}\times\mathcal{Z}\times\mathcal{Y}\times [0,t])=t.
%\end{equation}
\end{definition}

%\begin{remark}
%By (\ref{j3}), it follows that the last marginal of $\mathbf{P}$ is Lebesgue measure. Therefore, $\mathbf{P}$ can be decomposed in the form $\mathbf{P}(dh^1dh^2dydt)=\mathbf{P}_t(dh^1dh^2dy)dt$.
%\end{remark}

%{
%In what follows, we shall consider controls $h=h^\delta$  that may depend on $\delta$,  but such that there exists $M>0$ such that for all $\delta\in(0,1)$, we have $h=h^\delta\in\mathcal{A}_M$. }

Now we state the MDP result in the second regime.
\begin{theorem}\label{thj1}
Suppose that the conditions $(\mathbf{A1})$-$(\mathbf{A4})$ hold. Let  $\Theta$ be defined by (\ref{theta}), $\bar{X}$ be the solution to (\ref{e7}).
Then
%for every bounded and continuous function $a:C([0,T];\mathbb{R}^n )\to\mathbb{R}$, we have
%\begin{equation*}
%\liminf_{\delta\rightarrow0}-\frac{\delta}{\lambda^2(\delta)}\log\mathbb{E}\left[\exp\left\{-\frac{\lambda^2(\delta)}{\delta}a(Z^\delta)\right\}\right]=\inf_{\varphi\in C([0,T];\mathbb{R}^d )}\Big[S(\varphi)+a(\varphi)\Big].
%\end{equation*}
%Moreover, for each $s<\infty$, the set
%$$\Phi(s)=\Big\{\varphi\in C([0,T];\mathbb{R}^n ):I(\varphi)\leq s\Big\}~\text{is compact}.$$
%In other words,
$\{Z^\delta\}_{\delta>0}$ satisfies the LDP  with speed $\delta/\lambda(\delta)^2$ and rate function $I$ given by
\begin{equation}\label{thjj1}
I(\varphi):=\inf_{(\varphi,\mathbf{P})\in\mathcal{V}_{(\Theta,\nu^{\bar{X},\mathcal{L}_{\bar{X}}})}}\Bigg\{\frac{1}{2}\int_{\mathcal{Z}_1\times\mathcal{Z}_2\times\mathcal{Y}\times [0,T]}\Big[|h^1|^2+|h^2|^2\Big]\mathbf{P}(dh^1dh^2dydt)\Bigg\},
\end{equation}
with the convention that the infimum over the empty set is $\infty$.

As a consequence, the rate function $I$ has the following explicit representation
 {\begin{equation}\label{thjjj2}
  I(\varphi)=
   \left\{
  \begin{array}{ll}
\displaystyle \frac{1}{2}\int_0^T|Q_2^{-1/2}(\bar{X}_s,\mathcal{L}_{\bar{X}_s})(\dot{\varphi}_s-\partial_x \bar{b}(\bar{X}_s,\mathcal{L}_{\bar{X}_s})\varphi_s)|^2ds,~ &\mbox{$\varphi(0)=0$, $\varphi$ is absolutely continuous},\\
\displaystyle +\infty, \quad &\mbox{otherwise},
\end{array}
\right.
\end{equation}}
where
\begin{equation}\label{eqQ}
Q_2(x,\mu):=\int_{\mathcal{Y}}\sigma\sigma^*(x,\mu)+\gamma(\partial_y\Phi_g)(\partial_y\Phi_g)^*(x,\mu,y)\nu^{x,\mu}(dy),
 \end{equation}
 {and $\partial_y\Phi_g$ is defined by (\ref{partial Phig}).}
\end{theorem}

\begin{remark}
We mention that the operator $Q_2$ defined by (\ref{eqQ}) is invertible by the condition $(\mathbf{A4})$ and the fact that $\gamma(\partial_y\Phi_g)(\partial_y\Phi_g)^*$ is positive semi-definite.
\end{remark}

We also give some comments for the main results above.

\begin{remark}
(i) The LDP of multi-scale McKean-Vlasov SDEs was first studied in \cite{HLLS}, recently the authors in \cite{SZW23,WHY23} also established the LDP for multi-scale McKean-Vlasov SDEs with fractional noise, i.e. the slow component is driven by fractional Brownian motions and the fast one is driven by standard Brownian motions.
While, all the works \cite{HLLS,SZW23,WHY23} only consider the case of Regime 1, here we prove that the MDP holds both in Regime 1 and Regime 2, which is more general than those works.

(ii) The rate functions in two regimes are different.  We note that at the second regime, the rate function depends not only on the diffusion coefficient $\sigma$ in slow component, but also on the coefficient $g$ in fast component and the order $\gamma$ that rates the radio of $\delta$ and $\varepsilon$ converging to zero, which is essentially different from the Regime 1 and the case of large deviations (cf.~\cite{HLLS,SZW23,WHY23}). In particular, the dependence of the diffusion coefficient $g$ in the rate function reveal the influence induced by the noise of the fast equation in \eref{E2}.

(iii) For the examples of multi-scale McKean-Vlasov SDEs satisfying our assumptions, the reader can refer to \cite[Example 2.1]{HLLS}, where the authors have gave some detailed discussion on the applications of multi-scale system \eref{E2}.

\end{remark}

\section{Preliminaries}

In this section, we make some preparations before proving Theorems \ref{t3} and \ref{thj1}. In Subsection \ref{sec3.1}, we recall the regularities of the auxiliary Poisson equation on Wasserstein space, which play an important role in characterizing the convergence of the controlled processes. In Subsection \ref{sec3.2}, we give some a priori estimates for the solutions of the controlled processes.
\subsection{Poisson equation on Wasserstein space}\label{sec3.1}
Recall the frozen equation for any fixed $x\in\RR^n$ and $\mu\in\mathcal{P}_2$,
\begin{eqnarray}\label{FEQ2}
  \left\{
  \begin{aligned}
dY_{t}&=f(x, \mu, Y_{t})dt+g(x, \mu, Y_{t})d\tilde{W}_{t}^{2},\\
Y_{0}&=y\in\RR^m,
  \end{aligned}
\right.
\end{eqnarray}
where $\{\tilde{W}_{t}^{2}\}_{t\geq 0}$ is a $d_2$-dimensional Brownian motion on another complete filtered probability space $(\tilde{\Omega}, \tilde{\mathcal{F}}, \tilde{\mathcal{F}}_t, \tilde{\mathbb{P}})$. Under the conditions \eref{A21} and \eref{sm}, it is clear that
$(\ref{FEQ2})$ has a unique (strong) solution $\{Y_{t}^{x,\mu,y}\}_{t\geq 0}$ and admits a unique invariant probability measure $\nu^{x,\mu}$.

We recall the following regularities of solution to Poisson equation (\ref{PE}) w.r.t.~parameters $(x,\mu)$, which have been established in \cite{RSX1} (see also \cite{HLLS}).

\begin{proposition}\label{P3.6}
Suppose that $(\mathbf{A1})$ and $(\mathbf{A2})$ hold. Define
\begin{eqnarray}
\Phi(x,\mu,y)=\int^{\infty}_{0} \left[\tilde\EE b(x,\mu,Y^{x,\mu,y}_s)-\bar{b}(x,\mu)\right]ds,\label{SPE}
\end{eqnarray}
%where $Y^{x,\mu,y}$ is the solution of frozen equation (\ref{FEQ2}).
where $\tilde{\EE}$ is the expectation on $(\tilde{\Omega}, \tilde{\mathcal{F}}, \tilde{\mathbb{P}})$. Then $\Phi(x,\mu,y)$ is a solution of \eref{PE} and  $\Phi\in C^{2,(1,1),2}(\RR^n\times\mathcal{P}_2 \times\RR^m; \RR^n)$. Moreover, there exists $C>0$ such that
\begin{eqnarray}\label{PHI1}
&&\!\!\!\!\!\!\!\! \sup_{x\in\RR^n,\mu\in\mathcal{P}_2}\max\Big\{|\Phi(x,\mu,y)|,\|\partial_x \Phi(x,\mu,y)\|, \|\partial_{\mu}\Phi(x,\mu,y)\|_{L^2(\mu)},\|\partial^2_{xx} \Phi(x,\mu,y)\|\nonumber\\
&& \quad \|\partial_{z}\partial_{\mu}\Phi(x,\mu,y)(\cdot)\|_{L^2(\mu)}\Big \}\leq C(1+|y|)~
\end{eqnarray}
%\begin{eqnarray}\label{E2}
%&&\max\left \{\|\partial^2_{xx} \Phi(x,\mu,y)\| ,\|\partial_{z}\partial_{\mu}\Phi(x,\mu,y)(\cdot)\|_{L^2(\mu)}\right \}\nonumber\\
%\leq\!\!\!\!\!\!\!\!&& C\left \{1+|x|+|y|+[\mu(|\cdot|^2)]^{1/2}\right\}.
%\end{eqnarray}
and
\begin{eqnarray}\label{PHI3}
&&\!\!\!\!\!\!\!\!\sup_{x\in\RR^n,\mu\in\mathcal{P}_2, y\in \RR^m}\max\Big\{\|\partial_y \Phi(x,\mu,y)\|,\|\partial^2_{yy} \Phi(x,\mu,y)\|,\|\partial^2_{xy} \Phi(x,\mu,y)\|,  { \|\partial_{\mu}\partial_{y}\Phi(x,\mu,y)\|_{L^2(\mu)}}  \Big\}\leq C.
\end{eqnarray}
\end{proposition}

As preparations, we also collect some properties of $\partial_x\bar{b},\partial_{\mu}\bar{b}$, which have been established in Proposition 3.2 and Lemma 5.1 in \cite{HLLS}.

\begin{lemma} \label{C1} Suppose that $(\mathbf{A1})$ and $(\mathbf{A2})$ hold. Then
\begin{eqnarray*}
\sup_{x\in\mathbb{R}^n, \mu\in\mathcal{P}_2}\max\{\|\partial_x\bar{b}(x,\mu)\|,\|\partial_{\mu}\bar{b}(x,\mu)\|_{L^2(\mu)}\}<\infty,
\end{eqnarray*}
and for any fixed $x\in\mathbb{R}^n$, if sequence $\{(\mu_n,z_n)\}_{n\geq1}\subset \mathcal{P}_2\times \mathbb{R}^n$ satisfy $\mu_n\to\mu$ in $\mathbb{W}_2$ and $|z_n- z|\to 0$ as $n\to \infty$,
\begin{eqnarray*}
&&\lim_{n\to\infty}\big\|\partial_\mu \bar{b}(x,\mu_n)(z_n)-\partial_\mu \bar{b}(x,\mu)(z)\big\|=0,\\
&&\lim_{n\to\infty}\big\|\partial_x \bar{b}(z_n,\mu_n)-\partial_x \bar{b}(z,\mu)\big\|=0.
\end{eqnarray*}
\end{lemma}

\begin{remark}
We mention that by the boundedness of $\partial_{\mu}\bar{b}$ in $L^2(\mu)$, we can directly derive the Lipschitz continuity of $\bar{b}$ in $\mu$, which follows from the definition of Lions derivative and the Taylor argument for the Fr\'echet derivative.
In fact, for some $u: \mathcal{P}_2\rightarrow \mathbb{R}$ we denote by $U$ its ``extension" to $L^2(\Omega, \mathbb{P};\mathbb{R}^n)$ such that
$$
U(X)=u(\mathcal{L}_{X}),\quad X\in L^2(\Omega,\mathbb{P};\mathbb{R}^n).
$$
Then by (\ref{esfre}) and  using the Taylor argument for the function $U$ in the sense of Fr\'echet derivative, we have for any $X,Y\in L^2(\Omega,\mathbb{P};\mathbb{R}^n)$,
\begin{eqnarray*}
&&u(\mathcal{L}_{Y})-u(\mathcal{L}_{X})=U(Y)-U(X)=\int^1_0 DU(X+r(Y-X))\cdot (Y-X) dr\\
=\!\!\!\!\!\!\!\!&&~~~~~\int^1_0 \partial_{\mu}u(\mathcal{L}_{X+r(Y-X)})(X+r(Y-X))\cdot (Y-X) dr.
\end{eqnarray*}
Thus we have
\begin{eqnarray*}
|u(\mathcal{L}_{X})-u(\mathcal{L}_{Y})|\leq\!\!\!\!\!\!\!\!&&~~~~~ \sup_{\mu\in\mathcal{P}_2}\|\partial_{\mu}u(\mu)\|_{L^2(\mu)}\mathbb{W}_2(\mathcal{L}_{X},\mathcal{L}_{Y})
\nonumber\\
\leq\!\!\!\!\!\!\!\!&&~~~~~ C\mathbb{W}_2(\mathcal{L}_{X},\mathcal{L}_{Y}).
\end{eqnarray*}
\end{remark}

\subsection{A priori estimates}\label{sec3.2}
In this subsection, we present some necessary a priori estimates for the controlled processes, which will be used frequently in the proof of the main theorems.

Recall the following averaging principle result of stochastic system \eref{E2},
\begin{eqnarray}\label{33}
\EE\Big[\sup_{t\in[0,T]}|X^{\delta}_t-\bar{X}_t|^2\Big]\leq C_T\big(1+|x|^4+|y|^4\big)(\varepsilon+\delta),
\end{eqnarray}
whose proof could refer to \cite{HLLS}.
Hence, there exists a constant $C_T>0$ independent of $\varepsilon$ such that
\begin{eqnarray}
\EE\Big[\sup_{t\in [0,T]}|Z_t^{\delta}|^2\Big]\leq C_T\big(1+|x|^4+|y|^4\big)\frac{\varepsilon+\delta}{\lambda(\delta)^2}.\label{supZvare}
\end{eqnarray}

The estimates of solutions $(X^{\delta,h^\delta},Y^{\delta,h^\delta})$ to the control problem (\ref{e9}) are given as follows.
\begin{lemma}
Under $(\mathbf{A1})$ and $(\mathbf{A3})$, for any $\{h^{\delta}\}_{\delta>0}\subset\mathcal{A}_M$, there exists a constant $C_{T,M}>0$ and {is independent of $\delta$, }such that
\begin{equation}\label{X23}
\mathbb{E}\Big[\sup_{t\in[0,T]}|X^{\delta,h^\delta}_t|^2\Big]\leq C_{T,M}(1+|x|^{2}+|y|^2),
\end{equation}
and for any $p\geq 1$, there exists $C_{p,T,M}>0$
\begin{equation}\label{Y24}
\mathbb{E}\Big[\Big(\int_0^T|Y^{\delta,h^\delta}_t|^2dt\Big)^p\Big]\leq C_{p,T,M}(1+|y|^{2p}).
\end{equation}
\end{lemma}
\begin{proof}
Using It\^{o}'s formula for $|Y^{\delta,h^\delta}_t|^2$, we have
\begin{eqnarray}\label{eqy1}
|Y^{\delta,h^\delta}_t|^2
=\!\!\!\!\!\!\!\!&&~~~~{|y|^2}+\frac{1}{\varepsilon}\int_0^t\Big[2\langle f(X^{\delta,h^\delta}_s,\mathcal{L}_{X^{\delta}_s},Y^{\delta,h^\delta}_s),Y^{\delta,h^\delta}_s\rangle+\|g(X^{\delta,h^\delta}_s,\mathcal{L}_{X^{\delta}_s},Y^{\delta,h^\delta}_s)\|^2\Big]ds
\nonumber\\
\!\!\!\!\!\!\!\!&&~~~+\frac{2\lambda(\delta)}{\sqrt{\delta\varepsilon}}\int_0^t\langle g(X^{\delta,h^\delta}_s,\mathcal{L}_{X^{\delta}_s},Y^{\delta,h^\delta}_s)h^{2,\delta}_s,Y^{\delta,h^\delta}_s\rangle ds+M_t,
\end{eqnarray}
where
$$M_t:=\frac{2}{\sqrt{\varepsilon}}\int_0^t\langle Y^{\delta,h^\delta}_s, g(X^{\delta,h^\delta}_s,\mathcal{L}_{X^{\delta}_s},Y^{\delta,h^\delta}_s)dW_s^2\rangle.$$

By $(\mathbf{A3})$, we also have
\begin{eqnarray}\label{eqy2}
\frac{2\lambda(\delta)}{\sqrt{\delta\varepsilon}}\langle g(X^{\delta,h^\delta}_s,\mathcal{L}_{X^{\delta}_s},Y^{\delta,h^\delta}_s)h^{2,\delta}_s,Y^{\delta,h^\delta}_s\rangle
~~~\!\!\!\!\!\!\!\!&&\leq\frac{2\lambda(\delta)}{\sqrt{\delta\varepsilon}}|h^{2,\delta}_s||Y^{\delta,h^\delta}_s|
\nonumber\\
~~~\!\!\!\!\!\!\!\!&&\leq\frac{C\lambda(\delta)^2}{\delta}|h^{2,\delta}_s|^2+\frac{\tilde{\beta}}{\varepsilon}|Y^{\delta,h^\delta}_s|^2,
\end{eqnarray}
where we used Young's inequality in the last step for a small constant $\tilde{\beta}\in(0,\beta)$, $\beta$ is defined in (\ref{RE3}).

Combining (\ref{RE3}) and (\ref{eqy1})-(\ref{eqy2}), it follows that
\begin{eqnarray*}
\frac{\kappa_0}{\varepsilon}\int_0^T|Y^{\delta,h^\delta}_s|^2ds
\leq\!\!\!\!\!\!\!\!&&~~~|y|^2+\frac{C_{T}}{\varepsilon}+\frac{C\lambda(\delta)^2}{\delta}\int_0^T|h^{2,\delta}_s|^2ds+ \sup_{t\in[0,T]}|M_t|,
\end{eqnarray*}
where $\kappa_0:=\beta-\tilde{\beta}>0$. Then
\begin{eqnarray*}
\mathbb{E}\Big[\Big(\int_0^T|Y^{\delta,h^\delta}_s|^2ds\Big)^p\Big]
\leq\!\!\!\!\!\!\!\!&&~~~C_p\varepsilon^p|y|^{2p}+C_{p,T}+\frac{C_p\lambda(\delta)^{2p}\varepsilon^p}{\delta^p}\mathbb{E}\Big[\Big(\int_0^T|h^{2,\delta}_s|^2ds\Big)^p\Big]+C_p\varepsilon^p\mathbb{E}\Big[\Big(\sup_{t\in[0,T]}|M_t|\Big)^p\Big],
\nonumber\\
\leq\!\!\!\!\!\!\!\!&&~~~\frac{1}{2}\mathbb{E}\Big[\Big(\int_0^T|Y^{\delta,h^\delta}_s|^2ds\Big)^p\Big]+C_{p,T}(1+|y|^{2p})+\frac{C_{p,T,M}\lambda(\delta)^{2p}\varepsilon^p}{\delta^p},
\end{eqnarray*}
where we used the fact that $h^{\delta}\in\mathcal{A}_M$ and the following estimate in the second step
\begin{eqnarray*}
C_p\varepsilon^p\mathbb{E}\Big[\Big(\sup_{t\in[0,T]}|M_t|\Big)^p\Big]\leq\!\!\!\!\!\!\!\!&&~~~ C_p\varepsilon^{\frac{p}{2}}\mathbb{E}\Big[\Big(\int_0^T|Y^{\delta,h^\delta}_s|^2ds\Big)^{\frac{p}{2}}\Big]
\nonumber\\
\leq\!\!\!\!\!\!\!\!&&~~~ \frac{1}{2}\mathbb{E}\Big[\Big(\int_0^T|Y^{\delta,h^\delta}_s|^2ds\Big)^{p}\Big]+C_p.
\end{eqnarray*}
Note that $\lambda(\delta)\to 0$ as $\delta\to 0$, in view of Regimes 1 and 2, without loss of generality we can assume $(\frac{\varepsilon}{\delta})\lambda(\delta)^2< 1$, thus (\ref{Y24}) holds.

Applying It\^{o}'s formula to $|X^{\delta,h^\delta}_t|^2$, we have
\begin{eqnarray*}
\!\!\!\!\!\!\!\!&&|X^{\delta,h^\delta}_t|^2
\nonumber\\
=\!\!\!\!\!\!\!\!&&~~~|x|^2+2\int_0^t\langle b(X^{\delta,h^\delta}_s,\mathcal{L}_{X^{\delta}_s},Y^{\delta,h^\delta}_s),X^{\delta,h^\delta}_s\rangle ds
+2\lambda(\delta)\int_0^t\langle \sigma(X^{\delta,h^\delta}_s,\mathcal{L}_{X^{\delta}_s})h^{1,\delta}_s,X^{\delta,h^\delta}_s\rangle ds
\nonumber\\
\!\!\!\!\!\!\!\!&&~~~
+\delta\int_0^t\|\sigma(X^{\delta,h^\delta}_s,\mathcal{L}_{X^{\delta}_s})\|^2ds
+2\sqrt{\delta}\int_0^t\langle \sigma(X^{\delta,h^\delta}_s,\mathcal{L}_{X^{\delta}_s})dW_s^1,X^{\delta,h^\delta}_s\rangle.
\end{eqnarray*}
Due to the  condition $(\mathbf{A1})$,
\begin{eqnarray}\label{27}
\mathbb{E}\Big[\sup_{t\in[0,T]}|X^{\delta,h^\delta}_t|^2\Big]\leq \!\!\!\!\!\!\!\!&&~~~ |x|^2+C\mathbb{E}\int_0^T\Big(1+|X^{\delta,h^\delta}_t|^2+\mathcal{L}_{X^{\delta}_t}(|\cdot|^2)\Big)dt+C\mathbb{E}\int_0^T|Y^{\delta,h^\delta}_t|^2dt
\nonumber\\
\!\!\!\!\!\!\!\!&&~~~+2\lambda(\delta)\mathbb{E}\int_0^T\Big|\langle \sigma(X^{\delta,h^\delta}_t,\mathcal{L}_{X^{\delta}_t})h^{1,\delta}_t,X^{\delta,h^\delta}_t\rangle\Big|dt
\nonumber\\
\!\!\!\!\!\!\!\!&&~~~
+2\sqrt{\delta}\mathbb{E}\Bigg[\sup_{t\in[0,T]}\Big|\int_0^t\langle \sigma(X^{\delta,h^\delta}_s,\mathcal{L}_{X^{\delta}_s})dW_s^1,X^{\delta,h^\delta}_s\rangle\Big|\Bigg].
\end{eqnarray}
The fourth term on the right hand side of (\ref{27}) can be estimated as follows,
\begin{eqnarray}\label{28}
\!\!\!\!\!\!\!\!&&2\lambda(\delta)\mathbb{E}\int_0^T\Big|\langle \sigma(X^{\delta,h^\delta}_t,\mathcal{L}_{X^{\delta}_t})h^{1,\delta}_t,X^{\delta,h^\delta}_t\rangle\Big|dt
\nonumber\\
~\leq\!\!\!\!\!\!\!\!&&~~~\frac{1}{4}\mathbb{E}\Big[\sup_{t\in[0,T]}|X^{\delta,h^\delta}_t|^2\Big]+C\lambda(\delta)^2\mathbb{E}\Big(\int_0^T\|\sigma(X^{\delta,h^\delta}_t,\mathcal{L}_{X^{\delta}_t})\|\cdot|h^{1,\delta}_t|dt\Big)^2
\nonumber\\
~\leq\!\!\!\!\!\!\!\!&&~~~\frac{1}{4}\mathbb{E}\Big[\sup_{t\in[0,T]}|X^{\delta,h^\delta}_t|^2\Big]+C\lambda(\delta)^2\mathbb{E}\Big[\Big(\int_0^T\|\sigma(X^{\delta,h^\delta}_t,\mathcal{L}_{X^{\delta}_t})\|^2dt\Big)\Big(\int_0^T|h^{1,\delta}_t|^2dt\Big)\Big]
\nonumber\\
~\leq\!\!\!\!\!\!\!\!&&~~~\frac{1}{4}\mathbb{E}\Big[\sup_{t\in[0,T]}|X^{\delta,h^\delta}_t|^2\Big]+C_{M,T}+C_M\mathbb{E}\int_0^T|X^{\delta,h^\delta}_t|^2dt+C_M\mathbb{E}\int_0^T|X^{\delta}_t|^2dt.
\end{eqnarray}
Moreover, by Burkholder-Davis-Gundy's inequality, we obtain
\begin{eqnarray}\label{29}
\!\!\!\!\!\!\!\!&&2\sqrt{\delta}\mathbb{E}\Bigg[\sup_{t\in[0,T]}\Big|\int_0^t\langle \sigma(X^{\delta,h^\delta}_s,\mathcal{L}_{X^{\delta}_s})dW_s^1,X^{\delta,h^\delta}_s\rangle\Big|\Bigg]
\nonumber\\
~\leq\!\!\!\!\!\!\!\!&&~~~8\sqrt{\delta}\mathbb{E}\Bigg[\int_0^T\|\sigma(X^{\delta,h^\delta}_t,\mathcal{L}_{X^{\delta}_t})\|^2|X^{\delta,h^\delta}_t|^2dt\Bigg]^{\frac{1}{2}}
\nonumber\\
~\leq\!\!\!\!\!\!\!\!&&~~~\frac{1}{4}\mathbb{E}\Big[\sup_{t\in[0,T]}|X^{\delta,h^\delta}_t|^2\Big]+C\mathbb{E}\int_0^T|X^{\delta,h^\delta}_t|^2dt+C\mathbb{E}\int_0^T|X^{\delta}_t|^2dt+C_T,
\end{eqnarray}
where we used Young's inequality in the last step.

Combining (\ref{27})-(\ref{29}) yields that
\begin{eqnarray*}
\!\!\!\!\!\!\!\!&&\mathbb{E}\Big[\sup_{t\in[0,T]}|X^{\delta,h^\delta}_t|^2\Big]
\nonumber\\
~\leq \!\!\!\!\!\!\!\!&& ~~~ 2|x|^2+C_{M,T}+C_{M,T}\mathbb{E}\int_0^T|X^{\delta,h^\delta}_t|^2dt+C_{M,T}\mathbb{E}\int_0^T|X^{\delta}_t|^2dt+C_{M,T}\mathbb{E}\int_0^T|Y^{\delta,h^\delta}_t|^2dt
\nonumber\\
~\leq \!\!\!\!\!\!\!\!&&~~~ C_{M,T}(1+|y|^2+|x|^2)+C_{M,T}\Big(\frac{\varepsilon}{\delta}\Big)\lambda(\delta)+C_{M,T}\mathbb{E}\int_0^T|X^{\delta,h^\delta}_t|^2dt+C_{M,T}\mathbb{E}\int_0^T|X^{\delta}_t|^2dt.
\end{eqnarray*}
By Gronwall's lemma, we immediately get
\begin{eqnarray*}
\mathbb{E}\Big[\sup_{t\in[0,T]}|X^{\delta,h^\delta}_t|^2\Big]
~\leq \!\!\!\!\!\!\!\!&&~~~C_{M,T}\Big(1+|y|^2+|x|^2+\Big(\frac{\varepsilon}{\delta}\Big)\lambda(\delta)\Big),
\end{eqnarray*}
where we used the uniform estimate of $X^{\delta}_t$  from \cite[Lemma 3.1]{HLLS}.
Without loss of generality we can assume $(\frac{\varepsilon}{\delta})\lambda(\delta)< 1$, thus (\ref{X23}) holds.
\end{proof}

\begin{remark}
It should be pointed out that the integral estimate (\ref{Y24}) is stronger than that of Lemma 4.3 in \cite{HLLS}, and is mainly used to get the second moment (cf.~Subsection \ref{appendix2} below) of the controlled process $Z^{\delta,h^\delta}$ and the convergence (cf.~(\ref{ess4}) below) of the term $\text{I}_{14}^{\delta}(t)$ in the proof of Proposition \ref{p5}.

\end{remark}

\vspace{2mm}
The following lemma is crucial in proving the convergence of solution to the control problem (\ref{e9}), which states that the uniform (w.r.t.~$t\in[0,T]$) moment estimates or even pointwise estimates of controlled process $ Y_{t}^{\delta,h^\delta}$ is blow up when $\delta\to 0$ (hence $\varepsilon\to 0$). However, it is enough for our use in the proof of the convergence of solutions to the control problem (\ref{e9}).

\begin{lemma} \label{PMY}
Under $(\mathbf{A1})$ and $(\mathbf{A3})$, there exists $C>0$ such that for small enough $\ep,\delta>0$ we have
\begin{eqnarray}\label{Y11}
\mathbb{E}\Big[\sup_{t\in [0, T]}|Y_{t}^{\delta,h^\delta}|^{2}\Big]\leq \frac{C(1+|y|^2)T}{\ep}+\frac{C_M\lambda(\delta)^2}{\delta}.
\end{eqnarray}
\end{lemma}

\begin{proof}
We consider process $\tilde{Y}^{\delta,h^\delta}_t:=Y^{\delta,h^\delta}_{\varepsilon t}$ that solves the following equation
\begin{eqnarray}
d\tilde{Y}^{\delta,h^\delta}_t=~~\!\!\!\!\!\!\!\!&&f(X_{t\varepsilon}^{\delta,h^\delta},\mathcal{L}_{X_{t\varepsilon}^{\delta}},\tilde{Y}^{\delta,h^\delta}_t)dt+g(X_{t\varepsilon}^{\delta,h^\delta},\mathcal{L}_{X_{t\varepsilon}^{\delta}},\tilde{Y}^{\delta,h^\delta}_t)d\tilde{W}^2_t,
\nonumber\\
\!\!\!\!\!\!\!\!&&+\frac{\lambda(\delta)\sqrt{\varepsilon}}{\sqrt{\delta}}g(X_{t\varepsilon}^{\delta,h^\delta},\mathcal{L}_{X_{t\varepsilon}^{\delta}},\tilde{Y}^{\delta,h^\delta}_t)h_{t\varepsilon}^{2,\delta}dt,
~\tilde{Y}^{\delta,h^\delta}_0=y,
\end{eqnarray}
where $\tilde{W}^2_t:=\frac{1}{\sqrt{\varepsilon}}W^2_{t\varepsilon}$ that coincides in law with $W^2_t$.

By It\^{o}'s formula, we have
\begin{eqnarray}
|\tilde{Y}_{t}^{\delta,h^\delta}|^{2}=~~\!\!\!\!\!\!\!\!&&|y|^{2}+2\int_{0} ^{t}\langle f(X_{s\varepsilon}^{\delta,h^\delta},\mathcal{L}_{X_{s\varepsilon}^{\delta}},\tilde{Y}_{s}^{\delta,h^\delta}),\tilde{Y}_{s}^{\delta,h^\delta}\rangle ds
\nonumber\\
\!\!\!\!\!\!\!\!&&
+2\int_{0} ^{t}\langle \tilde{Y}_{s}^{\delta,h^\delta}, g(X_{s\varepsilon}^{\delta,h^\delta},\mathcal{L}_{X_{s\varepsilon}^{\delta}},\tilde{Y}_{s}^{\delta,h^\delta})d\tilde{W}^2_s\rangle  \nonumber\\
\!\!\!\!\!\!\!\!&&+\int_{0} ^{t}\|g(X_{s\varepsilon}^{\delta,h^\delta}, \mathcal{L}_{X_{s\varepsilon}^{\delta}}, \tilde{Y}_{s}^{\delta,h^\delta})\|^2ds\!
\nonumber\\
\!\!\!\!\!\!\!\!&&
+\frac{2\lambda(\delta)\sqrt{\varepsilon}}{\sqrt{\delta}}\int_{0} ^{t}\langle g(X_{s\varepsilon}^{\delta,h^\delta},\mathcal{L}_{X_{s\varepsilon}^{\delta}},\tilde{Y}_{s}^{\delta,h^\delta})h_{t\varepsilon}^{2,\delta},\tilde{Y}_{s}^{\delta,h^\delta}\rangle ds.\label{ITO}
\end{eqnarray}
Taking expectation on both sides of (\ref{ITO}), we get
\begin{eqnarray*}
\frac{d}{dt}\mathbb{E}\Big[|\tilde{Y}_{t}^{\delta,h^\delta}|^{2}\Big]=~~\!\!\!\!\!\!\!\!&&   \mathbb{E}\Big[2\langle f(X_{t\varepsilon}^{\delta,h^\delta},\mathcal{L}_{X_{t\varepsilon}^{\delta}},\tilde{Y}_{t}^{\delta,h^\delta}),\tilde{Y}_{t}^{\delta,h^\delta}\rangle+\|g(X_{t\varepsilon}^{\delta,h^\delta}, \mathcal{L}_{X_{t\varepsilon}^{\delta}}, \tilde{Y}_{t}^{\delta,h^\delta})\|^2\Big]\nonumber\\
\!\!\!\!\!\!\!\!&&+\frac{2\lambda(\delta)\sqrt{\varepsilon}}{\sqrt{\delta}}\mathbb{E}\langle g(X_{t\varepsilon}^{\delta,h^\delta},\mathcal{L}_{X_{t\varepsilon}^{\delta}},\tilde{Y}_{t}^{\delta,h^\delta})h_{t\varepsilon}^{2,\delta},\tilde{Y}_{t}^{\delta,h^\delta}\rangle .
\end{eqnarray*}

Note that
$$\frac{2\lambda(\delta)\sqrt{\varepsilon}}{\sqrt{\delta}}|\langle g(X_{t\varepsilon}^{\delta,h^\delta},\mathcal{L}_{X_{t\varepsilon}^{\delta}},\tilde{Y}_{t}^{\delta,h^\delta})h_{t\varepsilon}^{2,\delta},\tilde{Y}_{t}^{\delta,h^\delta}\rangle| \leq\frac{\beta}{2}|\tilde{Y}_{t}^{\delta,h^\delta}|^{2}+\frac{C_{\beta}\lambda(\delta)^2\varepsilon}{\delta}|h_{t\varepsilon}^{2,\delta}|^2.$$
Then by (\ref{RE3}) and Young's inequality, we have
\begin{eqnarray}
\frac{d}{dt}\mathbb{E}\Big[|\tilde{Y}_{t}^{\delta,h^\delta}|^{2}\Big]\leq~~\!\!\!\!\!\!\!\!&&-\frac{\beta}{2}\mathbb{E}|\tilde{Y}_{t}^{\delta,h^\delta}|^{2}+C_{\beta}+\frac{C_{\beta}\lambda(\delta)^2\varepsilon}{\delta}\mathbb{E}|h_{t\varepsilon}^{2,\delta}|^2. \nonumber
\end{eqnarray}
The Gronwall's lemma implies that
\begin{eqnarray*}
\mathbb{E}\Big[|\tilde{Y}_{t}^{\delta,h^\delta}|^{2}\Big]\leq~~\!\!\!\!\!\!\!\!&& |y|^{2}e^{-\frac{\beta}{2} t}+C_{\beta}\int^t_0 e^{-\frac{\beta}{2}(t-s)}ds+\frac{C_{\beta}\lambda(\delta)^2\varepsilon}{\delta}\int^t_0 e^{-\frac{\beta}{2}(t-s)}\mathbb{E}|h_{s\varepsilon}^{2,\delta}|^2ds
\nonumber\\
\leq~~\!\!\!\!\!\!\!\!&&C_{\beta}(1+|y|^2)+\frac{C_{\beta}\lambda(\delta)^2\varepsilon}{\delta}\mathbb{E}\int^t_0 e^{-\frac{\beta}{2}(t-s)}|h_{s\varepsilon}^{2,\delta}|^2ds.
\end{eqnarray*}

Note that \eref{ITO} and \eref{RE3} also imply that
\begin{eqnarray}
|\tilde{Y}_{t}^{\delta,h^\delta}|^{2}\leq~~\!\!\!\!\!\!\!\!&& |y|^{2}+C_{\beta}t+C\left|\int_{0} ^{t}\langle \tilde{Y}_{s}^{\delta,h^\delta}, g(X_{s\varepsilon}^{\delta,h^\delta},\mathcal{L}_{X_{s\varepsilon}^{\delta}},\tilde{Y}_{s}^{\delta,h^\delta})d\tilde{W}^2_s\rangle \right|
\nonumber\\
\!\!\!\!\!\!\!\!&&+C\int_0^t|\tilde{Y}_{s}^{\delta,h^\delta}|^{2}ds+\frac{C\lambda(\delta)^2\varepsilon}{\delta}\int_0^t|h_{s\varepsilon}^{2,\delta}|^2ds.
\end{eqnarray}
By Burkholder-Davis-Gundy's inequality and Young's inequality we get
\begin{eqnarray*}
\mathbb{E}\Big[\sup_{t\in[0,T]}|\tilde{Y}_{t}^{\delta,h^\delta}|^{2}\Big]\leq~~\!\!\!\!\!\!\!\!&&|y|^{2}+C_{\beta}T +C\EE\left[\int_0^T|\tilde{Y}^{\delta,h^\delta}_s|^{2}ds\right]^{\frac{1}{2}}
\nonumber\\
\!\!\!\!\!\!\!\!&&
+C\mathbb{E}\int_0^T|\tilde{Y}_{s}^{\delta,h^\delta}|^{2}ds+\frac{C\lambda(\delta)^2\varepsilon}{\delta}\mathbb{E}\int_0^T|h_{s\varepsilon}^{2,\delta}|^2ds
\nonumber\\
\leq~~\!\!\!\!\!\!\!\!&&|y|^{2}+C_{\beta}T
+C\int_0^T\mathbb{E}|\tilde{Y}_{s}^{\delta,h^\delta}|^{2}ds+\frac{C\lambda(\delta)^2\varepsilon}{\delta}\mathbb{E}\int_0^T|h_{s\varepsilon}^{2,\delta}|^2ds.
\end{eqnarray*}
Consequently, we obtain that for any $T\geq 1$,
\begin{eqnarray*}
\mathbb{E}\Big[\sup_{t\in[0,T]}|\tilde{Y}_{t}^{\delta,h^\delta}|^{2}\Big]\leq~~\!\!\!\!\!\!\!\!&& C_{\beta}(1+|y|^2)T+\frac{C_{\beta}\lambda(\delta)^2\varepsilon}{\delta}\mathbb{E}\int_0^T\int^s_0 e^{-\frac{\beta}{2}(s-r)}|h_{r\varepsilon}^{2,\delta}|^2drds
\nonumber\\
\!\!\!\!\!\!\!\!&&
+\frac{C\lambda(\delta)^2\varepsilon}{\delta}\mathbb{E}\int_0^T|h_{s\varepsilon}^{2,\delta}|^2ds
\nonumber\\
\leq~~\!\!\!\!\!\!\!\!&& C_{\beta}(1+|y|^2)T
+\frac{C_{\beta}\lambda(\delta)^2\varepsilon}{\delta}\mathbb{E}\int_0^T|h_{s\varepsilon}^{2,\delta}|^2ds,
\end{eqnarray*}
{where the last inequality follows from the fact that
\begin{align*}
 &\mathbb{E}\int_0^T\int^s_0 e^{-\frac{\beta}{2}(s-r)}|h_{r\varepsilon}^{2,\delta}|^2drds\\
  = & -\frac{2}{\beta}e^{-\frac{\beta s}{2}}\mathbb{E}\int^s_0 e^{\frac{\beta r}{2}}|h_{r\varepsilon}^{2,\delta}|^2dr+\frac{2}{\beta}\mathbb{E}\int_0^T|h_{s\varepsilon}^{2,\delta}|^2ds\\
 = &-\frac{2}{\beta}\mathbb{E}\int^s_0 e^{\frac{-\beta (s-r)}{2}}|h_{r\varepsilon}^{2,\delta}|^2dr+\frac{2}{\beta}\mathbb{E}\int_0^T|h_{s\varepsilon}^{2,\delta}|^2ds\\
\leq &~C\mathbb{E}\int_0^T|h_{s\varepsilon}^{2,\delta}|^2ds.
\end{align*}}
Hence,  for any $T>0$ and $\varepsilon$ small enough,
\begin{eqnarray*}
\mathbb{E}\Big[\sup_{t\in[0,T]}|Y_{t}^{\delta,h^\delta}|^{2}\Big]=~~\!\!\!\!\!\!\!\!&&\mathbb{E}\Big[\sup_{t\in\left[0,\frac{T}{\varepsilon}\right]}|\tilde{Y}_{t}^{\delta,h^\delta}|^{2}\Big]
\nonumber\\
\leq~~\!\!\!\!\!\!\!\!&&
\frac{C_{\beta}(1+|y|^2)T}{\varepsilon}
+\frac{C_{\beta}\lambda(\delta)^2\varepsilon}{\delta}\mathbb{E}\int_0^{\frac{T}{\varepsilon}}|h_{s\varepsilon}^{2,\delta}|^2ds
\nonumber\\
\leq~~\!\!\!\!\!\!\!\!&&
\frac{C_{\beta}(1+|y|^2)T}{\varepsilon}
+\frac{C_{\beta,M}\lambda(\delta)^2}{\delta},
\end{eqnarray*}
which yields (\ref{Y11}).  The proof is complete.
\end{proof}

{The following lemma is the time H\"{o}lder continuity estimate of controlled process $X^{\delta,h^\delta}$, which is useful in the proof of the existence of viable pair later.
\begin{lemma}\label{xg}
Under $(\mathbf{A1})$ and $(\mathbf{A3})$, there exists $C_T>0$ such that for any $0\leq t\leq t+\Delta\leq T$,
\begin{equation}\label{Holder}
\mathbb{E}\Big[|X_{t+\Delta}^{\delta,h^\delta}-X_{t}^{\delta,h^\delta}|^2\Big]\leq C_{M,T}(1+|x|^2+|y|^2)\Delta.
\end{equation}
\end{lemma}
\begin{proof}
By (\ref{e9}), we know that
\begin{eqnarray}\label{ho1}
\mathbb{E}\Big[|X_{t+\Delta}^{\delta,h^\delta}-X_{t}^{\delta,h^\delta}|^2\Big]=~~\!\!\!\!\!\!\!\!&&\mathbb{E}\Big[\Big|\int_t^{t+\Delta} b(X^{\delta,h^\delta}_s,\mathcal{L}_{X^{\delta}_s},Y^{\delta,h^\delta}_s) ds\Big|^2\Big]
\nonumber\\
\!\!\!\!\!\!\!\!&&
+\lambda(\delta)^2\mathbb{E}\Big[\Big|\int_t^{t+\Delta} \sigma(X^{\delta,h^\delta}_s,\mathcal{L}_{X^{\delta}_s})h^{1,\delta}_s ds\Big|^2\Big]
\nonumber\\
\!\!\!\!\!\!\!\!&&
+\delta\mathbb{E}\Big[\Big|\int_t^{t+\Delta} \sigma(X^{\delta,h^\delta}_s,\mathcal{L}_{X^{\delta}_s})dW_s^1\Big|^2\Big]
\nonumber\\
=:~~\!\!\!\!\!\!\!\!&&\text{I}+\text{II}+\text{III}.
\end{eqnarray}
For the term $\text{I}$, by (\ref{A11}), (\ref{X23}), (\ref{Y24}) and H\"{o}lder's inequality,
\begin{eqnarray}\label{ho2}
\text{I}\leq~~\!\!\!\!\!\!\!\!&&\Delta\mathbb{E}\int_t^{t+\Delta}|b(X^{\delta,h^\delta}_s,\mathcal{L}_{X^{\delta}_s},Y^{\delta,h^\delta}_s)|^2ds
\nonumber\\
\leq~~\!\!\!\!\!\!\!\!&&\Delta\Bigg(1+\mathbb{E}\int_t^{t+\Delta}|X^{\delta,h^\delta}_s|^2ds+\int_t^{t+\Delta}\mathbb{E}|X^{\delta}_s|^2ds+\mathbb{E}\int_t^{t+\Delta}|Y^{\delta,h^\delta}_s|^2ds\Bigg)
\nonumber\\
\leq~~\!\!\!\!\!\!\!\!&&C_{M,T}(1+|x|^2+|y|^2)\Delta.
\end{eqnarray}
For the term $\text{II}$, by $(\mathbf{A3})$ we have
\begin{eqnarray}\label{ho3}
\text{II}\leq~~\!\!\!\!\!\!\!\!&&\lambda(\delta)^2\mathbb{E}\Bigg(\int_t^{t+\Delta} \|\sigma(X^{\delta,h^\delta}_s,\mathcal{L}_{X^{\delta}_s})\|^2ds\cdot \int_0^{T}|h^{1,\delta}_s|^2 ds\Bigg)
\nonumber\\
\leq~~\!\!\!\!\!\!\!\!&&C_{M,T}(1+|x|^2+|y|^2)\lambda(\delta)^2\Delta.
\end{eqnarray}
Similarly, for the term $\text{III}$, by Burkholder-Davis-Gundy's inequality we also have
\begin{eqnarray}\label{ho4}
\text{III}\leq C_{M,T}(1+|x|^2+|y|^2)\delta\Delta.
\end{eqnarray}
Combining (\ref{ho1})-(\ref{ho4}) implies (\ref{Holder}) holds. We complete the proof.
\end{proof}

\section{Proof Theorem \ref{t3}}

 In this section, we aim to give the proof of Theorem \ref{t3}, which will be divided into three parts. In Subsection  \ref{sec4.1}, we construct an auxiliary process associated with the stochastic control problem (\ref{esZ}). In Subsection  \ref{sec4.2}, we intend to prove the criterion (ii) in Hypothesis \ref{h2}.
 In Subsection  \ref{sec4.3}, we show the  criterion (i) in Hypothesis \ref{h2}.

We first show the existence and uniqueness of solutions with some estimates to the skeleton equations (\ref{ske1}).
\begin{lemma}\label{existence of skeleton}
Under $(\mathbf{A1})$ and $(\mathbf{A3})$, for any $x\in\RR^n$, there exists a unique solution to (\ref{ske1})  satisfying
\begin{equation}\label{es19}
\sup_{h\in S_M}\left\{\sup_{t\in[0,T]}|Z^h_t|\right\}\leq C_{T,M}(1+|x|).
\end{equation}
\end{lemma}
\begin{proof}
{By (\ref{A11}) and the fact that $h\in S_M$,} we know
\begin{eqnarray*}
\int_0^T |\sigma(t,\bar{X}_t,\mathcal{L}_{\bar{X}_t})h^1_t|dt\leq~~\!\!\!\!\!\!\!\!&&\Bigg(\int_0^T \|\sigma(t,\bar{X}_t,\mathcal{L}_{\bar{X}_t})\|^2dt\Bigg)^{\frac{1}{2}}\Bigg(\int_0^T|h^1_t|^2dt\Bigg)^{\frac{1}{2}}
\nonumber\\
\leq~~\!\!\!\!\!\!\!\!&&C_{T,M}(1+|x|)<\infty.
\end{eqnarray*}
With this estimate, it is standard to show the existence and uniqueness of solutions to (\ref{ske1}) since it is a linear equation. The estimate (\ref{es19}) follows from Gronwall's inequality.
\end{proof}

\subsection{Construction of the auxiliary process}\label{sec4.1}

%{\color{blue}(HINT: Remember the limiting process of $ X^{\delta,h^\delta}$ is $\bar{X}$ (averaged process)!}

Now for any $\delta>0$ we introduce an auxiliary process $\psi^{\delta,h^\delta}$ by
\begin{equation}\label{e6}
\left\{ \begin{aligned}
d\psi_t^{\delta,h^\delta}=&\frac{1}{\lambda(\delta)}\Big[b\big(\lambda(\delta)Z_t^{\delta,h^\delta} +\bar{X}_t ,\mathcal{L}_{X^{\delta}_t},Y^{\delta,h^\delta}_t\big)-\bar{b}\big(\lambda(\delta)Z_t^{\delta,h^\delta} +\bar{X}_t ,\mathcal{L}_{X^{\delta}_t}\big)\Big]dt\\ \vspace{3mm}
&+\partial_x\bar{b}(\bar{X}_t,\mathcal{L}_{\bar{X}_t})\cdot\psi_t^{\delta,h^\delta} dt+\sigma(\bar{X}_t,\mathcal{L}_{\bar{X}_t})h_t^{1,\delta} dt,\\
\psi_0^{\delta,h^\delta}=0.&
\end{aligned} \right.
\end{equation}

In the following lemma, we show that the difference $Z^{\delta,h^\delta}-\psi^{\delta,h^\delta}$ converges in probability to zero in $C([0,T];\RR^n)$, as $\delta\to0 $.
\begin{lemma}\label{lemma1}
Suppose that conditions $(\mathbf{A1})$-$(\mathbf{A3})$ hold, then  we have
\begin{eqnarray}
\lim_{\delta\rightarrow 0}\EE\Big[\sup_{t\in [0,T]}|Z_t^{\delta,h^\delta}-\psi_t^{\delta,h^\delta}|\Big]=0.
\end{eqnarray}
\end{lemma}
\begin{proof}
Note that
\begin{equation*}
\left\{ \begin{aligned}
d(Z_t^{\delta,h^\delta}-\psi_t^{\delta,h^\delta})=
&
\Bigg(\frac{1}{\lambda(\delta)}\Big(\bar{b}\big(\lambda(\delta)Z_t^{\delta,h^\delta} +\bar{X}_t ,\mathcal{L}_{X^{\delta}_t}\big)-\bar{b}(\bar{X}_t,\mathcal{L}_{X^{\delta}_t})\Big)\\
&
-\partial_x\bar{b}(\bar{X}_t,\mathcal{L}_{\bar{X}_t})\cdot\psi_t^{\delta,h^\delta} \Bigg)dt\\
&+\frac{1}{\lambda(\delta)}\Big[\bar{b}(\bar{X}_t,\mathcal{L}_{X^{\delta}_t})-\bar{b}(\bar{X}_t,\mathcal{L}_{\bar{X}_t})\Big]dt\\
&+\frac{\sqrt{\delta}}{\lambda(\delta)}\sigma\big(\lambda(\delta)Z_t^{\delta,h^\delta} +\bar{X}_t ,\mathcal{L}_{X^{\delta}_t}\big)dW_t^1\\
&+\Big(\sigma\big(\lambda(\delta)Z_t^{\delta,h^\delta} +\bar{X}_t ,\mathcal{L}_{X^{\delta}_t}\big)-\sigma(\bar{X}_t,\mathcal{L}_{\bar{X}_t})\Big)h^{1,\delta}_t dt,\\
Z_0^{\delta,h^\delta}-\psi_0^{\delta,h^\delta}=0.&
\end{aligned} \right.
\end{equation*}
It follows that
\begin{eqnarray}\label{ess11}
\!\!\!\!\!\!\!\!&&\mathbb{E}\Big[\sup_{t\in [0,T]}|Z_t^{\delta,h^\delta}-\psi_t^{\delta,h^\delta}|\Big]
\nonumber\\
\leq~~\!\!\!\!\!\!\!\!&&\mathbb{E}\Bigg(\int_0^T\Big|
\frac{1}{\lambda(\delta)}\Big(\bar{b}\big(\lambda(\delta)Z_s^{\delta,h^\delta} +\bar{X}_s ,\mathcal{L}_{X^{\delta}_s}\big)-\bar{b}(\bar{X}_s,\mathcal{L}_{X^{\delta}_s})\Big)
\nonumber\\
\!\!\!\!\!\!\!\!&&
-\partial_x\bar{b}(\bar{X}_s,\mathcal{L}_{\bar{X}_s})\cdot\psi_s^{\delta,h^\delta}\Big|ds \Bigg)
\nonumber\\
\!\!\!\!\!\!\!\!&&+\frac{1}{\lambda(\delta)}\mathbb{E}\int_0^T|\bar{b}(\bar{X}_t,\mathcal{L}_{X^{\delta}_t})-\bar{b}(\bar{X}_t,\mathcal{L}_{\bar{X}_t})|dt
\nonumber\\
\!\!\!\!\!\!\!\!&&+\frac{\sqrt{\delta}}{\lambda(\delta)}\mathbb{E}\Bigg(\sup_{t\in [0,T]}\Big|\int_0^t\sigma\big(\lambda(\delta)Z_s^{\delta,h^\delta} +\bar{X}_s ,\mathcal{L}_{X^{\delta}_s}\big)dW_s^1\Big|\Bigg)
\nonumber\\
\!\!\!\!\!\!\!\!&&+\mathbb{E}\Bigg(\int_0^T\|\sigma\big(\lambda(\delta)Z_t^{\delta,h^\delta} +\bar{X}_t ,\mathcal{L}_{X^{\delta}_t}\big)-\sigma(\bar{X}_t,\mathcal{L}_{\bar{X}_t})\|\cdot|h^{1,\delta}_t|dt\Bigg)
\nonumber\\
=:~~\!\!\!\!\!\!\!\!&&\sum_{i=1}^4 J_i(T).
\end{eqnarray}
First, notice that
\begin{eqnarray*}
\!\!\!\!\!\!\!\!&&\mathbb{E}\Big|
\bar{b}\big(\lambda(\delta)Z_s^{\delta,h^\delta} +\bar{X}_s ,\mathcal{L}_{X^{\delta}_s}\big)-\bar{b}(\bar{X}_s,\mathcal{L}_{X^{\delta}_s})
-\partial_x\bar{b}(\bar{X}_s,\mathcal{L}_{\bar{X}_s})\cdot\lambda(\delta)Z_s^{\delta,h^\delta}\Big|
\nonumber\\
=~~\!\!\!\!\!\!\!\!&&\mathbb{E}\Big|\int_0^1\partial_r\bar{b}\big(r\lambda(\delta)Z_s^{\delta,h^\delta} +\bar{X}_s ,\mathcal{L}_{X^{\delta}_s}\big)dr-\partial_x\bar{b}(\bar{X}_s,\mathcal{L}_{\bar{X}_s})\cdot\lambda(\delta)Z_s^{\delta,h^\delta}\Big|
\nonumber\\
\leq~~\!\!\!\!\!\!\!\!&&\lambda(\delta)\mathbb{E}\int_0^1\|\partial_x\bar{b}\big(r\lambda(\delta)Z_s^{\delta,h^\delta} +\bar{X}_s ,\mathcal{L}_{X^{\delta}_s}\big)-\partial_x\bar{b}(\bar{X}_s,\mathcal{L}_{\bar{X}_s})\|\cdot|Z_s^{\delta,h^\delta}|dr
\nonumber\\
\leq~~\!\!\!\!\!\!\!\!&&\lambda(\delta)\int_0^1\Big(\mathbb{E}\|\partial_x\bar{b}\big(r\lambda(\delta)Z_s^{\delta,h^\delta} +\bar{X}_s ,\mathcal{L}_{X^{\delta}_s}\big)-\partial_x\bar{b}(\bar{X}_s,\mathcal{L}_{\bar{X}_s})\|^2\Big)^{\frac{1}{2}}dr\cdot\Big(\mathbb{E}|Z_s^{\delta,h^\delta}|^2\Big)^{\frac{1}{2}}.
\end{eqnarray*}
Thus for $J_1(T)$ we can get
\begin{eqnarray}\label{ess12}
\!\!\!\!\!\!\!\!&&J_1(T)
\nonumber\\
\leq~~\!\!\!\!\!\!\!\!&& \frac{1}{\lambda(\delta)}\mathbb{E}\Bigg(\int_0^T\Big|
\bar{b}\big(\lambda(\delta)Z_s^{\delta,h^\delta} +\bar{X}_s ,\mathcal{L}_{X^{\delta}_s}\big)-\bar{b}(\bar{X}_s,\mathcal{L}_{X^{\delta}_s})
\nonumber\\
\!\!\!\!\!\!\!\!&&
-\partial_x\bar{b}(\bar{X}_s,\mathcal{L}_{\bar{X}_s})\cdot\lambda(\delta)Z_s^{\delta,h^\delta}\Big|ds \Bigg)+\mathbb{E}\Bigg(\int_0^T\|\partial_x\bar{b}(\bar{X}_s,\mathcal{L}_{\bar{X}_s})\|\cdot|Z_s^{\delta,h^\delta}-\psi_s^{\delta,h^\delta}|ds\Bigg)
\nonumber\\
\leq~~\!\!\!\!\!\!\!\!&&\int_0^T\Bigg(\int_0^1\Big(\mathbb{E}\|\partial_x\bar{b}\big(r\lambda(\delta)Z_s^{\delta,h^\delta} +\bar{X}_s ,\mathcal{L}_{X^{\delta}_s}\big)-\partial_x\bar{b}(\bar{X}_s,\mathcal{L}_{\bar{X}_s})\|^2\Big)^{\frac{1}{2}}dr
\nonumber\\
\!\!\!\!\!\!\!\!&&
\cdot\Big(\mathbb{E}|Z_s^{\delta,h^\delta}|^2\Big)^{\frac{1}{2}}\Bigg)ds
+C\mathbb{E}\Bigg(\int_0^T|Z_s^{\delta,h^\delta}-\psi_s^{\delta,h^\delta}|ds\Bigg).
\end{eqnarray}
As for $J_2(T)$, by {Lemma \ref{C1}} and the convergence (\ref{33}), we deduce
\begin{eqnarray}\label{ess13}
J_2(T)\leq \frac{1}{\lambda(\delta)}\int_0^T\mathbb{W}_2(\mathcal{L}_{X^{\delta}_t},\mathcal{L}_{\bar{X}_t})dt
\leq \frac{C_T(\sqrt{\delta}+\sqrt{\varepsilon})}{\lambda(\delta)}(1+|x|^2+|y|^2).
\end{eqnarray}
As for $J_3(T)$, by Burkholder-Davis-Gundy's inequality and (\ref{X23}),
\begin{eqnarray}\label{ess14}
J_3(T)\leq~~\!\!\!\!\!\!\!\!&& \frac{\sqrt{\delta}}{\lambda(\delta)}\mathbb{E}\Bigg(\int_0^T\big\|\sigma\big(\lambda(\delta)Z_s^{\delta,h^\delta} +\bar{X}_s ,\mathcal{L}_{X^{\delta}_s}\big)\big\|^2ds\Bigg)^{\frac{1}{2}}
\nonumber\\
\leq~~\!\!\!\!\!\!\!\!&&\frac{\sqrt{\delta}}{\lambda(\delta)}C\Bigg(\mathbb{E}\int_0^T \big(1+|X^{\delta}_s|^2\big)ds\Bigg)^{\frac{1}{2}}
\nonumber\\
\leq~~\!\!\!\!\!\!\!\!&&C_{T,M}\frac{\sqrt{\delta}}{\lambda(\delta)}(1+|x|+|y|).
\end{eqnarray}
Now we recall the following estimate
\begin{equation}\label{ess16}
{\mathbb{E}\Big[\sup_{t\in[0,T]}|Z_t^{\delta,h^\delta}|^2\Big]\leq C_{T,M}(1+|x|^6+|y|^6),}
\end{equation}
whose proof is postponed in Subsection \ref{appendix2} in Appendix.

{As for $J_4(T)$, by (\ref{33}), (\ref{ess16}) and $h^\delta\in\mathcal{A}_M$,} we can see that
\begin{eqnarray}\label{ess15}
J_4(T)\leq~~\!\!\!\!\!\!\!\!&& \mathbb{E}\Bigg(\int_0^T \big(\lambda(\delta)|Z_t^{\delta,h^\delta}|+  \mathbb{W}_2(\mathcal{L}_{X^{\delta}_t},\mathcal{L}_{\bar{X}_t})\big)\cdot|h^{1,\delta}_t|dt\Bigg)
\nonumber\\
\leq~~\!\!\!\!\!\!\!\!&&\mathbb{E}\Bigg(\Big(\int_0^T\big(\lambda(\delta)^2|Z_t^{\delta,h^\delta}|^2+  \mathbb{W}_2(\mathcal{L}_{X^{\delta}_t},\mathcal{L}_{\bar{X}_t})^2\big)dt\Big)^{\frac{1}{2}} \cdot\Big(\int_0^T|h^{1,\delta}_t|^2dt\Big)^{\frac{1}{2}}\Bigg)
\nonumber\\
\leq~~\!\!\!\!\!\!\!\!&&C_{T,M}(\lambda(\delta)+\sqrt{\varepsilon}+\sqrt{\delta})(1+|x|^2+|y|^2).
\end{eqnarray}
Combining (\ref{ess11})-(\ref{ess15}), then by the Gronwall's lemma, the dominated convergence theorem and Lemma \ref{C1}, we conclude that
\begin{eqnarray*}
\lim_{\varepsilon\rightarrow 0}\EE\Big[\sup_{t\in [0,T]}|Z_t^{\delta,h^\delta}-\psi_t^{\delta,h^\delta}|\Big]=0.
\end{eqnarray*}
We complete the proof.
\end{proof}

\subsection{Proof of Hypothesis \ref{h2} (ii)}\label{sec4.2}
In this part, we prove Hypothesis \ref{h2} (ii).
Recall (\ref{e6}),  for any $0\leq t\leq T$, we denote
$$\psi^{\delta,h^\delta}_t=\text{I}^{\delta}(t)+\text{II}^{\delta}(t)+\text{III}^{\delta}(t),$$
where
\begin{eqnarray*}
\text{I}^{\delta}(t):=~~\!\!\!\!\!\!\!\!&&\frac{1}{\lambda(\delta)}\int_0^tb\big(\lambda(\delta)Z_s^{\delta,h^\delta} +\bar{X}_s ,\mathcal{L}_{X^{\delta}_s},Y^{\delta,h^\delta}_s\big)-\bar{b}\big(\lambda(\delta)Z_s^{\delta,h^\delta} +\bar{X}_s ,\mathcal{L}_{X^{\delta}_s}\big)ds,
\nonumber\\
=~~\!\!\!\!\!\!\!\!&&\frac{1}{\lambda(\delta)}\int_0^tb\big(X_s^{\delta,h^\delta}  ,\mathcal{L}_{X^{\delta}_s},Y^{\delta,h^\delta}_s\big)-\bar{b}\big(X_s^{\delta,h^\delta} ,\mathcal{L}_{X^{\delta}_s}\big)ds,
\nonumber\\
\vspace{2mm}
\text{II}^{\delta}(t):=~~\!\!\!\!\!\!\!\!&&\int_{0}^t\partial_x\bar{b}(\bar{X}_s,\mathcal{L}_{\bar{X}_s})\cdot \psi_s^{\delta,h^\delta} ds,
\nonumber\\
\text{III}^{\delta}(t):=~~\!\!\!\!\!\!\!\!&&\int_{0}^t\sigma(\bar{X}_s,\mathcal{L}_{\bar{X}_s})h_s^{1,\delta} ds.
\end{eqnarray*}

\begin{proposition}\label{p5}
Under the assumptions in Theorem \ref{t3}, let $\{ h^\delta\}_{\delta>0}\subset \mathcal{A}_M$ for
any $M<\infty$. Then for any $\eta>0$, we have
\begin{equation}\label{con2}
\lim_{\delta\to 0}\mathbb{P}\Big(d\Big(\Upsilon^\delta\Big(\sqrt{\delta}W_{\cdot}
+\lambda(\delta)\int_0^{\cdot}h^\delta_sd s\Big),\Upsilon^0\big(\int_0^{\cdot}h^\delta_sd s\big)\Big)>\eta\Big)=0,
\end{equation}
where $d(\cdot,\cdot)$ denotes the metric in the space $C([0,T];\RR^n)$, i.e.,
$$d(u_1,u_2):=\sup_{t\in [0,T]}|u_1(t)-u_2(t)|,\quad u_1,u_2\in C([0,T];\RR^n).$$
\end{proposition}

\begin{proof}
\textbf{Step 1:} In this step, we prove
\begin{equation}\label{con1}
\lim_{\delta\to 0}\mathbb{E}\Big[\sup_{t\in [0,T]}|\psi^{\delta,h^\delta}_t-Z_t^{h^\delta}|\Big]=0,
\end{equation}
where $Z^{h^\delta}$ is the solution of (\ref{ske1}) with $h^{1,\delta}$ replacing $h^1$.    In fact, if (\ref{con1}) holds, then by Lemma \ref{lemma1} we obtain for any $\eta>0$,
\begin{eqnarray*}
\!\!\!\!\!\!\!\!&&\mathbb{P}\Big(\sup_{t\in [0,T]}|Z^{\delta,h^\delta}_t-Z_t^{h^\delta}|>\eta \Big)
\nonumber\\
\leq~~\!\!\!\!\!\!\!\!&&\mathbb{P}\Big(\sup_{t\in [0,T]}|Z^{\delta,h^\delta}_t-\psi^{\delta,h^\delta}_t|>{\frac{\eta}{2}} \Big)+\mathbb{P}\Big(\sup_{t\in [0,T]}|\psi^{\delta,h^\delta}_t-Z_t^{h^\delta}|>{\frac{\eta}{2}} \Big)
\nonumber\\
\leq~~\!\!\!\!\!\!\!\!&&\frac{2\EE\Big[\sup_{t\in [0,T]}|Z_t^{\delta,h^\delta}-\psi_t^{\delta,h^\delta}|\Big]}{\eta}+
\frac{2\mathbb{E}\Big[\sup_{t\in [0,T]}|\psi^{\delta,h^\delta}_t-Z_t^{h^\delta}|\Big]}{\eta}
\to 0,~\text{as}~\delta\to 0,
\end{eqnarray*}
which implies the desired result (\ref{con2}). To prove (\ref{con1}), we recall
\begin{equation*}
\psi^{\delta,h^\delta}_t-Z_t^{h^\delta}=\text{I}^{\delta}(t)+\int_{0}^t\partial_x\bar{b}(\bar{X}_s,\mathcal{L}_{\bar{X}_s})\cdot (\psi_s^{\delta,h^\delta}-Z_s^{h^\delta}) ds.
\end{equation*}
Then,  {by Lemma \ref{C1}}, it follows that
\begin{equation}\label{esc1}
\mathbb{E}\Big[\sup_{t\in [0,T]}|\psi^{\delta,h^\delta}_t-Z_t^{h^\delta}|\Big]\leq C\mathbb{E}\Big[\sup_{t\in [0,T]}|\text{I}^{\delta}(t)|\Big]+C_T\mathbb{E}\int_{0}^T|\psi_t^{\delta,h^\delta}-Z_t^{h^\delta}| dt.
\end{equation}
Once we can prove
\begin{equation}\label{esc2}
\lim_{\delta\to 0}\mathbb{E}\Big[\sup_{t\in [0,T]}|\text{I}^{\delta}(t)|\Big]=0,
\end{equation}
by using Gronwall's inequality to (\ref{esc1}), we can get
$$\mathbb{E}\Big[\sup_{t\in [0,T]}|\psi^{\delta,h^\delta}_t-Z_t^{h^\delta}|\Big]\leq C_T\mathbb{E}\Big[\sup_{t\in [0,T]}|\text{I}^{\delta}(t)|\Big]\to 0,~\text{as}~\delta\to 0.$$

\textbf{Step 2:} In this step, we prove (\ref{esc2}). To this end, {we recall the  regularity of Poisson equation (\ref{PE}), which have been established  in Proposition \ref{P3.6},} by It\^{o}'s formula (cf.~\cite[Theorem 7.1]{BLPR}) for the function $\Phi(x,\mu,y)$, we have
\begin{eqnarray*}
&&\Phi(X_t^{\delta,h^\delta} ,\mathcal{L}_{X^{\delta}_t},Y^{\delta,h^\delta}_t)\\
=~~\!\!\!\!\!\!\!\!&&\Phi(x,\delta_x,y)
+\int^t_0 \EE\left[b(X^{\delta}_s,\mathcal{L}_{ X^{\delta}_{s}}, Y^{\delta}_s)\partial_{\mu}\Phi(x,\mathcal{L}_{X^{\delta}_{s}},y)(X^{\delta}_s)\right]\Big|_{x=X_s^{\delta,h^\delta} ,y=Y^{\delta,h^\delta}_{s}}ds\\
&&+\int^t_0 \frac{\delta}{2}\EE \text{Tr}\left[\sigma\sigma^{*}(X^{\delta}_s,\mathcal{L}_{ X^{\delta}_{s}})\partial_z\partial_{\mu}\Phi(x,\mathcal{L}_{X^{\delta}_{s}},y)(X^{\delta}_s)\right]\Big|_{x=X_s^{\delta,h^\delta} ,y=Y^{\delta,h^\delta}_{s}}ds\\
&&+\int^t_0 \mathbf{L}^{1,\delta}_{\mathcal{L}_{X^{\delta}_{s}},Y^{\delta,h^\delta}_{s}}\Phi(X_s^{\delta,h^\delta},\mathcal{L}_{X^{\delta}_{s}},Y^{\delta,h^\delta}_{s})ds\\
&&+\lambda(\delta)\int^t_0 \partial_x \Phi(X_s^{\delta,h^\delta} ,\mathcal{L}_{X^{\delta}_s},Y^{\delta,h^\delta}_s)\cdot \sigma(X^{\delta,h^\delta}_s,\mathcal{L}_{X^{\delta}_s})h^{1,\delta}_s ds\\
&&+\frac{1}{\varepsilon}\int^t_0 \mathbf{L}^{2}_{X_s^{\delta,h^\delta},\mathcal{L}_{X^{\delta}_{s}}}\Phi(X_s^{\delta,h^\delta},\mathcal{L}_{X^{\delta}_{s}},Y^{\delta,h^\delta}_{s})ds
\\
&&+\frac{\lambda(\delta)}{\sqrt{\delta\varepsilon}}\int^t_0\partial_y \Phi(X_s^{\delta,h^\delta} ,\mathcal{L}_{X^{\delta}_s},Y^{\delta,h^\delta}_s)\cdot g(X^{\delta,h^\delta}_s,\mathcal{L}_{X^{\delta}_s},Y^{\delta,h^\delta}_s)h^{2,\delta}_s ds
\\
&&+M^{1,\delta,h^\delta}_t+\frac{1}{\sqrt{\varepsilon}}M^{2,\delta,h^\delta}_t,
\end{eqnarray*}
where $\mathbf{L}^{1,\delta}_{\mu,y}\Phi(x,\mu,y):=(\mathfrak{L}^{1,\delta}_{\mu,y}\Phi_1(x,\mu,y),\ldots, \mathfrak{L}^{1,\delta}_{\mu,y}\Phi_n(x,\mu,y))$ with
\begin{eqnarray}\label{inf2}
 \mathfrak{L}^{1,\delta}_{\mu,y}\Phi_k(x,\mu,y):=~~\!\!\!\!\!\!\!\!&&\left\langle b(x,\mu,y), \partial_x \Phi_k(x,\mu,y)\right \rangle\nonumber\\
 &&+\frac{\delta}{2}\text{Tr}\left[\sigma\sigma^{*}(x,\mu)\partial^2_{xx} \Phi_k(x,\mu,y)\right ],\quad k=1,\ldots, n,
\end{eqnarray}
and $M^{1,\delta,h^\delta}_t, M^{2,\delta,h^\delta}_t$ are two local martingales which are defined by
\begin{eqnarray}
&&M^{1,\delta,h^\delta}_t:=\sqrt{\delta}\int^t_0 \partial_x \Phi(X_{s}^{\delta,h^\delta},\mathcal{L}_{X^{\delta}_{s}},Y_{s}^{\delta,h^\delta})\cdot \sigma(X^{\delta,h^\delta}_s,\mathcal{L}_{ X^{\delta}_{s}}) dW^1_s,\label{mart1}\\
&&M^{2,\delta,h^\delta}_t:=\int^t_0 \partial_y \Phi(X_{s}^{\delta,h^\delta},\mathcal{L}_{X^{\delta}_{s}},Y_{s}^{\delta,h^\delta})\cdot g(X^{\delta,h^\delta}_s,\mathcal{L}_{ X^{\delta}_{s}}, Y^{\delta,h^\delta}_s) dW^2_s\label{mart2},
\end{eqnarray}
{and
\begin{equation*}
-\mathbf{L}^{2}_{x,\mu}\Phi(x,\mu,y)=b(x,\mu,y)-\bar{b}(x,\mu),\label{PE}
\end{equation*}
where $\Phi(x,\mu,y):=(\Phi_1(x,\mu,y),\ldots, \Phi_n(x,\mu,y))$,
$$\mathbf{L}^{2}_{x,\mu}\Phi(x,\mu,y):=(\mathfrak{L}^{2}_{x,\mu}\Phi_1(x,\mu,y),\ldots, \mathfrak{L}^{2}_{x,\mu}\Phi_n(x,\mu,y))$$
and for any $k=1,\ldots,n,$
\begin{eqnarray*}
\mathfrak{L}^{2}_{x,\mu}\Phi_k(x,\mu,y):=\langle f(x,\mu,y), \partial_y \Phi_k(x,\mu,y)\rangle+\frac{1}{2}\text{Tr}[g g^{*}(x,\mu,y)\partial^2_{yy} \Phi_k(x,\mu,y)].
\end{eqnarray*}}
Then it follows that
\begin{eqnarray}\label{400}
\text{I}^{\delta}(t)=~~\!\!\!\!\!\!\!\!&&\frac{\varepsilon}{\lambda(\delta)}\Big\{-\Phi(X_t^{\delta,h^\delta} ,\mathcal{L}_{X^{\delta}_t},Y^{\delta,h^\delta}_t)+\Phi(x,\delta_x,y)
\nonumber\\
\!\!\!\!\!\!\!\!&&+\int^t_0 \EE\left[b(X^{\delta}_s,\mathcal{L}_{ X^{\delta}_{s}}, Y^{\delta}_s)\partial_{\mu}\Phi(x,\mathcal{L}_{X^{\delta}_{s}},y)(X^{\delta}_s)\right]\Big|_{x=X_s^{\delta,h^\delta} ,y=Y^{\delta,h^\delta}_{s}}ds\nonumber\\
\!\!\!\!\!\!\!\!&&+\int^t_0 {\frac{\delta}{2}}\EE \text{Tr}\left[\sigma\sigma^{*}(X^{\delta}_s,\mathcal{L}_{ X^{\delta}_{s}})\partial_z\partial_{\mu}\Phi(x,\mathcal{L}_{X^{\delta}_{s}},y)(X^{\delta}_s)\right]\Big|_{x=X_s^{\delta,h^\delta} ,y=Y^{\delta,h^\delta}_{s}}ds\nonumber\\
\!\!\!\!\!\!\!\!&&+\int^t_0 \mathbf{L}^{1,\delta}_{\mathcal{L}_{X^{\delta}_{s}},Y^{\delta,h^\delta}_{s}}\Phi(X_s^{\delta,h^\delta},\mathcal{L}_{X^{\delta}_{s}},Y^{\delta,h^\delta}_{s})ds
+M^{1,\delta,h^\delta}_t\Big\}
\nonumber\\
\!\!\!\!\!\!\!\!&&+\varepsilon\int^t_0 \partial_x \Phi(X_s^{\delta,h^\delta} ,\mathcal{L}_{X^{\delta}_s},Y^{\delta,h^\delta}_s)\cdot \sigma(X^{\delta,h^\delta}_s,\mathcal{L}_{X^{\delta}_s})h^{1,\delta}_s ds
\nonumber\\
\!\!\!\!\!\!\!\!&&+\frac{\sqrt{\varepsilon}}{\sqrt{\delta}}\int^t_0\partial_y \Phi(X_s^{\delta,h^\delta} ,\mathcal{L}_{X^{\delta}_s},Y^{\delta,h^\delta}_s)\cdot g(X^{\delta,h^\delta}_s,\mathcal{L}_{X^{\delta}_s},Y^{\delta,h^\delta}_s)h^{2,\delta}_s ds
\nonumber\\
\!\!\!\!\!\!\!\!&&
+\frac{\sqrt{\varepsilon}}{\lambda(\delta)}M^{2,\delta,h^\delta}_t\nonumber\\
=:~~\!\!\!\!\!\!\!\!&&\sum_{i=1}^{3}\text{I}_{i}^{\delta}(t)+\frac{\sqrt{\varepsilon}}{\lambda(\delta)}M^{2,\delta,h^\delta}_t.
\end{eqnarray}
We now consider the terms on the right hand side of (\ref{400}) one by one. In order to estimate $I_{1}^{\delta}(t)$, we denote
$$\text{I}_{1}^{\delta}(t)=\sum_{i=1}^{4}\text{I}_{1i}^{\delta}(t)+M^{1,\delta,h^\delta}_t,$$
where
\begin{eqnarray*}
\!\!\!\!\!\!\!\!&&\text{I}_{11}^{\delta}(t):=\frac{\varepsilon}{\lambda(\delta)}\Big[\Phi(x,\delta_x,y)-\Phi(X_t^{\delta,h^\delta} ,\mathcal{L}_{X^{\delta}_t},Y^{\delta,h^\delta}_t)\Big],
\nonumber\\
\!\!\!\!\!\!\!\!&&\text{I}_{12}^{\delta}(t):=\frac{\varepsilon}{\lambda(\delta)}\int^t_0 \EE\left[b(X^{\delta}_s,\mathcal{L}_{ X^{\delta}_{s}}, Y^{\delta}_s)\partial_{\mu}\Phi(x,\mathcal{L}_{X^{\delta}_{s}},y)(X^{\delta}_s)\right]\Big|_{x=X_s^{\delta,h^\delta} ,y=Y^{\delta,h^\delta}_{s}}ds,
\nonumber\\
\!\!\!\!\!\!\!\!&&\text{I}_{13}^{\delta}(t):={\frac{\delta\varepsilon}{2\lambda(\delta)}}\int^t_0 \EE \text{Tr}\left[\sigma\sigma^{*}(X^{\delta}_s,\mathcal{L}_{ X^{\delta}_{s}})\partial_z\partial_{\mu}\Phi(x,\mathcal{L}_{X^{\delta}_{s}},y)(X^{\delta}_s)\right]\Big|_{x=X_s^{\delta,h^\delta} ,y=Y^{\delta,h^\delta}_{s}}ds,
\nonumber\\
\!\!\!\!\!\!\!\!&&\text{I}_{14}^{\delta}(t):=\frac{\varepsilon}{\lambda(\delta)}\int^t_0 \mathbf{L}^{1,\delta}_{\mathcal{L}_{X^{\delta}_{s}},Y^{\delta,h^\delta}_{s}}\Phi(X_s^{\delta,h^\delta},\mathcal{L}_{X^{\delta}_{s}},Y^{\delta,h^\delta}_{s})ds.
\end{eqnarray*}
As for $\text{I}_{11}^{\delta}(t)$, we use the estimates (\ref{PHI1}), (\ref{X23}) and (\ref{Y11}) to get
\begin{eqnarray}\label{ess1}
\mathbb{E}\Big[\sup_{t\in[0,T]}|\text{I}_{11}^{\delta}(t)|^2\Big]\leq\!\!\!\!\!\!\!\!&&~~~\frac{C\varepsilon^2}{\lambda(\delta)^2}\Bigg(1+|y|^2+\mathbb{E}\Big[\sup_{t\in[0,T]}|Y^{\delta,h^\delta}_t|^2\Big]\Bigg)
\nonumber\\
\leq\!\!\!\!\!\!\!\!&&~~~\frac{C_T(1+|y|^2)\varepsilon}{\lambda(\delta)^2}+\frac{C_M\varepsilon^2}{\delta}.
\end{eqnarray}
As for $\text{I}_{12}^{\delta}(t)$, by (\ref{PHI1}) and (\ref{Y24})
\begin{eqnarray}\label{ess2}
\!\!\!\!\!\!\!\!&&\mathbb{E}\Big[\sup_{t\in[0,T]}|\text{I}_{12}^{\delta}(t)|^2\Big]
\nonumber\\
\leq\!\!\!\!\!\!\!\!&&~~~\frac{\varepsilon^2}{\lambda(\delta)^2}\mathbb{E}\Bigg(\int_0^T\Big(\mathbb{E}|b(X^{\delta}_s,\mathcal{L}_{ X^{\delta}_{s}}, Y^{\delta}_s)|^2\Big)^{\frac{1}{2}}\cdot\|\partial_{\mu}\Phi(X_s^{\delta,h^\delta},\mathcal{L}_{X^{\delta}_{s}},Y^{\delta,h^\delta}_{s})\|_{{L^2(\mathcal{L}_{X^{\delta}_{s}})}}ds\Bigg)^2
\nonumber\\
\leq\!\!\!\!\!\!\!\!&&~~~\frac{C\varepsilon^2}{\lambda(\delta)^2}\Big(\sup_{s\in[0,T]}\big(1+\mathbb{E}|X^{\delta}_s|^2+\mathbb{E}|Y^{\delta}_s|^2\big)\Big)\cdot
\Big(\mathbb{E}\int_0^T\big(1+|Y^{\delta,h^\delta}_{s}|^2\big)ds\Big)
\nonumber\\
\leq\!\!\!\!\!\!\!\!&&~~~\frac{C_{T,M}\varepsilon^2}{\lambda(\delta)^2}(1+|x|^4+|y|^4),
\end{eqnarray}
{where we recall the following estimates of $X^{\delta},Y^{\delta}$ established in \cite[Lemma 3.1]{RSX1},
$$\mathbb{E}\Big[\sup_{t\in[0,T]}\big(|X^{\delta}_t|^4+|Y^{\delta}_t|^4\big)\Big]\leq C_T(1+|x|^4+|y|^4).$$}

Similarly, for $\text{I}_{13}^{\delta}(t)$ and $\text{I}_{14}^{\delta}(t)$ we obtain
\begin{eqnarray}\label{ess3}
\!\!\!\!\!\!\!\!&&\mathbb{E}\Big[\sup_{t\in[0,T]}|\text{I}_{13}^{\delta}(t)|^2\Big]
\nonumber\\
\leq\!\!\!\!\!\!\!\!&&~~~\frac{C\delta^2\varepsilon^2}{\lambda(\delta)^2}\mathbb{E}\Bigg(\int_0^T\Big(\mathbb{E}\|\sigma(X^{\delta}_s,\mathcal{L}_{ X^{\delta}_{s}})\|^4\Big)^{\frac{1}{2}}\cdot\|\partial_{z}\partial_{\mu}\Phi(X_s^{\delta,h^\delta},\mathcal{L}_{X^{\delta}_{s}},Y^{\delta,h^\delta}_{s})\|_{L^2(\Omega)}ds\Bigg)^2
\nonumber\\
\leq\!\!\!\!\!\!\!\!&&~~~\frac{C_T\delta^2\varepsilon^2}{\lambda(\delta)^2}\Big(\sup_{s\in[0,T]}\big(1+\mathbb{E}|X^{\delta}_s|^4\big)\Big)\cdot
\Big(\mathbb{E}\int_0^T\big(1+|Y^{\delta,h^\delta}_{s}|^2\big)ds\Big)
\nonumber\\
\leq\!\!\!\!\!\!\!\!&&~~~\frac{C_{T,M}\delta^2\varepsilon^2}{\lambda(\delta)^2}(1+|x|^6+|y|^6)
\end{eqnarray}
and
\begin{eqnarray}\label{ess4}
\!\!\!\!\!\!\!\!&&\mathbb{E}\Big[\sup_{t\in[0,T]}|\text{I}_{14}^{\delta}(t)|^2\Big]
\nonumber\\
\leq\!\!\!\!\!\!\!\!&&~~~\frac{C\varepsilon^2}{\lambda(\delta)^2}\mathbb{E}\Bigg(\int_0^T\Big(1+|X^{\delta,h^\delta}_s|+(\mathbb{E}| X^{\delta}_{s}|^2)^{\frac{1}{2}}+|Y^{\delta,h^\delta}_s|\Big)\Big(1+|Y^{\delta,h^\delta}_s|\Big)ds\Bigg)^2
\nonumber\\
\!\!\!\!\!\!\!\!&&~~~+\frac{C\varepsilon^2\delta^2}{\lambda(\delta)^2}\mathbb{E}\Bigg(\int_0^T\Big(1+\mathbb{E}| X^{\delta}_{s}|^2\Big)\Big(1+|Y^{\delta,h^\delta}_s|\Big)ds\Bigg)^2
\nonumber\\
\leq\!\!\!\!\!\!\!\!&&~~~\frac{C_{T,M}\varepsilon^2}{\lambda(\delta)^2}(1+|x|^4+|y|^4)+\frac{C_{T,M}\varepsilon^2\delta^2}{\lambda(\delta)^2}(1+|x|^6+|y|^6).
\end{eqnarray}
As for $M^{1,\delta,h^\delta}_t$, by (\ref{Y24}) and Burkholder-Davis-Gundy's inequality,
\begin{eqnarray}\label{ess5}
\!\!\!\!\!\!\!\!&&\mathbb{E}\Big[\sup_{t\in[0,T]}|M^{1,\delta,h^\delta}_t|^2\Big]
\nonumber\\
\leq\!\!\!\!\!\!\!\!&&~~~\frac{\varepsilon^2\delta^2}{\lambda(\delta)^2}\mathbb{E}\Bigg(\int_0^T
\|\partial_x \Phi(X_{s}^{\delta,h^\delta},\mathcal{L}_{X^{\delta}_{s}},Y_{s}^{\delta,h^\delta})\|^2\cdot \|\sigma(X^{\delta,h^\delta}_s,\mathcal{L}_{ X^{\delta}_{s}})\|^2ds\Bigg)
\nonumber\\
\leq\!\!\!\!\!\!\!\!&&~~~\frac{C\varepsilon^2\delta^2}{\lambda(\delta)^2}\mathbb{E}\Bigg(\int_0^T\big(1+|Y_{s}^{\delta,h^\delta}|^2\big)\cdot\big(1+\mathbb{E}|X^{\delta}_{s}|^2\big)ds\Bigg)
\nonumber\\
\leq\!\!\!\!\!\!\!\!&&~~~\frac{C_{T,M}\varepsilon^2\delta^2}{\lambda(\delta)^2}(1+|x|^4+|y|^4).
\end{eqnarray}
Combining (\ref{ess1})-(\ref{ess5}), by Jensen's inequality, we deduce that
\begin{eqnarray}\label{ess6}
\mathbb{E}\Big[\sup_{t\in[0,T]}|\text{I}_{1}^{\delta}(t)|^2\Big]\leq\frac{C_{T,M}\varepsilon}{\lambda(\delta)^2}(1+|x|^6+|y|^6)+\frac{C_M\varepsilon^2}{\delta}.
\end{eqnarray}
Finally, if we can prove
\begin{eqnarray}\label{ess7}
\sum_{i=2}^{3}\mathbb{E}\Big[\sup_{t\in[0,T]}|\text{I}_{i}^{\delta}(t)|^2\Big]\to0,~\text{as}~\delta\to 0,
\end{eqnarray}
and
\begin{eqnarray}\label{ess9}
\mathbb{E}\Big[\sup_{t\in[0,T]}|M^{2,\delta,h^\delta}_t|^2\Big]\leq C_T,
\end{eqnarray}
then combining (\ref{ess6})-(\ref{ess9}), by the assumption in Regime 1 of (\ref{regime}), we can conclude that (\ref{esc2}) holds.

\textbf{Step 3:} In this step, we prove (\ref{ess7}) and (\ref{ess9}). First, due to (\ref{A11}) and (\ref{PHI1}), we deduce
\begin{eqnarray}\label{ess8}
\!\!\!\!\!\!\!\!&&\mathbb{E}\Big[\sup_{t\in[0,T]}|\text{I}_{2}^{\delta}(t)|^2\Big]
\nonumber\\
\leq \!\!\!\!\!\!\!\!&&~~~\varepsilon^2\mathbb{E}\Bigg(\int_0^T\big(1+|Y_{s}^{\delta,h^\delta}|\big)\big(1+(\mathbb{E}|X_{s}^{\delta}|^2)^{\frac{1}{2}}\big)|h_s^{1,\delta}|ds\Bigg)^2
\nonumber\\
\leq \!\!\!\!\!\!\!\!&&~~~\varepsilon^2\mathbb{E}\Bigg(\int_0^T\big(1+|Y_{s}^{\delta,h^\delta}|^2\big)\big(1+\mathbb{E}|X_{s}^{\delta}|^2\big)ds\cdot\int_0^T\big(1+|h_s^{1,\delta}|^2\big)ds\Bigg)
\nonumber\\
\leq \!\!\!\!\!\!\!\!&&~~~C_{T,M}\varepsilon^2\mathbb{E}\int_0^T\big(1+|Y_{s}^{\delta,h^\delta}|^2\big)\big(1+\mathbb{E}|X_{s}^{\delta}|^2\big)ds
\nonumber\\
\leq \!\!\!\!\!\!\!\!&&~~~C_{T,M}{(1+|x|^2)}\varepsilon^2\mathbb{E}\int_0^T\big(1+|Y_{s}^{\delta,h^\delta}|^2\big)ds
\nonumber\\
\leq \!\!\!\!\!\!\!\!&&~~~C_{T,M}{(1+|x|^2)}\varepsilon^2(1+|x|^4+|y|^4).
\end{eqnarray}
Similarly, for $\text{I}_{3}^{\delta}(t)$ we have
\begin{eqnarray}\label{ess10}
\mathbb{E}\Big[\sup_{t\in[0,T]}|\text{I}_{3}^{\delta}(t)|^2\Big]\leq \frac{C_{T,M}\varepsilon}{\delta}.
\end{eqnarray}
Finally, by Burkholder-Davis-Gundy's inequality, $(\mathbf{A3})$ and (\ref{PHI3}), it is easy to see that
\begin{eqnarray*}
\mathbb{E}\Big[\sup_{t\in[0,T]}|M^{2,\delta,h^\delta}_t|^2\Big]\leq\mathbb{E}\int^T_0 \|\partial_y \Phi(X_{s}^{\delta,h^\delta},\mathcal{L}_{X^{\delta}_{s}},Y_{s}^{\delta,h^\delta})\|^2\cdot \|g(X^{\delta,h^\delta}_s,\mathcal{L}_{ X^{\delta}_{s}}, Y^{\delta,h^\delta}_s)\|^2 ds
\leq C_T.
\end{eqnarray*}
We complete the proof.
\end{proof}
\subsection{Proof of Hypothesis \ref{h2} (i)} \label{sec4.3}
In this part, we prove Hypothesis \ref{h2} (i) by the following proposition.
\begin{proposition}\label{p4}
Under the assumptions of Theorem \ref{t3}, let $\{h^\delta\}_{\delta>0}\subset S_M$ for any $M<\infty$ such that $h^\delta$ converges to element $h$ in $S_M$ as $\delta\to0$, then
$\Upsilon^0\big(h^\delta\big)$ converges to $\Upsilon^0\big(h\big)$ in  $C([0,T];\RR^n)$.
\end{proposition}

\begin{proof}
Recall $Z^h,Z^{h^\delta}$ are the solutions of (\ref{ske1}) corresponding to $h,h^\delta\in S_M $ respectively. It suffices to prove the following result
$$\lim_{\delta\to 0}\sup_{t\in[0,T]}|Z^{h^\delta}_t-Z^{h}_t|=0.$$
Since (\ref{ske1}) is a linear equation, the proof is standard and separated into the following two steps:

\vspace{2mm}
\textbf{Step 1:} In this step, we claim that $\{Z^{h^\delta}\}_{\delta>0}$  is pre-compact in $C([0,T];\mathbb{R}^n)$. Firstly, by (\ref{es19}), we know $\{Z^{h^\delta}\}_{\delta>0}$ is uniformly bounded. Next, for any $s,t\in[0,T]$ with $s<t$, by
\begin{eqnarray*}
|Z^{h^\delta}_t-Z^{h^\delta}_s|\leq~~\!\!\!\!\!\!\!\!&& \int_s^t\|\partial_x\bar{b}(\bar{X}_r,\mathcal{L}_{\bar{X}_r})\|\cdot |Z_r^{h^\delta}|dr+\Big(\int_s^t|h^{1,\delta}_r|^2dr\Big)^{\frac{1}{2}}\Big(\int_s^t\|\sigma(\bar{X}_r,\mathcal{L}_{\bar{X}_r})\|^2dr\Big)^{\frac{1}{2}}
\nonumber\\
\leq~~\!\!\!\!\!\!\!\!&& {C_{M,T,x}}|t-s|^{\frac{1}{2}},
\end{eqnarray*}
where $C_{M,T}>0$ is independent of $\delta$. Thus $\{Z^{h^\delta}\}_{\delta>0}$ is equi-continuous in $C([0,T];\mathbb{R}^n)$. The claim follows.

\vspace{2mm}
\textbf{Step 2:}  Let $\bar{Z}$ be a limit of some subsequence denoted by $\{Z^{h^{\delta_i}}\}_{i\geq 1}$ of  $\{Z^{h^\delta}\}_{\delta>0}$ in $C([0,T];\mathbb{R}^n)$, i.e.,
 \begin{equation}\label{esz0}
 \lim_{i\to\infty}\sup_{t\in[0,T]}|Z^{h^{\delta_i}}_t-\bar{Z}_t|=0.
 \end{equation}
 We will show that $\bar{Z}=Z^h$.
Note that
\begin{equation*}
\sup_{t\in[0,T]}|\bar{Z}_t|\leq \sup_{i\geq 1}\sup_{t\in[0,T]}|Z^{h^{\delta_i}}_t|\leq \sup_{h\in S_M}\sup_{t\in[0,T]}|Z^{h}_t|\leq C_{M,T}<\infty.
\end{equation*}
Since $h^{\delta_i}$ converges to $h$ weakly in $S_M$ as $i\to \infty$, for any $e\in\mathbb{R}^{n}$, we obtain
$$\int_0^t\langle \sigma(\bar{X}_s,\mathcal{L}_{\bar{X}_s})P_1h^{\delta_i}_s,e\rangle ds\to \int_0^t\langle \sigma(\bar{X}_s,\mathcal{L}_{\bar{X}_s})P_1h_s,e\rangle ds,~i\to \infty,$$
which implies
 \begin{equation}\label{esz2}
\int_0^t \sigma(\bar{X}_s,\mathcal{L}_{\bar{X}_s})P_1h^{\delta_i}_s ds\to \int_0^t \sigma(\bar{X}_s,\mathcal{L}_{\bar{X}_s})P_1h_s ds,~i\to \infty.
 \end{equation}
Recall
 \begin{equation}\label{esz3}
dZ_t^{h^{\delta_i}}=  \partial_x\bar{b}(\bar{X}_t,\mathcal{L}_{\bar{X}_t})\cdot Z_t^{h^{\delta_i}}dt+\sigma(\bar{X}_t,\mathcal{L}_{\bar{X}_t})P_1h^{\delta_i}_tdt,~Z_0^{h^{\delta_i}}=0.
 \end{equation}
Now letting $i\to\infty$, combining (\ref{esz0}) and (\ref{esz2}) and the uniqueness of solutions to (\ref{esz3}), it leads to $\bar{Z}=Z^h$.
We complete the proof.
\end{proof}

\vspace{3mm}
Now we are able to finish the proof of Theorem \ref{t3}.

\vspace{2mm}
\textbf{Proof of Theorem \ref{t3}}. Note that Propositions \ref{p5} and \ref{p4}  imply conditions (i) and (ii) in Hypothesis \ref{h2} hold respectively, then by Lemma \ref{app1} we get  $\{X^{\delta}\}_{\delta>0}$ satisfies the MDP in $C([0,T];\RR^n)$ with a good rate function $I$ defined in (\ref{rf}).

\section{Proof of Theorem \ref{thj1}}\label{sec5}
\subsection{Tightness of controlled processes}\label{sec5.1}

In the following proposition, we show the tightness of $\{(Z^{\delta,h^\delta},\mathbf{P}^{\delta,\Delta})\}_{\delta>0}$ and uniform integrability of $\{\mathbf{P}^{\delta,\Delta}\}_{\delta>0}$.
\begin{proposition}\label{prj1}
Suppose $\{h^\delta\}_{\delta>0}\subset\mathcal{A}_M$ for any $M<\infty$.  We have

\vspace{1mm}
{\rm(i) } The family $\{(Z^{\delta,h^\delta},\mathbf{P}^{\delta,\Delta})\}_{\delta>0}$ is tight.

\vspace{1mm}
{\rm(ii)} Define the set
$$\mathbb{U}_N:=\Big\{(h^1,h^2,y)\in \mathcal{Z}_1\times\mathcal{Z}_2\times\mathcal{Y}:|h^1|> N,|h^2|> N,|y|> N\Big\}.$$

Then the family $\{\mathbf{P}^{\delta,\Delta}\}_{\delta>0}$ is uniformly integrable in the sense that
$$\lim_{N\rightarrow\infty}\sup_{\delta>0}\mathbb{E}\Bigg\{\int_{\mathbb{U}_N\times[0,T]}\Big[|h^1|+|h^2|+|y|\Big] \mathbf{P}^{\delta,\Delta}(dh^1dh^2dydt)\Bigg\}=0.$$
\end{proposition}

%In particular, in Section 3.1  we give some prior estimates for $\bar{q}^{\varepsilon}$ and $\bar{p}^{\varepsilon}$ which are crucial in the weak convergence analysis, then we present the proof for Regime 1.  Note that the approach is the same for all  regimes. Therefore, we present the proof in detail for Regime 1 and for Regime 2  only outline the differences  in Section 3.2.
\begin{proof}
(i) \textbf{Tightness  of $\{Z^{\delta,h^\delta}\}_{\delta>0}$:} For reader's convenience, we recall
\begin{eqnarray}\label{zz1}
Z_t^{\delta,h^\delta}=\!\!\!\!\!\!\!\!&&~~~\frac{1}{\lambda(\delta)}\int_0^tb\big(\lambda(\delta)Z_s^{\delta,h^\delta} +\bar{X}_s ,\mathcal{L}_{X^{\delta}_s},Y^{\delta,h^\delta}_s\big)-\bar{b}\big(\lambda(\delta)Z_s^{\delta,h^\delta} +\bar{X}_s ,\mathcal{L}_{X^{\delta}_s}\big)ds
\nonumber\\
\!\!\!\!\!\!\!\!&&~~~
+\frac{1}{\lambda(\delta)}\int_0^t\bar{b}\big(\lambda(\delta)Z_s^{\delta,h^\delta} +\bar{X}_s ,\mathcal{L}_{X^{\delta}_s}\big)-\bar{b}(\bar{X}_s,\mathcal{L}_{X^{\delta}_s})ds
\nonumber\\
\!\!\!\!\!\!\!\!&&~~~+\frac{1}{\lambda(\delta)}\int_0^t\bar{b}(\bar{X}_s,\mathcal{L}_{X^{\delta}_s})-\bar{b}(\bar{X}_s,\mathcal{L}_{\bar{X}_s})ds
\nonumber\\
\!\!\!\!\!\!\!\!&&~~~+\frac{\sqrt{\delta}}{\lambda(\delta)}\int_0^t\sigma\big(\lambda(\delta)Z_s^{\delta,h^\delta} +\bar{X}_s ,\mathcal{L}_{X^{\delta}_s}\big)dW_s^1
\nonumber\\
\!\!\!\!\!\!\!\!&&~~~+\int_0^t\sigma\big(\lambda(\delta)Z_s^{\delta,h^\delta} +\bar{X}_s ,\mathcal{L}_{X^{\delta}_s}\big)h^{1,\delta}_s ds,
\nonumber\\
=:\!\!\!\!\!\!\!\!&&~~~\sum_{i=1}^5\mathcal{K}_i^{\delta}(t).
\end{eqnarray}
Recall the uniform moment estimate (\ref{ess16}). In light of the criterion of tightness (cf.~\cite[Theorem 7.3]{B0}), it suffices to prove that for any positive $\theta,\epsilon$, there exist constants $\delta_0$ such that
\begin{equation}\label{tight1}
\sup_{\delta\in(0,1)}\mathbb{P}\Big(\sup_{t_1,t_2\in[0,T], |t_1-t_2|<\delta_0}|Z^{\delta,h^\delta}_{t_1}-Z^{\delta,h^\delta}_{t_2}|\geq \theta\Big)\leq \epsilon.
\end{equation}
Note that in light of (\ref{esa2}) and (\ref{esa3}) below, we know that the terms
$\mathcal{K}_3^{\delta}(t)$ and $\mathcal{K}_4^{\delta}(t)$ vanish in probability in $C([0,T];\mathbb{R}^n)$ as $\delta\to 0$.  Hence, it is sufficient to show (\ref{tight1}) holds for the remaining terms on the right hand side of (\ref{zz1}). Specifically, as for $\mathcal{K}_2^{\delta}(t)$, {by Lemma \ref{C1}} and (\ref{ess16}), for any $t_1,t_2\in[0,T]$,
\begin{eqnarray}\label{tt1}
\mathbb{E}|\mathcal{K}_2^{\delta}(t_1)-\mathcal{K}_2^{\delta}(t_2)|^2
\leq~~\!\!\!\!\!\!\!\!&&\mathbb{E}\Bigg(\int_{t_2}^{t_1}|Z_s^{\delta,h^\delta}|ds\Bigg)^2\leq C_T|t_1-t_2|^2,
\end{eqnarray}
which implies (\ref{tight1}) holds by the Kolmogorov's continuity criterion. As for $\mathcal{K}_5^{\delta}(t)$, due to $(\mathbf{A3})$ and the fact that $h^\delta\in\mathcal{A}_M$,
\begin{eqnarray}\label{tt2}
\mathbb{E}|\mathcal{K}_5^{\delta}(t_1)-\mathcal{K}_5^{\delta}(t_2)|^4
\leq~~\!\!\!\!\!\!\!\!&&\mathbb{E}\Bigg(\int_{t_2}^{t_1}\|\sigma\big(\lambda(\delta)Z_s^{\delta,h^\delta} +\bar{X}_s ,\mathcal{L}_{X^{\delta}_s}\big)\|\cdot|h^{1,\delta}_s|ds\Bigg)^4
\nonumber\\
\leq~~\!\!\!\!\!\!\!\!&& \Big(\mathbb{E}\big[\sup_{t\in[0,T]}|X^{\delta}_t|^2\big]\Big)^2\mathbb{E}\Bigg(\int_{t_2}^{t_1}|h^{1,\delta}_s|ds\Bigg)^4
\nonumber\\
\leq~~\!\!\!\!\!\!\!\!&& C_T|t_1-t_2|^2\mathbb{E}\Bigg(\int_{0}^{T}|h^{1,\delta}_s|^2ds\Bigg)^2
\nonumber\\
\leq~~\!\!\!\!\!\!\!\!&& C_{M,T}|t_1-t_2|^2.
\end{eqnarray}
{As for $\mathcal{K}_1^{\delta}(t)$, by Proposition \ref{P3.6} and (\ref{Y11}), we know
$$\mathbb{E}\Big[\sup_{t\in [0, T]}|\text{I}_{11}^{\delta}(t)|^2\Big]\leq C_{M,T}(1+|y|^2)(\frac{\varepsilon}{\lambda(\delta)^2}+\frac{\varepsilon^2}{\delta}),$$
which implies  that  $I^{\delta}_{11}(t)$    vanishes in probability in $C([0,T];\mathbb{R}^n)$, as $\delta\to 0$. }
 Thus in light of  the proof of Proposition \ref{p5}, we only need to show the following term
$$\text{I}_3^{\delta}(t):=\frac{\sqrt{\varepsilon}}{\sqrt{\delta}}\int^t_0\partial_y \Phi(X_s^{\delta,h^\delta} ,\mathcal{L}_{X^{\delta}_s},Y^{\delta,h^\delta}_s)\cdot g(X^{\delta,h^\delta}_s,\mathcal{L}_{X^{\delta}_s},Y^{\delta,h^\delta}_s)h^{2,\delta}_s ds$$
satisfies (\ref{tight1}). Indeed, by $(\mathbf{A3})$ and (\ref{PHI3}), we can obtain
\begin{eqnarray*}
\mathbb{E}|\text{I}_3^{\delta}(t_1)-\text{I}_3^{\delta}(t_2)|^4\leq~~\!\!\!\!\!\!\!\!&&\frac{C\varepsilon^4}{\delta^4}|t_1-t_2|^2\mathbb{E}\Bigg(\int_{t_2}^{t_1}|h^\delta_s|^2ds\Bigg)^2
\nonumber\\
\leq~~\!\!\!\!\!\!\!\!&& \frac{C_{M,T}\varepsilon^4}{\delta^4}|t_1-t_2|^2,
\end{eqnarray*}
then according to the regime $\varepsilon=O(\delta)$, it follows that (\ref{tight1}) holds.

\vspace{2mm}
(ii) \textbf{Tightness  of  $\{\mathbf{P}^{\delta,\Delta}\}_{\delta>0}$:}
 We first claim that the function
$$\Psi(\Pi):=\int_{\mathcal{Z}_1\times\mathcal{Z}_2\times\mathcal{Y}\times[0,T]}\Big[|h^1|^2+|h^2|^2+|y|^2\Big]\Pi(dh^1dh^2dydt), ~~\Pi\in\mathcal{P}(\mathcal{Z}_1\times\mathcal{Z}_2\times\mathcal{Y}\times[0,T])$$
is a tightness function  by the fact that it is nonnegative and that the level set  $$\mathcal{R}_k:=\Big\{\Pi\in\mathcal{P}(\mathcal{Z}_1\times\mathcal{Z}_2\times\mathcal{Y}\times[0,T]):\Psi(\Pi)\leq k\Big\}$$
is relatively compact in $\mathcal{P}(\mathcal{Z}_1\times\mathcal{Z}_2\times\mathcal{Y}\times[0,T])$, for each $k<\infty$. To prove the relative compactness, we observe that by Chebyshev's inequality
$$\sup_{\Pi\in \mathcal{R}_k}\Pi\Big(\big\{(h^1,h^2,y,t)\in\mathbb{U}_N\times[0,T]\big\}\Big)\leq\sup_{\Pi\in \mathcal{R}_k}\frac{\Psi(\Pi)}{N^2}\leq\frac{k}{N^2}.$$
Hence, $\mathcal{R}_k$ is tight and thus relatively compact as a subset of $\mathcal{P}(\mathcal{Z}_1\times\mathcal{Z}_2\times\mathcal{Y}\times[0,T])$.

Since $\Psi$ is a tightness function, by Theorem A.3.17 of \cite{de}, the tightness of $\{\mathbf{P}^{\delta,\Delta}\}_{\delta>0}$ holds if we can prove that $\sup_{\delta>0}\mathbb{E}\Big[\Psi(\mathbf{P}^{\delta,\Delta})\Big]<\infty$.
Indeed,
\begin{align}\label{L2es}
  \sup_{\delta>0}\mathbb{E}\Big[\Psi(\mathbf{P}^{\delta,\Delta})\Big] & = \sup_{\delta>0}\mathbb{E}\int_{\mathcal{Z}_1\times\mathcal{Z}_2\times\mathcal{Y}\times[0,T]}\Big[|h^1|^2+|h^2|^2+|y|^2\Big]\mathbf{P}^{\delta,\Delta}(dh^1dh^2dydt)\nonumber\\
  & = \sup_{\delta>0}\mathbb{E}\int_0^T\frac{1}{\Delta}\int_t^{t+\Delta}\Big[|h^{1,\delta}_s|^2+|h^{2,\delta}_s|^2+|Y^{\delta,h^\delta}_s|^2\Big]dsdt\nonumber\\
   &\leq C\sup_{\delta>0}\mathbb{E}\int_0^{T+\Delta}\Big[|h^{1,\delta}_s|^2+|h^{2,\delta}_s|^2+|Y^{\delta,h^\delta}_s|^2\Big]ds<\infty.
\end{align}

(ii) \textbf{Uniform integrability of  $\{\mathbf{P}^{\delta,\Delta}\}_{\delta>0}$:} This statement follows from the last display and the following observation
\begin{align*}
 &\mathbb{E}\left[\int_{\mathbb{U}_N\times[0,T]}\Big[|h^1|+|h^2|+|y|\Big] \mathbf{P}^{\delta,\Delta}(dh^1dh^2dydt)\right] \\
  \leq& \frac{C}{N}\mathbb{E}\left[\int_{\mathcal{Z}_1\times\mathcal{Z}_2\times\mathcal{Y}\times[0,T]}\Big[|h^1|^2+|h^2|^2+|y|^2\Big]  \mathbf{P}^{\delta,\Delta}(dh^1dh^2dydt)\right].
\end{align*}

The proof is complete.
\end{proof}

\vspace{3mm}
Under Proposition \ref{prj1}, for any subsequence of $\{(Z^{\delta,h^\delta},\mathbf{P}^{\delta,\Delta})\}_{\delta>0}$, there exists a subsequence (still denoted  by $(Z^{\delta,h^\delta},\mathbf{P}^{\delta,\Delta})$) such that
$$(Z^{\delta,h^\delta},\mathbf{P}^{\delta,\Delta})\Rightarrow(\bar{Z},\bar{\mathbf{P}}),~\delta\to 0,$$
where we used ``$\Rightarrow$'' to denote the weak convergence of random variables throughout this section. Applying the Skorokhod representation theorem, it is possible to construct another probability space and random variables (still denoted by $(Z^{\delta,h^\delta},\mathbf{P}^{\delta,\Delta})$) such that
\begin{equation}\label{con3}
(Z^{\delta,h^\delta},\mathbf{P}^{\delta,\Delta})\to(\bar{Z},\bar{\mathbf{P}}),~\mathbb{P}\text{-a.s.},~\delta\to 0.
\end{equation}
Therefore, our next aim is to prove that
the accumulation point $(\bar{Z},\bar{\mathbf{P}})$ is a viable pair w.r.t~$(\Theta,\nu^{\bar{X},\mathcal{L}_{\bar{X}}})$ in the sense of Definition \ref{defj1}, i.e. $(\bar{Z},\bar{\mathbf{P}})\in\mathcal{V}_{(\Theta,\nu^{\bar{X},\mathcal{L}_{\bar{X}}})}$.

\subsection{Existence of viable pair}\label{sec5.2}
%We give the following Lemma that is similar to Lemma 3.2 from \cite{ds} and Lemma 5.1 from \cite{ms}.

%\begin{lemma}
% Let $\tau_1,\tau_2$ be positive numbers and $0<\tau_1<\tau_2\leq T$. Consider a continuous function $\xi:\mathbb{R}^d\times\mathbb{R}^d\times\mathcal{P}(\mathbb{R}^d)\times\mathcal{Y}\times \mathcal{Z}\times\mathcal{Z}\rightarrow\mathbb{R}$ that is  affine in the first argument, bounded in the second and third arguments, not growing faster than $|y|$ in the forth argument   and affine in the last two arguments. Assume that $(Z^{\delta,h^\delta},\mathbf{P}^{\delta,\Delta})\rightarrow(\bar{Z}_i,\bar{P}_i)$ in distribution for some subsequence of $\delta\rightarrow0$, and  $h^\delta(s)\in\mathcal{A}^N$. Then the following limits are valid in distribution along this subsequence:
%\begin{equation}\label{j11}
 % \int_{\mathcal{Z}\times\mathcal{Z}\times\mathcal{Y}\times[\tau_1,\tau_2]}\xi(Z^{\delta,h^\delta}_t,\bar{X}_t,\mathcal{L}_{\bar{X}_t},y,h^1,h^2)\mathbf{P}^{\delta,\Delta}(dh^1dh^2dydt)\rightarrow\int_{\mathcal{Z}\times\mathcal{Z}\times\mathcal{Y}\times[\tau_1,\tau_2]}\xi(\bar{Z}_i,\bar{X}_t,\mathcal{L}_{\bar{X}_t},y,h^1,h^2)\bar{P}_i(dh^1dh^2dydt)
%\end{equation}
%and
%\begin{equation}\label{j12}
% \int_{\tau_1}^{\tau_2}\xi(Z^{\delta,h^\delta}_t,\bar{X}_t,\mathcal{L}_{\bar{X}_t},Y^{\delta,h^\delta}_t,h^{1,\delta},h^{2,\delta})dt-\int_{\mathcal{Z}\times\mathcal{Z}\times\mathcal{Y}\times[\tau_1,\tau_2]}\xi(\bar{Z}_i,\bar{X}_t,\mathcal{L}_{\bar{X}_t},y,h^1,h^2)\mathbf{P}^{\delta,\Delta}(dh^1dh^2dydt)\rightarrow0.
%\end{equation}
%\end{lemma}

Note that by Fatou's Lemma and (\ref{L2es}), we have
\begin{align*}
 &\mathbb{E}\int_{\mathcal{Z}_1\times\mathcal{Z}_2\times\mathcal{Y}\times[0,T]}\Big[|h^1|^2 +|h^2|^2+|y|^2 \Big]\bar{\mathbf{P}}(dh^1dh^2dydt)\\
 \leq& \liminf_{\delta\to 0} \mathbb{E}\int_{\mathcal{Z}_1\times\mathcal{Z}_2\times\mathcal{Y}\times[0,T]}\Big[|h^1|^2+|h^2|^2 +|y|^2 \Big] \mathbf{P}^{\delta,\Delta}(dh^1dh^2dydt)\\
 \leq&\sup_{\delta>0}\mathbb{E}\Big[\Psi(\mathbf{P}^{\delta,\Delta})\Big]<\infty,
\end{align*}
which then implies that
$$\int_{\mathcal{Z}_1\times\mathcal{Z}_2\times\mathcal{Y}\times[0,T]}\Big[|h^1|^2 +|h^2|^2+|y|^2 \Big] \bar{\mathbf{P}}(dh^1dh^2dydt)<\infty,~\mathbb{P}\text{-a.s.}.$$
Therefore, the first statement in Definition \ref{defj1} holds.  It remains to show that $(\bar{Z},\bar{\mathbf{P}})$ satisfy (\ref{j1})-(\ref{j2}).

\vspace{2mm}
\textbf{Proof of (\ref{j3}):} We first prove (\ref{j3}) which is helpful in proving (\ref{j1}). By the fact that $\mathbf{P}^{\delta,\Delta}(\mathcal{Z}_1\times\mathcal{Z}_2\times\mathcal{Y}\times[0,t])=t$, along with $\bar{\mathbf{P}}(\mathcal{Z}_1\times\mathcal{Z}_2\times\mathcal{Y}\times \{t\})=0$ and the continuity of $t\mapsto\bar{\mathbf{P}}(\mathcal{Z}_1\times\mathcal{Z}_2\times\mathcal{Y}\times[0,t])$ to deal with null sets, we can get (\ref{j3}) holds.
\hspace{\fill}$\Box$

\vspace{2mm}
{\textbf{Proof of (\ref{j1}):}} Recall (\ref{zz1}).
We next deal with the convergence of the terms $\mathcal{K}_i^{\delta}(t)$, $i=1,\ldots,5$.
Notice that in light of (\ref{esa2}) and (\ref{esa3}) below, we know that the terms
$\mathcal{K}_3^{\delta}(t)$ and $\mathcal{K}_4^{\delta}(t)$ vanish in probability in $C([0,T];\mathbb{R}^n)$, as $\delta\to 0$. Thus it is sufficient to show the convergence of the remaining terms. We will separate the proof into following Lemmas \ref{lemcon1}-\ref{lemcon3}.

\begin{lemma}\label{lemcon1}
The following limit is valid with probability $1$:
\begin{equation}\label{prop1}
\lim_{\delta\rightarrow0}\sup_{t\in[0,T]}\left|\mathcal{K}_5^{\delta}(t)-\int_{\mathcal{Z}_1\times\mathcal{Z}_2\times\mathcal{Y}\times[0,t]}\sigma(\bar{X}_s,\mathcal{L}_{\bar{X}_s})h^1\bar{\mathbf{P}}(dh^1dh^2dyds)\right|=0.
\end{equation}
\end{lemma}
\begin{proof}
Note that
\begin{eqnarray}\label{eA5}
\mathcal{K}_5^{\delta}(t)=~~\!\!\!\!\!\!\!\!&&\Bigg(\int_0^t\sigma\big(\lambda(\delta)Z_s^{\delta,h^\delta}+\bar{X}_s,\mathcal{L}_{X^{\delta}_s}\big){h}^{1,\delta}_sds
\nonumber \\
\!\!\!\!\!\!\!\!&&
-\int_{\mathcal{Z}_1\times\mathcal{Z}_2\times\mathcal{Y}\times [0,t]}\sigma\big(\lambda(\delta)Z_s^{\delta,h^\delta}+\bar{X}_s,\mathcal{L}_{X^{\delta}_s}\big){h}^1\mathbf{P}^{\delta,\Delta}(dh^1dh^2dyds)\Bigg)
\nonumber \\
\!\!\!\!\!\!\!\!&&+\Bigg(\int_{\mathcal{Z}_1\times\mathcal{Z}_2\times\mathcal{Y}\times [0,t]}\sigma\big(\lambda(\delta)Z_s^{\delta,h^\delta}+\bar{X}_s,\mathcal{L}_{X^{\delta}_s}\big){h}^1\mathbf{P}^{\delta,\Delta}(dh^1dh^2dyds)
\nonumber \\
\!\!\!\!\!\!\!\!&&-\int_{\mathcal{Z}_1\times\mathcal{Z}_2\times\mathcal{Y}\times [0,t]}\sigma\big(\bar{X}_s,\mathcal{L}_{\bar{X}_s}\big){h}^1\mathbf{P}^{\delta,\Delta}(dh^1dh^2dyds)\Bigg)
\nonumber \\
\!\!\!\!\!\!\!\!&&+\int_{\mathcal{Z}_1\times\mathcal{Z}_2\times\mathcal{Y}\times [0,t]}\sigma\big(\bar{X}_s,\mathcal{L}_{\bar{X}_s}\big){h}^1\mathbf{P}^{\delta,\Delta}(dh^1dh^2dyds)
\nonumber \\
=:~~\!\!\!\!\!\!\!\!&&\sum_{i=1}^{3}\mathcal{K}_{5i}^{\delta}(t).
\end{eqnarray}

\textbf{Step 1:} In this step, we estimate $\mathcal{K}_{51}^{\delta}(t)$. First, we observe that
\begin{eqnarray*}
\!\!\!\!\!\!\!\!&&\int_{\mathcal{Z}_1\times\mathcal{Z}_2\times\mathcal{Y}\times [0,t]}\sigma\big(\lambda(\delta)Z_s^{\delta,h^\delta}+\bar{X}_s,\mathcal{L}_{X^{\delta}_s}\big){h}^1\mathbf{P}^{\delta,\Delta}(dh^1dh^2dyds)
\nonumber \\
=~~\!\!\!\!\!\!\!\!&&\Bigg(\int_0^t\frac{1}{\Delta}\int_s^{s+\Delta}\sigma\big(\lambda(\delta)Z_s^{\delta,h^\delta}+\bar{X}_s,\mathcal{L}_{X^{\delta}_s}\big){h}^{1,\delta}_rdrds
\nonumber \\
\!\!\!\!\!\!\!\!&&-\int_0^t\frac{1}{\Delta}\int_s^{s+\Delta}\sigma\big(\lambda(\delta)Z_r^{\delta,h^\delta}+\bar{X}_r,\mathcal{L}_{X^{\delta}_r}\big){h}^{1,\delta}_rdrds\Bigg)
\nonumber \\
\!\!\!\!\!\!\!\!&&+\int_0^t\frac{1}{\Delta}\int_s^{s+\Delta}\sigma\big(\lambda(\delta)Z_r^{\delta,h^\delta}+\bar{X}_r,\mathcal{L}_{X^{\delta}_r}\big){h}^{1,\delta}_rdrds.
\nonumber \\
=:~~\!\!\!\!\!\!\!\!&&\mathcal{H}_1(t)+\mathcal{H}_2(t).
\end{eqnarray*}
As for $\mathcal{H}_1(t)$, due to (\ref{Holder}) we can deduce that
\begin{eqnarray}\label{oc1}
\mathbb{E}\sup_{t\in[0,T]}|\mathcal{H}_1(t)|^2{\leq}~~\!\!\!\!\!\!\!\!&&\frac{1}{\Delta^2}\mathbb{E}\Bigg(\int_0^T\int_s^{s+\Delta}\big(|X_s^{\delta,h^\delta}-X_r^{\delta,h^\delta}|^2+\mathbb{E}|X_s^{\delta}-X_r^{\delta}|^2\big)drds
\nonumber \\
\!\!\!\!\!\!\!\!&&
\cdot\int_0^T\int_s^{s+\Delta}|{h}^{1,\delta}_r|^2drds\Bigg)
\nonumber \\
\leq~~\!\!\!\!\!\!\!\!&&\frac{C_{M,T}}{\Delta}\int_0^T\int_s^{s+\Delta}\big(\mathbb{E}|X_s^{\delta,h^\delta}-X_r^{\delta,h^\delta}|^2+\mathbb{E}|X_s^{\delta}-X_r^{\delta}|^2\big)drds
\nonumber \\
\leq~~\!\!\!\!\!\!\!\!&&\frac{C_{M,T}}{\Delta}\int_0^T\int_s^{s+\Delta}(r-s)drds
\nonumber \\
\leq~~\!\!\!\!\!\!\!\!&&C_{M,T}\Delta\to 0,~~\text{as}~\delta\to 0,
\end{eqnarray}
where the second step is due to for any $t\in[0,T]$,
\begin{eqnarray*}
\int_0^t\int_s^{s+\Delta}|{h}^{1,\delta}_r|^2drds=~~\!\!\!\!\!\!\!\!&&\int_0^{\Delta}\int_0^r|{h}^{1,\delta}_r|^2dsdr+\int_{\Delta}^t\int_{r-\Delta}^r|{h}^{1,\delta}_r|^2dsdr
\nonumber \\
\!\!\!\!\!\!\!\!&&+\int_t^{t+\Delta}\int_{r-\Delta}^t|{h}^{1,\delta}_r|^2dsdr
\nonumber \\
=~~\!\!\!\!\!\!\!\!&&\int_0^{\Delta}|{h}^{1,\delta}_r|^2rdr+\Delta\int_{\Delta}^t|{h}^{1,\delta}_r|^2dr+{\int_t^{t+\Delta}(t-r+\Delta)|{h}^{1,\delta}_r|^2dr}
\nonumber \\
\leq~~\!\!\!\!\!\!\!\!&&{3\Delta\int_{0}^t|{h}^{1,\delta}_r|^2dr.}
\end{eqnarray*}

On the other hand,  note that $\mathcal{H}_2(t)$ have the composition
\begin{eqnarray*}
\mathcal{H}_2(t)=~~\!\!\!\!\!\!\!\!&&\int_0^{\Delta}\frac{1}{\Delta}\int_0^{r}\sigma\big(\lambda(\delta)Z_r^{\delta,h^\delta}+\bar{X}_r,\mathcal{L}_{X^{\delta}_r}\big){h}^{1,\delta}_rdsdr
\nonumber \\
\!\!\!\!\!\!\!\!&&+\int_{\Delta}^t\frac{1}{\Delta}\int_{r-\Delta}^{r}\sigma\big(\lambda(\delta)Z_r^{\delta,h^\delta}+\bar{X}_r,\mathcal{L}_{X^{\delta}_r}\big){h}^{1,\delta}_rdsdr
\nonumber \\
\!\!\!\!\!\!\!\!&&+\int_{t}^{t+\Delta}\frac{1}{\Delta}\int_{r-\Delta}^{t}\sigma\big(\lambda(\delta)Z_r^{\delta,h^\delta}+\bar{X}_r,\mathcal{L}_{X^{\delta}_r}\big){h}^{1,\delta}_rdsdr.
\nonumber \\
=:~~\!\!\!\!\!\!\!\!&&\sum_{i=1}^3\mathcal{H}_{2i}(t).
\end{eqnarray*}
In light of $(\mathbf{A4})$, it is easy to see that
\begin{equation}\label{oc2}
\mathbb{E}\sup_{t\in[0,T]}|\mathcal{H}_{21}(t)|^2+\mathbb{E}\sup_{t\in[0,T]}|\mathcal{H}_{23}(t)|^2\to 0,~~\text{as}~\delta\to 0.
\end{equation}
Therefore, for the term $\mathcal{K}_{51}^\delta(t)$ in (\ref{eA5})  it follows that
\begin{eqnarray*}
\mathbb{E}\sup_{t\in[0,T]}|\mathcal{K}_{51}^{\delta}(t)|^2\leq~~\!\!\!\!\!\!\!\!&&C\mathbb{E}\Big|\int_0^\Delta\sigma\big(\lambda(\delta)Z_r^{\delta,h^\delta}+\bar{X}_r,\mathcal{L}_{X^{\delta}_r}\big){h}^{1,\delta}_rdr\Big|^2
\nonumber \\
\!\!\!\!\!\!\!\!&&+C\mathbb{E}\sup_{t\in[0,T]}\Big|\int_\Delta^t\sigma\big(\lambda(\delta)Z_r^{\delta,h^\delta}+\bar{X}_r,\mathcal{L}_{X^{\delta}_r}\big){h}^{1,\delta}_rdr-{\mathcal{H}_{22}(t)}\Big|^2
\nonumber \\
\!\!\!\!\!\!\!\!&&+C\mathbb{E}\sup_{t\in[0,T]}|\mathcal{H}_1(t)|^2+C\mathbb{E}\sup_{t\in[0,T]}|\mathcal{H}_{21}(t)|^2+C\mathbb{E}\sup_{t\in[0,T]}|\mathcal{H}_{23}(t)|^2.
\end{eqnarray*}
Then by $(\mathbf{A4})$, (\ref{oc1}), (\ref{oc2}) and the definition of $\mathcal{H}_{22}(t)$, it leads to
\begin{equation}\label{eA51}
\mathbb{E}\sup_{t\in[0,T]}|\mathcal{K}_{51}^{\delta}(t)|^2\to 0,~~~\text{as}~\delta\to 0.
\end{equation}

\textbf{Step 2:} In this step, we focus on $\mathcal{K}_{52}^{\delta}(t)$ and $\mathcal{K}_{53}^{\delta}(t)$. As for $\mathcal{K}_{52}^{\delta}(t)$,
\begin{eqnarray*}
\!\!\!\!\!\!\!\!&&\mathbb{E}\sup_{t\in[0,T]}|\mathcal{K}_{52}^{\delta}(t)|^2
\nonumber \\
\leq~~\!\!\!\!\!\!\!\!&&\mathbb{E}\sup_{t\in[0,T]}\Big|\int_{\mathcal{Z}_1\times\mathcal{Z}_2\times\mathcal{Y}\times [0,t]}\Big(\sigma\big(\lambda(\delta)Z_s^{\delta,h^\delta}+\bar{X}_s,\mathcal{L}_{X^{\delta}_s}\big)-\sigma\big(\bar{X}_s,\mathcal{L}_{\bar{X}_s}\big)\Big){h}^1\mathbf{P}^{\delta,\Delta}(dh^1dh^2dyds)\Big|^2
\nonumber \\
\leq~~\!\!\!\!\!\!\!\!&&\mathbb{E}\Bigg(\int_{\mathcal{Z}_1\times\mathcal{Z}_2\times\mathcal{Y}\times [0,T]}\|\sigma\big(\lambda(\delta)Z_s^{\delta,h^\delta}+\bar{X}_s,\mathcal{L}_{X^{\delta}_s}\big)-\sigma\big(\bar{X}_s,\mathcal{L}_{\bar{X}_s}\big)\|^2\mathbf{P}^{\delta,\Delta}(dh^1dh^2dyds)
\nonumber \\
\!\!\!\!\!\!\!\!&&\cdot\int_{\mathcal{Z}_1\times\mathcal{Z}_2\times\mathcal{Y}\times [0,T]}    |{h}^1|^2\mathbf{P}^{\delta,\Delta}(dh^1dh^2dyds)\Bigg)
\nonumber \\
\leq~~\!\!\!\!\!\!\!\!&&C_{M,T}\lambda(\delta)^2\mathbb{E}\Big[\sup_{s\in[0,T]}|Z_s^{\delta,h^\delta}|^2\Big]+C_{M,T}\mathbb{E}\Big[\sup_{s\in[0,T]}|X^{\delta}_s-\bar{X}_s|^2\Big].
\end{eqnarray*}
According to (\ref{33}) and (\ref{ess16}), we deduce
\begin{equation}\label{eA52}
\mathbb{E}\sup_{t\in[0,T]}|\mathcal{K}_{52}^{\delta}(t)|^2\to 0,~~~\text{as}~\delta\to 0.
\end{equation}
Now we turn to study the limit of  $\mathcal{K}_{53}^{\delta}(t)$. Recall that as stated in Subsection \ref{sec5.1}, due to the use of Skorohod's representation theorem, we have construct a sequence $\{\mathbf{P}^{\delta,\Delta}\}_{\delta>0}$, as random variable in $\mathcal{P}(\mathcal{Z}_1\times\mathcal{Z}_2\times\mathcal{Y}\times[0,T])$, converge almost surely to $\bar{\mathbf{P}}$. Notice that $\{\mathbf{P}^{\delta,\Delta}\}_{\delta>0}$ is uniformly integrable by Theorem \ref{prj1}, thus we can get that $\mathbf{P}^{\delta,\Delta}$, as random variable in $\mathcal{P}_1(\mathcal{Z}_1\times\mathcal{Z}_2\times\mathcal{Y}\times[0,T])$, converges almost surely to $\bar{\mathbf{P}}$, i.e. $\mathbb{P}$-a.s.
$$\mathbb{W}_1(\mathbf{P}^{\delta,\Delta},\bar{\mathbf{P}})\to 0,~~~\text{as}~\delta\to 0.$$
Hence, due to the assumption $(\mathbf{A4})$, we obtain the following convergence with probability $1$,
\begin{eqnarray}\label{eA53}
\!\!\!\!\!\!\!\!&&\int_{\mathcal{Z}_1\times\mathcal{Z}_2\times\mathcal{Y}\times[0,T]}\sigma(\bar{X}_s,\mathcal{L}_{\bar{X}_s})h^1\mathbf{P}^{\delta,\Delta}(dh^1dh^2dyds)
\nonumber \\
\!\!\!\!\!\!\!\!&&\to\int_{\mathcal{Z}_1\times\mathcal{Z}_2\times\mathcal{Y}\times[0,T]}\sigma(\bar{X}_s,\mathcal{L}_{\bar{X}_s})h^1\bar{\mathbf{P}}(dh^1dh^2dyds),~\text{as}~\delta\to 0.
\end{eqnarray}
%for any coupling $\pi\in\mathcal{C}_{\mathbf{P}^{\delta,\Delta},\bar{\mathbf{P}}}$,
%\begin{eqnarray*}
%\!\!\!\!\!\!\!\!&&\Big|\int_{\mathcal{Z}\times\mathcal{Z}\times\mathcal{Y}\times[0,T]}\sigma(\bar{X}_s,\mathcal{L}_{\bar{X}_s})h^1\mathbf{P}^{\delta,\Delta}(dh^1dh^2dyds)-\int_{\mathcal{Z}\times\mathcal{Z}\times\mathcal{Y}\times[0,T]}\sigma(\bar{X}_s,\mathcal{L}_{\bar{X}_s})h^1\bar{\mathbf{P}}(dh^1dh^2dyds)\Big|
%\nonumber \\
%=~~\!\!\!\!\!\!\!\!&&\Big|\int_{\mathcal{Z}\times\mathcal{Z}\times\mathcal{Y}\times[0,T]}\sigma(\bar{X}_s,\mathcal{L}_{\bar{X}_s})h^1\mathbf{P}^{\delta,\Delta}(dh^1dh^2dyds)\Big|
%\end{eqnarray*}
Finally, combining (\ref{eA51})-(\ref{eA53}), it follows that (\ref{prop1}) holds. We complete the proof.
\end{proof}

\begin{lemma}\label{lemcon2}
The following limit is valid with probability $1$:
\begin{equation}\label{prop2}
\lim_{\delta\rightarrow0}\sup_{t\in[0,T]}\left|\mathcal{K}_1^\delta(t)
-\int_{\mathcal{Z}_1\times\mathcal{Z}_2\times\mathcal{Y}\times[0,t]}\sqrt{\gamma}\partial_y\Phi_ g(\bar{X}_s,\mathcal{L}_{\bar{X}_s},y)h^2\bar{\mathbf{P}}(dh^1dh^2dyds)\right|=0,
\end{equation}
where $\partial_y\Phi_ g$ is defined by (\ref{partial Phig}).
\end{lemma}

\begin{proof}
Recall
\begin{eqnarray*}
\mathcal{K}_1^\delta(t)=~~\!\!\!\!\!\!\!\!&&\frac{\varepsilon}{\lambda(\delta)}\Big\{-\Phi(X_t^{\delta,h^\delta} ,\mathcal{L}_{X^{\delta}_t},Y^{\delta,h^\delta}_t)+\Phi(x,\delta_x,y)
\nonumber\\
\!\!\!\!\!\!\!\!&&+\int^t_0 \mathbb{E}\left[b(X^{\delta}_s,\mathcal{L}_{ X^{\delta}_{s}}, Y^{\delta}_s)\partial_{\mu}\Phi(x,\mathcal{L}_{X^{\delta}_{s}},y)(X^{\delta}_s)\right]\Big|_{x=X_s^{\delta,h^\delta} ,y=Y^{\delta,h^\delta}_{s}}ds\nonumber\\
\!\!\!\!\!\!\!\!&&+\int^t_0 {\frac{\delta}{2}}\mathbb{E}\text{Tr}\left[\sigma\sigma^{*}(X^{\delta}_s,\mathcal{L}_{ X^{\delta}_{s}})\partial_z\partial_{\mu}\Phi(x,\mathcal{L}_{X^{\delta}_{s}},y)(X^{\delta}_s)\right]\Big|_{x=X_s^{\delta,h^\delta} ,y=Y^{\delta,h^\delta}_{s}}ds\nonumber\\
\!\!\!\!\!\!\!\!&&+\int^t_0 \mathbf{L}^{1,\delta}_{\mathcal{L}_{X^{\delta}_{s}},Y^{\delta,h^\delta}_{s}}\Phi(X_s^{\delta,h^\delta},\mathcal{L}_{X^{\delta}_{s}},Y^{\delta,h^\delta}_{s})ds
+M^{1,\delta,h^\delta}_t\Big\}
\nonumber\\
\!\!\!\!\!\!\!\!&&+\varepsilon\int^t_0 \partial_x \Phi(X_s^{\delta,h^\delta} ,\mathcal{L}_{X^{\delta}_s},Y^{\delta,h^\delta}_s)\cdot \sigma(X^{\delta,h^\delta}_s,\mathcal{L}_{X^{\delta}_s}){h}^{1,\delta}_s ds
\nonumber\\
\!\!\!\!\!\!\!\!&&+\frac{\sqrt{\varepsilon}}{\sqrt{\delta}}\int^t_0\partial_y \Phi(X_s^{\delta,h^\delta} ,\mathcal{L}_{X^{\delta}_s},Y^{\delta,h^\delta}_s)\cdot g(X^{\delta,h^\delta}_s,\mathcal{L}_{X^{\delta}_s},Y^{\delta,h^\delta}_s){h}^{2,\delta}_s ds
\nonumber\\
\!\!\!\!\!\!\!\!&&
+\frac{\sqrt{\varepsilon}}{\lambda(\delta)}M^{2,\delta,h^\delta}_t.
\end{eqnarray*}
{Collecting the calculations of Steps 2 and 3 in Proposition \ref{p5}, (\ref{prop2}) holds if we can prove
\begin{eqnarray}\label{lem5.2}
\!\!\!\!\!\!\!\!&&~~~\lim_{\delta\rightarrow0}\sup_{t\in[0,T]}\Bigg|\frac{\sqrt{\varepsilon}}{\sqrt{\delta}}\int^t_0\partial_y \Phi_g(X_s^{\delta,h^\delta} ,\mathcal{L}_{X^{\delta}_s},Y^{\delta,h^\delta}_s){h}^{2,\delta}_s ds
\nonumber\\
\!\!\!\!\!\!\!\!&&
-\int_{\mathcal{Z}_1\times\mathcal{Z}_2\times\mathcal{Y}\times[0,t]}\sqrt{\gamma}\partial_y\Phi_ g(\bar{X}_s,\mathcal{L}_{\bar{X}_s},y)h^2\bar{\mathbf{P}}(dh^1dh^2dyds)\Bigg|=0.
\end{eqnarray}
We have the following  composition
\begin{eqnarray*}
\!\!\!\!\!\!\!\!&&\frac{\sqrt{\varepsilon}}{\sqrt{\delta}}\int^t_0\partial_y \Phi_g(X_s^{\delta,h^\delta} ,\mathcal{L}_{X^{\delta}_s},Y^{\delta,h^\delta}_s){h}^{2,\delta}_s ds-\int_{\mathcal{Z}_1\times\mathcal{Z}_2\times\mathcal{Y}\times[0,t]}\sqrt{\gamma}\partial_y\Phi_ g(\bar{X}_s,\mathcal{L}_{\bar{X}_s},y)h^2\bar{\mathbf{P}}(dh^1dh^2dyds)
\nonumber \\
=~~\!\!\!\!\!\!\!\!&&\frac{\sqrt{\varepsilon}}{\sqrt{\delta}}\Bigg(\int^t_0\partial_y \Phi_g(X_s^{\delta,h^\delta} ,\mathcal{L}_{X^{\delta}_s},Y^{\delta,h^\delta}_s){h}^{2,\delta}_s ds
\nonumber \\
\!\!\!\!\!\!\!\!&&
-\int_{\mathcal{Z}_1\times\mathcal{Z}_2\times\mathcal{Y}\times [0,t]}\partial_y \Phi_g(X_s^{\delta,h^\delta} ,\mathcal{L}_{X^{\delta}_s},y){h}^{2} \mathbf{P}^{\delta,\Delta}(dh^1dh^2dyds)  \Bigg)
\nonumber \\
\!\!\!\!\!\!\!\!&&
+\frac{\sqrt{\varepsilon}}{\sqrt{\delta}}\Bigg(\int_{\mathcal{Z}_1\times\mathcal{Z}_2\times\mathcal{Y}\times [0,t]}\partial_y \Phi_g(X_s^{\delta,h^\delta} ,\mathcal{L}_{X^{\delta}_s},y){h}^{2} \mathbf{P}^{\delta,\Delta}(dh^1dh^2dyds)
\nonumber \\
\!\!\!\!\!\!\!\!&&
-\int_{\mathcal{Z}_1\times\mathcal{Z}_2\times\mathcal{Y}\times [0,t]}\partial_y \Phi_g(\bar{X}_s ,\mathcal{L}_{\bar{X}_s},y){h}^{2} \mathbf{P}^{\delta,\Delta}(dh^1dh^2dyds) \Bigg)
\nonumber \\
\!\!\!\!\!\!\!\!&&+\Bigg(\frac{\sqrt{\varepsilon}}{\sqrt{\delta}}\int_{\mathcal{Z}_1\times\mathcal{Z}_2\times\mathcal{Y}\times [0,t]}\partial_y \Phi_g(\bar{X}_s ,\mathcal{L}_{\bar{X}_s},y){h}^{2} \mathbf{P}^{\delta,\Delta}(dh^1dh^2dyds)
\nonumber \\
\!\!\!\!\!\!\!\!&&-\int_{\mathcal{Z}_1\times\mathcal{Z}_2\times\mathcal{Y}\times[0,t]}\sqrt{\gamma}\partial_y\Phi_ g(\bar{X}_s,\mathcal{L}_{\bar{X}_s},y)h^2\bar{\mathbf{P}}(dh^1dh^2dyds)\Bigg)
\nonumber \\
=:~~\!\!\!\!\!\!\!\!&&\sum_{i=1}^{3}\mathcal{O}_{i}^{\delta}(t).
\end{eqnarray*}
Therefore, (\ref{lem5.2}) holds if we prove that
$$\mathbb{E}\Big[\sup_{t\in[0,T]}\Big|\sum_{i=1}^{3}\mathcal{O}_{i}^{\delta}(t)\Big|\Big]\to 0,~\text{as}~\delta\to 0.$$

First, we claim that $\partial_y\Phi_ g$ is bounded and joint Lipschitz continuous in $(x,\mu,y)$. Indeed, by Proposition \ref{P3.6} it implies that $\partial_y\Phi$ is bounded and Lipschitz continuous w.r.t.~$(x,\mu,y)$, thus by condition $({\mathbf{A}}{\mathbf{3}})$, it is easy to see that  $\partial_y\Phi_g$ is bounded. Moreover, by $({\mathbf{A}}{\mathbf{1}})$ and $({\mathbf{A}}{\mathbf{3}})$, for all $x_1,x_2\in\mathbb{R}^n, \mu_1,\mu_2\in \mathcal{P}_2, y_1,y_2\in\mathbb{R}^m$, we have
\begin{align*}
 &\|\partial_y\Phi(x_1,\mu_1,y_1)g(x_1,\mu_1,y_1)-\partial_y\Phi(x_2,\mu_2,y_2)g(x_2,\mu_2,y_2)\|\\
  %\leq& \|\partial_y\Phi(x_1,\mu_1,y_1)g(x_1,\mu_1,y_1)-\partial_y\Phi(x_2,\mu_2,y_2)g(x_1,\mu_1,y_1)\|\\
% &+\|\partial_y\Phi(x_2,\mu_2,y_2)g(x_1,\mu_1,y_1)-\partial_y\Phi(x_2,\mu_2,y_2)g(x_2,\mu_2,y_2)\| \\
 \leq& \|g(x_1,\mu_1,y_1)\|\cdot\|\partial_y\Phi(x_1,\mu_1,y_1)-\partial_y\Phi(x_2,\mu_2,y_2)\|+\|\partial_y\Phi(x_2,\mu_2,y_2)\|\cdot\|g(x_1,\mu_1,y_1)
 -g(x_2,\mu_2,y_2)\|\\
  \leq& C\Big(|x_1-x_2|+|y_1-y_2|+\mathbb{W}_2(\mu_1, \mu_2)\Big).
\end{align*}

Nextly, we deal with terms $\mathcal{O}_{i}^{\delta}(t), i=1,2,3$.
First, we estimate the term $\mathcal{O}^{\delta}_{1}(t)$. By (\ref{p1}), we observe that
\begin{align*}
 &\int_{\mathcal{Z}_1\times\mathcal{Z}_2\times\mathcal{Y}\times [0,t]}\partial_y \Phi_g(X_s^{\delta,h^\delta} ,\mathcal{L}_{X^{\delta}_s},y){h}^{2}
 \mathbf{P}^{\delta,\Delta}(dh^1dh^2dyds) \\
  =& \Bigg(\int_0^t\frac{1}{\Delta}\int_s^{s+\Delta}\partial_y \Phi_g(X_s^{\delta,h^\delta} ,\mathcal{L}_{X^{\delta}_s},Y_r^{\delta,h^\delta} ){h}^{2,\delta}_rdrds\\
 &-\int_0^t\frac{1}{\Delta}\int_s^{s+\Delta}\partial_y \Phi_g(X_r^{\delta,h^\delta} ,\mathcal{L}_{X^{\delta}_r},Y_r^{\delta,h^\delta} ){h}^{2,\delta}_rdrds\Bigg)\\
 &+\int_0^t\frac{1}{\Delta}\int_s^{s+\Delta}\partial_y \Phi_g(X_r^{\delta,h^\delta} ,\mathcal{L}_{X^{\delta}_r},Y_r^{\delta,h^\delta} ){h}^{2,\delta}_rdrds\\
 =:&~{A}^\delta_1(t)+A^\delta_2(t).
\end{align*}
As for $A_1^\delta(t)$, due to the Lipschitz continuity of $\partial_y\Phi_g$ and  (\ref{Holder}),  we can deduce that
\begin{eqnarray}\label{hfx5}
\mathbb{E}\sup_{t\in[0,T]}|A_1^\delta(t)|^2\leq~~\!\!\!\!\!\!\!\!&&\frac{1}{\Delta^2}\mathbb{E}\Bigg(\int_0^T\int_s^{s+\Delta}\big(|X_s^{\delta,h^\delta}-X_r^{\delta,h^\delta}|^2+\mathbb{E}|X_s^{\delta}-X_r^{\delta}|^2\big)drds
\cdot\int_0^T\int_s^{s+\Delta}|{h}^{2,\delta}_r|^2drds\Bigg)
\nonumber \\
\leq~~\!\!\!\!\!\!\!\!&&\frac{C_{M,T}}{\Delta}\int_0^T\int_s^{s+\Delta}\big(\mathbb{E}|X_s^{\delta,h^\delta}-X_r^{\delta,h^\delta}|^2+\mathbb{E}|X_s^{\delta}-X_r^{\delta}|^2\big)drds
\nonumber \\
\leq~~\!\!\!\!\!\!\!\!&&\frac{C_{M,T}}{\Delta}\int_0^T\int_s^{s+\Delta}(r-s)drds
\nonumber \\
\leq~~\!\!\!\!\!\!\!\!&&C_{M,T}\Delta\to 0,~~\text{as}~\delta\to 0.
\end{eqnarray}
On the other hand,  note that $A^\delta_2(t)$ have the composition
\begin{align*}
A^\delta_2(t)=& \int_0^{\Delta}\frac{1}{\Delta}\int_0^{r}\partial_y \Phi_g(X_r^{\delta,h^\delta} ,\mathcal{L}_{X^{\delta}_r},Y_r^{\delta,h^\delta} ){h}^{2,\delta}_rdsdr \\
&+\int_{\Delta}^t\frac{1}{\Delta}\int_{r-\Delta}^{r}\partial_y \Phi_g(X_r^{\delta,h^\delta} ,\mathcal{L}_{X^{\delta}_r},Y_r^{\delta,h^\delta} ){h}^{2,\delta}_rdsdr \\
&+\int_{t}^{t+\Delta}\frac{1}{\Delta}\int_{r-\Delta}^{t}\partial_y \Phi_g(X_r^{\delta,h^\delta} ,\mathcal{L}_{X^{\delta}_r},Y_r^{\delta,h^\delta} ){h}^{2,\delta}_rdsdr \\
=:&\sum_{i=1}^3A_{2i}(t).
\end{align*}
Due to the boundedness of $\partial_y \Phi_g$ and ${h}^{2,\delta}\in\mathcal{A}_M$, it is easy to see that
\begin{equation}\label{hfx6}
\mathbb{E}\sup_{t\in[0,T]}|A_{21}(t)|^2+\mathbb{E}\sup_{t\in[0,T]}|A_{23}(t)|^2\to 0,~~\text{as}~\delta\to 0.
\end{equation}
Therefore, for the term $\mathcal{O}_{1}^\delta(t)$,  it follows that
\begin{align*}
  \mathbb{E}\sup_{t\in[0,T]}|\mathcal{O}_{1}^\delta(t)|^2\leq&~ \frac{\varepsilon}{\delta}C \mathbb{E}\Big|\int_0^\Delta\partial_y \Phi_g(X_r^{\delta,h^\delta} ,\mathcal{L}_{X^{\delta}_r},Y^{\delta,h^\delta}_r){h}^{2,\delta}_r dr\Big|^2\\
 &+ \frac{\varepsilon}{\delta}C\mathbb{E}\sup_{t\in[0,T]}\Big|\int_\Delta^t\partial_y \Phi_g(X_r^{\delta,h^\delta} ,\mathcal{L}_{X^{\delta}_r},Y^{\delta,h^\delta}_r){h}^{2,\delta}_r dr-A_{22}(t)\Big|^2\\
 &+ \frac{\varepsilon}{\delta}C\Big(\mathbb{E}\sup_{t\in[0,T]}|A^\delta_1(t)|^2+\mathbb{E}\sup_{t\in[0,T]}|A_{21}(t)|^2+\mathbb{E}\sup_{t\in[0,T]}|A_{23}(t)|^2\Big).
\end{align*}
Then by the boundedness of $\partial_y \Phi_g$, ${h}^{2,\delta}\in\mathcal{A}_M$ , (\ref{hfx5}), (\ref{hfx6}) and the definition of $A_{22}(t)$, it leads to
\begin{equation}\label{hfx7}
\mathbb{E}\sup_{t\in[0,T]}|\mathcal{O}_{1}^{\delta}(t)|^2\to 0,~~~\text{as}~\delta\to 0.
\end{equation}

For the term $\mathcal{O}_{2}^{\delta}(t)$ and $\mathcal{O}_{3}^{\delta}(t)$, following the same argument as in the proof of Step 2 in Lemma \ref{lemcon1}, it is straightforward that
$$\mathbb{E}\Big[\sup_{t\in[0,T]}\Big|\sum_{i=1}^{2}\mathcal{O}_{i}^{\delta}(t)\Big|\Big]\to 0,~\text{as}~\delta\to 0.$$
The proof is complete. }
\end{proof}

 \begin{lemma}\label{lemcon3}
 The following limit is valid with probability $1$:
\begin{align}\label{prop3}
&\lim_{\delta\rightarrow0}\sup_{t\in[0,T]}\left|\mathcal{K}_2^\delta(t)-\int_{\mathcal{Z}_1\times\mathcal{Z}_2\times\mathcal{Y}\times[0,t]}\partial_x\bar{b}(\bar{X}_s,\mathcal{L}_{\bar{X}_s})\cdot\bar{Z}_s\bar{\mathbf{P}}(dh^1dh^2dyds)
\right|=0.
\end{align}
\end{lemma}

\begin{proof}
Note that by  (\ref{ess12}), we know
\begin{align}\label{s1}
 \mathbb{E}\sup_{t\in[0,T]}\left|\mathcal{K}_2^\delta(t)
-\int_0^t\partial_x\bar{b}(\bar{X}_s,\mathcal{L}_{X^{\delta}_s})\cdot Z_s^{\delta,h^\delta}ds\right|\to0,~\text{as}~\delta\to 0.
\end{align}
Thus in order to prove (\ref{prop3}), it suffices to show the following convergence
\begin{eqnarray}
\!\!\!\!\!\!\!\!&&\int_0^t\partial_x\bar{b}(\bar{X}_s,\mathcal{L}_{X^{\delta}_s})\cdot Z_s^{\delta,h^\delta}ds-\int_{\mathcal{Z}_1\times\mathcal{Z}_2\times\mathcal{Y}\times[0,t]}\partial_x\bar{b}(\bar{X}_s,\mathcal{L}_{\bar{X}_s})\cdot \bar{Z}_s\bar{\mathbf{P}}(dh^1dh^2dyds)
\nonumber \\
=~~\!\!\!\!\!\!\!\!&&\Bigg(\int_0^t\partial_x\bar{b}(\bar{X}_s,\mathcal{L}_{X^{\delta}_s})\cdot Z_s^{\delta,h^\delta}ds
\nonumber \\
\!\!\!\!\!\!\!\!&&
-\int_{\mathcal{Z}_1\times\mathcal{Z}_2\times\mathcal{Y}\times[0,t]}\partial_x\bar{b}(\bar{X}_s,\mathcal{L}_{X^{\delta}_s})\cdot Z_s^{\delta,h^\delta}\mathbf{P}^{\delta,\Delta}(dh^1dh^2dyds)\Bigg)
\nonumber \\
\!\!\!\!\!\!\!\!&&+\Bigg(\int_{\mathcal{Z}_1\times\mathcal{Z}_2\times\mathcal{Y}\times[0,t]}\partial_x\bar{b}(\bar{X}_s,\mathcal{L}_{X^{\delta}_s})\cdot Z_s^{\delta,h^\delta}\mathbf{P}^{\delta,\Delta}(dh^1dh^2dyds)
\nonumber \\
\!\!\!\!\!\!\!\!&&
-\int_{\mathcal{Z}_1\times\mathcal{Z}_2\times\mathcal{Y}\times[0,t]}\partial_x\bar{b}(\bar{X}_s,\mathcal{L}_{\bar{X}_s})\cdot \bar{Z}_s\mathbf{P}^{\delta,\Delta}(dh^1dh^2dyds)\Bigg)
\nonumber \\
\!\!\!\!\!\!\!\!&&+\Bigg(\int_{\mathcal{Z}_1\times\mathcal{Z}_2\times\mathcal{Y}\times[0,t]}\partial_x\bar{b}(\bar{X}_s,\mathcal{L}_{\bar{X}_s})\cdot\bar{Z}_s\mathbf{P}^{\delta,\Delta}(dh^1dh^2dyds)
\nonumber \\
\!\!\!\!\!\!\!\!&&
-\int_{\mathcal{Z}_1\times\mathcal{Z}_2\times\mathcal{Y}\times[0,t]}\partial_x\bar{b}(\bar{X}_s,\mathcal{L}_{\bar{X}_s})\cdot\bar{Z}_s\bar{\mathbf{P}}(dh^1dh^2dyds)\Bigg)
\nonumber \\
=:~~\!\!\!\!\!\!\!\!&&\sum_{i=1}^{3}\mathcal{O}_{i}^{\delta}(t)
\to 0,~\text{as}~\delta\to 0.\label{con4}
\end{eqnarray}
{First, by (\ref{p1}) we have
\begin{align*}
  \int_{\mathcal{Z}_1\times\mathcal{Z}_2\times\mathcal{Y}\times[0,t]}\partial_x\bar{b}(\bar{X}_s,\mathcal{L}_{X^{\delta}_s})\cdot Z_s^{\delta,h^\delta}\mathbf{P}^{\delta,\Delta}(dh^1dh^2dyds)=&\int_0^t\frac{1}{\Delta}\int_s^{s+\Delta}\partial_x\bar{b}(\bar{X}_s,\mathcal{L}_{X^{\delta}_s})\cdot Z_s^{\delta,h^\delta}drds\\
 =&\int_0^t\partial_x\bar{b}(\bar{X}_s,\mathcal{L}_{X^{\delta}_s})\cdot Z_s^{\delta,h^\delta}ds.
\end{align*}
Therefore, it clear that $\mathcal{O}_{1}^{\delta}(t)= 0.$

Next, making use of the { boundedness and joint continuity }of $\partial_x\bar{b}$, (\ref{ess16}) and $Z_s^{\delta,h^\delta}\rightarrow\bar{Z}_s$, we get
\begin{align*}
 &\mathbb{E}\sup_{t\in[0,T]}\Bigg| \int_{\mathcal{Z}_1\times\mathcal{Z}_2\times\mathcal{Y}\times[0,t]}\partial_x\bar{b}(\bar{X}_s,\mathcal{L}_{X^{\delta}_s})\cdot Z_s^{\delta,h^\delta}\mathbf{P}^{\delta,\Delta}(dh^1dh^2dyds)\\
 &-\int_{\mathcal{Z}_1\times\mathcal{Z}_2\times\mathcal{Y}\times[0,t]}\partial_x\bar{b}(\bar{X}_s,\mathcal{L}_{\bar{X}_s})\cdot \bar{Z}_s\mathbf{P}^{\delta,\Delta}(dh^1dh^2dyds)\Bigg| \\
 \leq&~\mathbb{E} \int_{\mathcal{Z}_1\times\mathcal{Z}_2\times\mathcal{Y}\times[0,T]}\Big| \partial_x\bar{b}(\bar{X}_s,\mathcal{L}_{X^{\delta}_s})\cdot Z_s^{\delta,h^\delta}-\partial_x\bar{b}(\bar{X}_s,\mathcal{L}_{\bar{X}_s})\cdot Z_s^{\delta,h^\delta}\Big|\mathbf{P}^{\delta,\Delta}(dh^1dh^2dyds)\\
 &+\mathbb{E} \int_{\mathcal{Z}_1\times\mathcal{Z}_2\times\mathcal{Y}\times[0,T]}\Big| \partial_x\bar{b}(\bar{X}_s,\mathcal{L}_{\bar{X}_s})\cdot Z_s^{\delta,h^\delta}-\partial_x\bar{b}(\bar{X}_s,\mathcal{L}_{\bar{X}_s})\cdot \bar{Z}_s\Big|\mathbf{P}^{\delta,\Delta}(dh^1dh^2dyds)\\
\rightarrow&~ 0,~\text{as}~\delta\to 0.
\end{align*}
Finally, by (\ref{j3}) we can  deduce that
$\mathcal{O}_{3}^{\delta}(t)= 0.$
The proof is complete.}
\end{proof}

Note that Lemmas \ref{lemcon1}-\ref{lemcon3} imply that the limit $(\bar{Z},\bar{\mathbf{P}})$ of sequence  $\{(Z^{\delta,h^\delta},\mathbf{P}^{\delta,\Delta})\}_{\delta>0}$  satisfies the following integral equation
\begin{align*}
 \bar{Z}_t=&\int_{\mathcal{Z}_1\times\mathcal{Z}_2\times\mathcal{Y}\times[0,t]}\partial_x\bar{b}(\bar{X}_s,\mathcal{L}_{\bar{X}_s})\cdot \bar{Z}_s+\sigma(\bar{X}_s,\mathcal{L}_{\bar{X}_s}){h}^1\\
 &+\sqrt{\gamma}{\partial_y\Phi_g}(\bar{X}_s,\mathcal{L}_{\bar{X}_s},y){h}^2\bar{\mathbf{P}}(dh^1dh^2dyds) \\
 =& \int_{\mathcal{Z}_1\times\mathcal{Z}_2\times\mathcal{Y}\times[0,t]}\Theta(\bar{Z}_s,\bar{X}_s,\mathcal{L}_{\bar{X}_s},y,h^1,h^2)\bar{\mathbf{P}}(dh^1dh^2dyds),
\end{align*}
with probability 1. Hence, $(\bar{Z},\bar{\mathbf{P}})$ satisfies (\ref{j1}) in Definition \ref{defj1}. \hspace{\fill}$\Box$

\vspace{2mm}
\textbf{Proof of (\ref{j2}):}  It is sufficient to show that the third and fourth marginals of $\bar{\mathbf{P}}$ are given by the product of the invariant  measure $\nu^{\bar{X}_t,\mathcal{L}_{\bar{X}_t}}$ and Lebesgue measure. This will be included in Lemma \ref{va3}. Before giving that, we need some preliminary lemmas, i.e.~Lemmas \ref{va1} and \ref{va2} below.

Recall the controlled fast process $Y^{\delta,h^\delta}$ satisfying the equation
\begin{align*}
dY^{\delta,h^\delta}_t=&\frac{1}{\varepsilon}f(X^{\delta,h^\delta}_t,\mathcal{L}_{X^{\delta}_t},Y^{\delta,h^\delta}_t)dt
+\frac{\lambda(\delta)}{\sqrt{\delta\varepsilon}}g(X^{\delta,h^\delta}_t,\mathcal{L}_{X^{\delta}_t},Y^{\delta,h^\delta}_t){h}^{2,\delta}_t dt\\
&~~~~~~~~~~~~~+\frac{1}{\sqrt{\varepsilon}}g(X^{\delta,h^\delta}_t,\mathcal{L}_{X^{\delta}_t},Y^{\delta,h^\delta}_t)dW_t^2, \quad Y^{\delta,h^\delta}_0=y.
\end{align*}
Let $\tilde{Y}^\delta$ denote the uncontrolled fast process depending on the controlled
slow process $X^{\delta,h^{\delta}}$, which is given by
$$d\tilde{Y}_t^\delta=\frac{1}{\varepsilon}f(X^{\delta,h^{\delta}}_t,\mathcal{L}_{X^{\delta}_t},\tilde{Y}_t^\delta)dt
+\frac{1}{\sqrt{\varepsilon}}g(X^{\delta,h^{\delta}}_t,\mathcal{L}_{X^{\delta}_t},\tilde{Y}_t^\delta)d W_t^2,\quad\tilde{Y}_0^\delta=y.$$

In the following lemma, we show that the processes $Y^{\delta,h^\delta}_t$ and  $\tilde{Y}_t^\delta$ are close in $L^2$-sense.
\begin{lemma}\label{va1}
Let $M>0, \{h^\delta\}_{\delta>0}\subset\mathcal{A}_M$,  $\Delta$ as in Definition (\ref{delta}). Then
\begin{equation}\label{t1}
\frac{1}{\Delta}\mathbb{E}\Bigg[\int_0^T|Y^{\delta,h^{\delta}}_t-\tilde{Y}_t^\delta|^2dt\Bigg]\rightarrow0,~~\text{as}~\delta\rightarrow0.
\end{equation}
\end{lemma}
\begin{proof}
Let $\rho_t:=Y^{\delta,h^\delta}_t-\tilde{Y}_t^\delta$, using It\^{o}'s formula for $|\rho_t|^2$ and then taking expectation, we derive that
\begin{eqnarray*}
\frac{d}{dt}\mathbb{E}|\rho_t|^2
=~~\!\!\!\!\!\!\!\!&&\frac{2}{\varepsilon}\mathbb{E}\Big[\langle f(X^{\delta,h^\delta}_t,\mathcal{L}_{X^{\delta}_t},Y^{\delta,h^\delta}_t)-f(X^{\delta,h^\delta}_t,\mathcal{L}_{X^{\delta}_t},\tilde{Y}_t^\delta),\rho_t\rangle\Big]\nonumber\\
\!\!\!\!\!\!\!\!&&+\frac{1}{\varepsilon}\mathbb{E}\|g(X^{\delta,h^\delta}_t,\mathcal{L}_{X^{\delta}_t},Y^{\delta,h^\delta}_t)
-g(X^{\delta,h^\delta}_t,\mathcal{L}_{X^{\delta}_t},\tilde{Y}_t^\delta)\|^2
\nonumber\\
\!\!\!\!\!\!\!\!&&+\frac{2\lambda(\delta)}{\sqrt{\delta\varepsilon}}\mathbb{E}\Big[\langle g(X^{\delta,h^\delta}_t,\mathcal{L}_{X^{\delta}_t},Y^{\delta,h^\delta}_t){h}^{2,\delta}_t,\rho_t\rangle\Big].
\nonumber\\
=:~~\!\!\!\!\!\!\!\!&&\sum_{i=1}^3I_i.
\end{eqnarray*}
In view of the condition \eref{sm}, we can deduce that
\begin{equation*}
I_1+I_2\leq-\frac{\kappa}{\varepsilon}\mathbb{E}|\rho_t|^2.
\end{equation*}
Moreover, we have
\begin{eqnarray*}
I_3\!\!\!\!\!\!\!\!&&~~~\leq\frac{C\lambda(\delta)}{\sqrt{\delta\varepsilon}}\mathbb{E}\Big[|{h}^{2,\delta}_t||\rho_t|\Big]
\leq\frac{C\lambda(\delta)^2}{\delta}\mathbb{E}|{h}^{2,\delta}_t|^2+\frac{\epsilon_0}{\varepsilon}\mathbb{E}|\rho_t|^2,
\end{eqnarray*}
where $\epsilon_0\in(0,\kappa)$.

Then we have
\begin{eqnarray*}
\frac{d}{dt}\mathbb{E}|\rho_t|^2
\leq~~\!\!\!\!\!\!\!\!&&-\frac{\beta}{\varepsilon}\mathbb{E}|\rho_t|^2+\frac{C\lambda(\delta)^2}{\delta}\mathbb{E}|{h}^{2,\delta}_t|^2,
\end{eqnarray*}
where $\beta:=\kappa-\epsilon_0>0$. The Gronwall lemma leads to
\begin{eqnarray*}
\mathbb{E}|\rho_t|^2\leq~~\!\!\!\!\!\!\!\!&&
\frac{C\lambda(\delta)^2}{\delta}\int_0^te^{-\frac{\beta}{\varepsilon}(t-s)}\mathbb{E}|{h}^{2,\delta}_s|^2ds.
\end{eqnarray*}
Therefore, for $h^{\delta}\in\mathcal{A}_M$, we obtain
\begin{eqnarray*}
\mathbb{E}\Big[\int_0^T|\rho_t|^2dt\Big]
~\leq
\!\!\!\!\!\!\!\!&&~~~\frac{C\lambda(\delta)^2}{\delta}\mathbb{E}\Big[\int_0^T\int_0^te^{-\frac{\beta}{\varepsilon}(t-s)}|{h}^{2,\delta}_s|^2dsdt\Big]
\leq\frac{C_{M,T}\varepsilon\lambda(\delta)^2}{\delta}.
\end{eqnarray*}
Note that $\frac{\varepsilon\lambda(\delta)^2}{\delta\Delta}\rightarrow0$ as $\delta\rightarrow0$, we complete the proof of (\ref{t1}).
\end{proof}

Similarly, for $s\geq t$ we can define the two parameter process $Y^{\delta,X^{\delta,h^{\delta}}_t}(s;t)$ solving
\begin{align*}
dY^{\delta,X^{\delta,h^{\delta}}_t}(s;t)&=\frac{1}{\varepsilon}f(X^{\delta,h^{\delta}}_t,\mathcal{L}_{X^{\delta}_t},Y^{\delta,X^{\delta,h^{\delta}}_t}(s;t))ds
+\frac{1}{\sqrt{\varepsilon}}g(X^{\delta,h^{\delta}}_t,\mathcal{L}_{X^{\delta}_t},Y^{\delta,X^{\delta,h^{\delta}}_t}(s;t))d W_s^2,\\
Y^{\delta,X^{\delta,h^{\delta}}_t}(t;t)&=\tilde{Y}^\delta_t.
\end{align*}
The following lemma shows that the process $\tilde{Y}_t^\delta$  is close to the  process $Y^{\delta,X^{\delta,h^{\delta}}_t}(s;t)$
in  $L^2$-sense on the interval $s\in[t,t+\Delta]$.
\begin{lemma}\label{va2}
Let $M>0, \{h^\delta\}_{\delta>0}\subset\mathcal{A}_M$,  $\Delta$ as in Definition (\ref{delta}). Then
\begin{equation}\label{t2}
\frac{1}{\Delta}\mathbb{E}\Bigg[\int_t^{t+\Delta}|\tilde{Y}_s^\delta-Y^{\delta,X^{\delta,h^{\delta}}_t}(s;t)|^2ds\Bigg]\rightarrow0,~~\text{as}~\delta\rightarrow0.
\end{equation}
\end{lemma}
\begin{proof}
By \cite[Lemma 3.2]{RSX1} and Lemma \ref{xg}, we know for any $s\in[t,t+\Delta]$,
$$\mathbb{E}|X_{t}^{\delta,h^\delta}-X_{s}^{\delta,h^\delta}|^2\leq C_{M,T}(1+|x|^2+|y|^2)\Delta$$
and
$$\mathbb{E}|X_{t}^{\delta}-X_{s}^{\delta}|^2\leq C_{T}(1+|x|^2+|y|^2)\Delta.$$
With these estimates at hand, we can follow the similar argument as in Lemma \ref{va1}  to obtain (\ref{t2}).
\end{proof}

We now have all the necessary ingredients to show that (\ref{j2}) holds.
\begin{lemma}\label{va3}
 $\mathbf{\bar{P}}$ has the  decomposition  (\ref{j2}), i.e.~for any  $F\in C_b(\mathcal{Y})$,
 $$\int_{\mathcal{Z}_1\times\mathcal{Z}_2\times\mathcal{Y}\times[0,T]}F(y)\mathbf{\bar{P}}(dh^1dh^2dydt)=\int_0^T\int_{\mathcal{Y}}F(y)\nu^{\bar{X}_t,\mathcal{L}_{\bar{X}_t}}(dy)dt.$$
\end{lemma}
\begin{proof}
 Without loss of generality, we assume that $F$ is Lipschitz continuous. Note that we have the following decomposition
\begin{align*}
 &\int_{\mathcal{Z}_1\times\mathcal{Z}_2\times\mathcal{Y}\times[0,T]}F(y)\mathbf{\bar{P}}(dh^1dh^2dydt)-\int_0^T\int_{\mathcal{Y}}F(y)\nu^{\bar{X}_t,\mathcal{L}_{\bar{X}_t}}(dy)dt\\
 =&\int_{\mathcal{Z}_1\times\mathcal{Z}_2\times\mathcal{Y}\times[0,T]}F(y)\mathbf{\bar{P}}(dh^1dh^2dydt)- \int_{\mathcal{Z}_1\times\mathcal{Z}_2\times\mathcal{Y}\times[0,T]}F(y)\mathbf{P}^{\delta,\Delta}(dh^1dh^2dydt)\\
 &+\int_{\mathcal{Z}_1\times\mathcal{Z}_2\times\mathcal{Y}\times[0,T]}F(y)\mathbf{P}^{\delta,\Delta}(dh^1dh^2dydt)
-\int_0^T\int_{\mathcal{Y}}F(y)\nu^{\bar{X}_t,\mathcal{L}_{\bar{X}_t}}(dy)dt\\
=&\left(\int_{\mathcal{Z}_1\times\mathcal{Z}_2\times\mathcal{Y}\times[0,T]}F(y)\mathbf{\bar{P}}(dh^1dh^2dydt)- \int_{\mathcal{Z}_1\times\mathcal{Z}_2\times\mathcal{Y}\times[0,T]}F(y)\mathbf{P}^{\delta,\Delta}(dh^1dh^2dydt)\right)\\
 &+\left(\int_0^T\frac{1}{\Delta}\int_t^{t+\Delta}F(Y^{\delta, h^\delta}_s)dsdt
- \int_0^T\frac{1}{\Delta}\int_t^{t+\Delta}F(\tilde{Y}^{\delta}_s)dsdt\right)\\
&+\left(\int_0^T\frac{1}{\Delta}\int_t^{t+\Delta}F(\tilde{Y}^{\delta}_s)dsdt
-\int_0^T\frac{1}{\Delta}\int_t^{t+\Delta}F(Y^{\delta,X^{\delta,h^{\delta}}_t}(s;t))dsdt\right)\\
&+\left(\int_0^T\frac{1}{\Delta}\int_t^{t+\Delta}F(Y^{\delta,\bar{X}_t}(s;t))dsdt
-\int_0^T\int_{\mathcal{Y}}F(y)\nu^{\bar{X}_t,\mathcal{L}_{\bar{X}_t}}(dy)dt\right)\\
&+\left(\int_0^T\frac{1}{\Delta}\int_t^{t+\Delta}F(Y^{\delta,X^{\delta,h^{\delta}}_t}(s;t))dsdt
-\int_0^T\frac{1}{\Delta}\int_t^{t+\Delta}F(Y^{\delta,\bar{X}_t}(s;t))dsdt\right)\\
=:&\sum_{i=1}^5J^{\delta}_i.
\end{align*}

We intend to show that the terms $J^{\delta}_i$, $i=1,\ldots,5$,  vanish in probability as $\delta\rightarrow0$.
First, it obvious that $J^{\delta}_1$ tends to zero in probability, as $\delta\rightarrow0$, by (\ref{con3}). By Lemma \ref{va1}  and the dominated convergence theorem, we can deduce that $J^{\delta}_2$ tends to zero in probability as $\delta\rightarrow0$. Similarly, Lemma \ref{va2}  and the dominated convergence theorem show that $J^{\delta}_3$  tends to zero in probability. We next deal with the  term $J^{\delta}_4$. We  define the following two parameter process $Y^{\delta,\bar{X}_t}(s;t)$ solving
\begin{align*}
dY^{\delta,\bar{X}_t}(s;t)=&\frac{1}{\varepsilon}f(\bar{X}_t,\mathcal{L}_{\bar{X}_t},Y^{\delta,\bar{X}_t}(s;t))ds
+\frac{1}{\sqrt{\varepsilon}}g(\bar{X}_t,\mathcal{L}_{\bar{X}_t},Y^{\delta,\bar{X}_t}(s;t))dW_s^2,~
Y^{\delta,\bar{X}_t}(t;t)=\tilde{Y}^\delta_t,
\end{align*}
and the  time-rescaled process $Y^{\bar{X}_t}_s:=Y^{\delta,\bar{X}_t}(t+\varepsilon s;t)$
\begin{align*}
dY^{\bar{X}_t}_s=&f(\bar{X}_t,\mathcal{L}_{\bar{X}_t},Y^{\bar{X}_t}_s)ds
+g(\bar{X}_t,\mathcal{L}_{\bar{X}_t},Y^{\bar{X}_t}_s)d W_s^2,~~
Y^{\bar{X}_t}_0=\tilde{Y}^\delta_t,\quad 0\leq s\leq \frac{\Delta}{\varepsilon}.
\end{align*}
We notice that
$$\frac{1}{\Delta}\int_t^{t+\Delta}F(Y^{\delta,\bar{X}_t}(s;t))ds=\frac{\varepsilon}{\Delta}\int_0^{\frac{\Delta}{\varepsilon}}F(Y^{\bar{X}_t}_s)ds.$$
Hence, {according to the ergodic theorem \cite[Theorem 4.2]{khas},} we obtain that
\begin{equation}\label{eser}
\lim_{\delta\rightarrow0}\frac{\varepsilon}{\Delta}\int_0^{\frac{\Delta}{\varepsilon}}F(Y^{\bar{X}_t}_s)ds=\int_{\mathcal{Y}}F(y)\nu^{\bar{X}_t,\mathcal{L}_{\bar{X}_t}}(dy),
\end{equation}
which together with the dominated convergence theorem imply that $J^{\delta}_4$  tends to zero in probability.
Finally, it remains to study the term
\begin{align*}
  J^{\delta}_5=&\int_0^T\frac{1}{\Delta}\int_t^{t+\Delta}F(Y^{\delta,X^{\delta,h^{\delta}}_t}(s;t))dsdt
-\int_0^T\frac{1}{\Delta}\int_t^{t+\Delta}F(Y^{\delta,\bar{X}_t}(s;t))dsdt.
\end{align*}
Recall that
$$\lim_{\delta\rightarrow0}\mathbb{E}|X_{t}^{\delta,h^\delta}-\bar{X}_t|=0$$
and
$$\lim_{\delta\rightarrow0}\mathbb{E}|X_{t}^{\delta}-\bar{X}_t|=0.$$
With these estimates at hand, we can then proceed using the same argument as we did in Lemma \ref{va1} to obtain
$$\frac{1}{\Delta}\int_t^{t+\Delta}F(Y^{\delta,X^{\delta,h^\delta}_t}(s;t))ds-
\frac{1}{\Delta}\int_t^{t+\Delta}F(Y^{\delta,\bar{X}_t}(s;t))ds\rightarrow0,~\text{as}~\delta\rightarrow0.$$
By the dominated convergence theorem, it follows that $J^{\delta}_5$  tends to zero in probability. We finish the proof.
\end{proof}

\subsection{Laplace principle lower bound}
We now prove the Laplace principle lower bound. More precisely, we want to show that for all bounded and continuous  functions $\Lambda:C([0,T];\mathbb{R}^n)\to\mathbb{R}$,
\begin{align*}
 & \liminf_{\delta\rightarrow0}\left(-\frac{\delta}{\lambda^2(\delta)}\log\mathbb{E}\left[\exp\left\{-\frac{\lambda^2(\delta)}{\delta}\Lambda(Z^\delta)\right\}\right]\right)  \\
  \geq &\inf_{(\varphi,\mathbf{P})\in\mathcal{V}_{(\Theta,\nu^{\bar{X},\mathcal{L}_{\bar{X}}})}}\left[\frac{1}{2}\int_{\mathcal{Z}\times\mathcal{Z}\times\mathcal{Y}\times[0,T]}\Big(|h^1|^2+|h^2|^2\Big){\mathbf{P}}(dh^1dh^2dydt)+\Lambda(\varphi)\right].
\end{align*}
To this end, it is sufficient to prove the lower limit along any subsequence such that
$$-\frac{\delta}{\lambda^2(\delta)}\log\mathbb{E}\left[\exp\left\{-\frac{\lambda^2(\delta)}{\delta}\Lambda(Z^\delta)\right\}\right]$$
converges.  Such a subsequence exists since $$\Bigg|-\frac{\delta}{\lambda^2(\delta)}\log\mathbb{E}\left[\exp\left\{-\frac{\lambda^2(\delta)}{\delta}\Lambda(Z^\delta)\right\}\right]\Bigg|\leq C\|\Lambda\|_{\infty}.$$
By  \cite[Theorem 3.17]{bd},  it follows that for any $\eta>0$ there exists $M>0$ such that for any $\delta>0$ there exist $h^{1,\delta},h^{2,\delta}\in\mathcal{A}_M$, we have
$$-\frac{\delta}{\lambda^2(\delta)}\log\mathbb{E}\left[\exp\left\{-\frac{\lambda^2(\delta)}{\delta}\Lambda(Z^\delta)\right\}\right]\geq\mathbb{E}\left[\left(\frac{1}{2}\int_0^T\Big(|h^{1,\delta}_s|^2+|h^{2,\delta}_s|^2\Big)ds+\Lambda(Z^{\delta,h^\delta})\right)\right]-\eta.$$
Hence if we use this family of controls $h^\delta=(h^{1,\delta},h^{2,\delta})$ and the associated controlled process $Z^{\delta,h^\delta}$
to construct occupation measures $\mathbf{P}^{\delta,\Delta}$, then by Theorem \ref{prj1} the
family $\{(Z^{\delta,h^\delta},\mathbf{P}^{\delta,\Delta})\}_{\delta>0}$ is tight. Thus given any sequence in $\{\delta\}$ there is a subsequence for which
$$(Z^{\delta,h^\delta},\mathbf{P}^{\delta,\Delta})\Rightarrow(\bar{Z},\bar{\mathbf{P}}),$$
with $(\bar{Z},\bar{\mathbf{P}})\in\mathcal{V}_{(\Theta,\nu^{\bar{X},\mathcal{L}_{\bar{X}}})}$.  By Fatou's lemma, we  obtain
\begin{align*}
  &\liminf_{\delta\rightarrow0}\left(-\frac{\delta}{\lambda^2(\delta)}\log\mathbb{E}\left[\exp\left\{-\frac{\lambda^2(\delta)}{\delta}\Lambda(Z^\delta)\right\}\right]\right) \\
  \geq  & \liminf_{\delta\rightarrow0}\mathbb{E}\left[\left(\frac{1}{2}\int_0^T\Big(|h^{1,\delta}_s|^2+|h^{2,\delta}_s|^2\Big)ds+\Lambda(Z^{\delta,h^\delta})\right)\right]-\eta\\
   \geq&\liminf_{\delta\rightarrow0}\mathbb{E}\left[\frac{1}{2}\int_0^T\frac{1}{\Delta}\int_t^{t+\Delta}\Big( |h^{1,\delta}_s|^2+|h^{2,\delta}_s|^2\Big)dsdt+\Lambda(Z^{\delta,h^\delta})\right]-\eta\\
  =&\liminf_{\delta\rightarrow0}\mathbb{E}\left[\frac{1}{2}\int_{\mathcal{Z}_1\times\mathcal{Z}_2\times\mathcal{Y}\times[0,T]}\Big(| h^1|^2+|h^2|^2\Big)\mathbf{P}^{\delta,\Delta}(dh^1dh^2dydt)+\Lambda(Z^{\delta,h^\delta})\right]-\eta\\
  \geq & \mathbb{E}\left[\frac{1}{2} \int_{\mathcal{Z}_1\times\mathcal{Z}_2\times\mathcal{Y}\times[0,T]}\Big(|h^1|^2+|h^2|^2\Big)\bar{\mathbf{P}}(dh^1dh^2dydt)+\Lambda(\bar{Z})\right]-\eta\\
  \geq & \inf_{(\varphi,\mathbf{P})\in\mathcal{V}_{(\Theta,\nu^{\bar{X},\mathcal{L}_{\bar{X}}})}}\left[\frac{1}{2}\int_{\mathcal{Z}_1\times\mathcal{Z}_2\times\mathcal{Y}\times[0,T]}\Big(|h^1|^2+|h^2|^2\Big){\mathbf{P}}(dh^1dh^2dydt)+\Lambda(\varphi)\right]-\eta\\
\geq&\inf_{\varphi\in C([0,T];\mathbb{R}^n )}\big[I(\varphi)+\Lambda(\varphi)\big]-\eta.
\end{align*}
Since $\eta>0$ is arbitrary, the lower bound is proved. \hspace{\fill}$\Box$

\subsection{Compactness of level sets of $I(\cdot)$}
We want to prove that for each $s<\infty$, the level set
$$\Gamma(s):=\Big\{\varphi\in C([0,T];\mathbb{R}^n): I(\varphi)\leq s\Big\}$$
is a compact subset of $C([0,T];\mathbb{R}^n)$. The proof is analogous to the proof of the Laplace principle lower bound. Specifically, in Lemma \ref{com1} we show the pre-compactness of $\Phi(s)$ and  in Lemma \ref{com3} we show that it is closed, then the assertion follows.

\begin{lemma}\label{com1}
Fix $K<\infty$ and consider any sequence $\{(\varphi^k,\mathbf{P}^k)\}_{k\in\mathbb{N}}\subset\mathcal{V}_{(\Theta,\nu^{\bar{X},\mathcal{L}_{\bar{X}}})}$ such that for every $k\in\mathbb{N}$, $(\varphi^k,\mathbf{P}^k)$
 is viable and
\begin{equation}\label{esuni}
\int_{\mathcal{Z}_1\times\mathcal{Z}_2\times\mathcal{Y}\times [0,T]}\Big(|h^1|^2+|h^2|^2+|y|^2\Big)\mathbf{P}^k(dh^1dh^2dydt)\leq K.
\end{equation}
Then $\{(\varphi^k,\mathbf{P}^k)\}_{k\in\mathbb{N}}$ is pre-compact.
\end{lemma}
\begin{proof}
Note that for any $0\leq t_1<t_2\leq T$ and $k\in\mathbb{N}$, {by Lemma \ref{C1}, we have}
\begin{align*}
  |\varphi^k_{t_2}-\varphi^k_{t_1}|=&\left|\int_{\mathcal{Z}_1\times\mathcal{Z}_2\times\mathcal{Y}\times [t_1,t_2]}\Theta(\varphi^k_s,\bar{X}_s,\mathcal{L}_{\bar{X}_s},y,h^1,h^2)\mathbf{P}^k(dh^1dh^2dyds)\right|\\
\leq & C(t_2-t_1)^{\frac{1}{2}}\Bigg(\int_{\mathcal{Z}_1\times\mathcal{Z}_2\times\mathcal{Y}\times [t_1,t_2]}|\partial_x \bar{b}(\bar{X}_s,\mathcal{L}_{\bar{X}_s})\varphi_s^k|\mathbf{P}(dh^1dh^2dyds)\Bigg)^{\frac{1}{2}}\\
&+C(t_2-t_1)^{\frac{1}{2}}\Bigg(\int_{\mathcal{Z}_1\times\mathcal{Z}_2\times\mathcal{Y}\times [t_1,t_2]}|h^1|^2+|h^2|^2\mathbf{P}(dh^1dh^2dyds)\Bigg)^{\frac{1}{2}}\\
\leq & C(t_2-t_1)^{\frac{1}{2}}.
\end{align*}
This and the fact that $\varphi^k_0=0$ imply the pre-compactness of $\{\varphi^k\}_{k\in\mathbb{N}}$ by Arzel\`{a}-Ascoli theorem.

Pre-compactness of
$\{\mathbf{P}^k\}_{k\in\mathbb{N}}$ follows from (\ref{esuni}) and the exactly same argument as in Proposition \ref{prj1}.
\end{proof}

In order to show the level set $\Gamma(s)$ is closed, we give an important lemma  which states that the limit of any sequence of viable pair is also viable.

\begin{lemma}\label{com2}
Fix $K<\infty$ and consider any  convergent sequence $\{(\varphi^k,\mathbf{P}^k)\}_{k\in\mathbb{N}}$ such that for every $k\in\mathbb{N}$, $(\varphi^k,\mathbf{P}^k)$
 is viable and
 \begin{equation}\label{ass1}
  \int_{\mathcal{Z}_1\times\mathcal{Z}_2\times\mathcal{Y}\times [0,T]}\Big(|h^1|^2+|h^2|^2+|y|^2\Big)\mathbf{P}^k(dh^1dh^2dydt)\leq K.
 \end{equation}
 Then  the limit $(\varphi,\mathbf{P})$ is a viable pair.
\end{lemma}
\begin{proof}
Since $(\varphi^k,\mathbf{P}^k)$ is viable, we know that
\begin{equation}\label{co11}
 \varphi^k_t=\int_{\mathcal{Z}_1\times\mathcal{Z}_2\times\mathcal{Y}\times [0,t]}\Theta(\varphi^k_s,\bar{X}_s,\mathcal{L}_{\bar{X}_s},y,h^1,h^2)\mathbf{P}^k(dh^1dh^2dyds)
\end{equation}
for every $t\in[0,T]$, and
\begin{equation}\label{co12}
 \mathbf{P}^k(dh^1dh^2dydt)=\eta^k(dh^1dh^2|y,t)\nu^{\bar{X}_t,\mathcal{L}_{\bar{X}_t}}(dy)dt.
\end{equation}
where $\eta^k$ is a sequence of stochastic kernels.

First, applying Fatou's Lemma we can get $\mathbf{P}$ has a finite second moment in the sense of (i) in Definition \ref{defj1}. Moreover, observe that the function
 $\Theta(\varphi,x,\mu,y,h^1,h^2)$ and the operator $\mathbf{L}^2_{x,\mu}$ are continuous in $\varphi,x,\mu,y$ and affine in $h^1,h^2$.  In addition, one can prove the uniform integrability of $\mathbf{P}^k$ analogously to Theorem \ref{prj1}. Hence by assumption (\ref{ass1}) and the convergence $\mathbf{P}^k\rightarrow\mathbf{P}$ and $\varphi^k\rightarrow\varphi$, we conclude that $(\varphi,\mathbf{P})$ satisfy (\ref{j1}). Due to the same reason, it is clear that $ \mathbf{P}$ satisfies (\ref{j2}).
%Finally, it follows from $\mathbf{P}^n(\mathcal{Z}_1\times\mathcal{Z}_2\times\mathcal{Y}\times[0,t])=t$ and $\mathbf{P}(\mathcal{Z}_1\times\mathcal{Z}_2\times\mathcal{Y}\times\{t\})=0$ that $\mathbf{P}(\mathcal{Z}_1\times\mathcal{Z}_2\times\mathcal{Y}\times[0,t])=t$
%for all $t\in[0,T]$.
We complete the proof.
\end{proof}

\begin{lemma}\label{com3}
The functional  $I(\varphi)$  is lower semicontinuous.
\end{lemma}
\begin{proof}
Consider a sequence $\varphi^k$ with limit $\varphi$. We want to prove
$$\liminf_{k\rightarrow\infty }I(\varphi^k)\geq I(\varphi).$$
It suffices to consider the case when $I(\varphi^k)$ has a finite limit, i.e., there exists a $M<\infty$ such that
$\lim_{k\rightarrow\infty }I(\varphi^k)\leq M$.
We recall the definition
$$I(\varphi)=\inf_{(\varphi,\mathbf{P})\in\mathcal{V}_{(\Theta,\nu^{\bar{X},\mathcal{L}_{\bar{X}}})}}\Bigg\{\frac{1}{2}\int_{\mathcal{Z}_1\times\mathcal{Z}_2\times\mathcal{Y}\times [0,T]}\Big(|h^1|^2+|h^2|^2\Big)\mathbf{P}(dh^1dh^2dydt)\Bigg\}.$$
{Hence there exists a sequence $\{\mathbf{P}^k\}_{k\in\mathbb{N}}$ satisfying $(\varphi^k,\mathbf{P}^k)\in\mathcal{V}_{(\Theta,\nu^{\bar{X},\mathcal{L}_{\bar{X}}})}$ and
$$\sup_{k\in\mathbb{N}}\frac{1}{2}\int_{\mathcal{Z}_1\times\mathcal{Z}_2\times\mathcal{Y}\times [0,T]}\Big(|h^1|^2+|h^2|^2\Big)\mathbf{P}^k(dh^1dh^2dydt)\leq M+1+\frac{1}{2}\int_0^T\int_\mathcal{Y}|y|^2\nu^{\bar{X}_{t},\mathcal{L}_{\bar{X}_{t}}}(dy)dt,$$
and such that
$$I(\varphi^k)\geq\frac{1}{2}\int_{\mathcal{Z}_1\times\mathcal{Z}_2\times\mathcal{Y}\times [0,T]}\Big(|h^1|^2+|h^2|^2\Big)\mathbf{P}^k(dh^1dh^2dydt)-\frac{1}{k}.$$
By  the reference \cite[(4.28)]{HLL4} (see also \cite[Theorem 4.3.9]{LR1}), we have $\int_0^T\int_\mathcal{Y}|y|^2\nu^{\bar{X}_{t},\mathcal{L}_{\bar{X}_{t}}}(dy)dt<\infty$. Hence there exists a constant $M'>0$ such that
$$\sup_{k\in\mathbb{N}}\frac{1}{2}\int_{\mathcal{Z}_1\times\mathcal{Z}_2\times\mathcal{Y}\times [0,T]}\Big(|h^1|^2+|h^2|^2\Big)\mathbf{P}^k(dh^1dh^2dydt)\leq M'.$$}
It follows from Lemma \ref{com1} that we can consider a subsequence along which $(\varphi^k,\mathbf{P}^k)$ converges
to a limit $(\varphi,\mathbf{P})$. By Lemma \ref{com2}, we know that $(\varphi,\mathbf{P})$ is viable. Hence by Fatou's Lemma
\begin{align*}
 \liminf_{k\rightarrow\infty }I(\varphi^k)\geq & \liminf_{k\rightarrow\infty }\Big(\frac{1}{2}\int_{\mathcal{Z}_1\times\mathcal{Z}_2\times\mathcal{Y}\times [0,T]}\Big[|h^1|^2+|h^2|^2\Big]\mathbf{P}^k(dh^1dh^2dydt)-\frac{1}{k}\Big)\\
  \geq & \frac{1}{2}\int_{\mathcal{Z}_1\times\mathcal{Z}_2\times\mathcal{Y}\times [0,T]}\Big(|h^1|^2+|h^2|^2\Big)\mathbf{P}(dh^1dh^2dydt)\\
   \geq & \inf_{(\varphi,\mathbf{P})\in\mathcal{V}_{(\Theta,\nu^{\bar{X},\mathcal{L}_{\bar{X}}})}}\Bigg\{\frac{1}{2}\int_{\mathcal{Z}_1\times\mathcal{Z}_2\times\mathcal{Y}\times [0,T]}\Big(|h^1|^2+|h^2|^2\Big)\mathbf{P}(dh^1dh^2dydt)\Bigg\}\\
   =&I(\varphi),
\end{align*}
which concludes the proof of lower-semicontinuity of $I$.
\end{proof}

\subsection{Laplace principle upper bound}

The proof of the Laplace principle upper bound is quite involved  for which we have to establish the nearly optimal controls   that achieve the bound.

Define function
$$I(\varphi):=\inf_{\mathbf{P}\in\Xi_{\varphi}}\frac{1}{2}\int_0^T\int_{\mathcal{Z}_1\times\mathcal{Z}_2\times\mathcal{Y}}\Big(|h^1|^2+|h^2|^2\Big)\mathbf{P}_t(dh^1dh^2dy)dt,$$
where the set
\begin{equation*}
  \Xi_{\varphi}:=\left\{
\begin{array}{ll}
\mathbf{P}:[0,T]&\rightarrow\mathcal{P}(\mathcal{Z}_1\times\mathcal{Z}_2\times\mathcal{Y}):\\
&\mathbf{P}_t(A_1\times A_2\times A_3)=\int_{A_3}\eta(A_1\times A_2|y,t)\nu^{\bar{X}_t,\mathcal{L}_{\bar{X}_t}}(dy),\\
&\int_0^T\int_{\mathcal{Z}_1\times\mathcal{Z}_2\times\mathcal{Y}}\Big(|h^1|^2+|h^2|^2+|y|^2\Big)\mathbf{P}_t(dh^1dh^2dy)dt<\infty, \\ &\varphi_t=\int_0^t\int_{\mathcal{Z}_1\times\mathcal{Z}_2\times\mathcal{Y}}\Theta(\varphi_s,\bar{X}_s,\mathcal{L}_{\bar{X}_s},y,h^1,h^2)\mathbf{P}_s(dh^1dh^2dy)ds.
\end{array}
\right\}.
\end{equation*}
Moreover,  we also define
\begin{equation}\label{L0}
\tilde{I}(\varphi):=\inf_{(z^1,z^2)\in\tilde{\Xi}_{\varphi}}\frac{1}{2}\int_0^T\int_{\mathcal{Y}}|z^1_t(y)|^2+|z^2_t(y)|^2\nu^{\bar{X}_t,\mathcal{L}_{\bar{X}_t}}(dy)dt,
\end{equation}
where the set
\begin{equation*}
  \tilde{\Xi}_{\varphi}
  :=\left\{
\begin{array}{ll}
  z:=(z^1,z^2):&[0,T]\times\mathcal{Y}\to\mathbb{R}^{d_1}\times\mathbb{R}^{d_2}: \\ &\int_0^T\int_{\mathcal{Y}}\Big(|z^1_t(y)|^2+|z^2_t(y)|^2+|y|^2\Big)\nu^{\bar{X}_t,\mathcal{L}_{\bar{X}_t}}(dy)dt<\infty, \\
   &\varphi_t=\int_0^t\int_{\mathcal{Y}}\Theta(\varphi_s,\bar{X}_s,\mathcal{L}_{\bar{X}_s},y,z^1_s(y),z^2_s(y))\nu^{\bar{X}_s,\mathcal{L}_{\bar{X}_s}}(dy)ds.
   \end{array}
\right\}.
\end{equation*}

\textbf{Claim:}
$I(\varphi)=\tilde{I}(\varphi)$, $\varphi\in C([0,T];\mathbb{R}^n)$.
\begin{proof}
First, for any given $z:=(z^1,z^2)\in\tilde{\Xi}_{\varphi}$, we can define $\mathbf{P}\in \Xi_{\varphi}$ by
$$ \mathbf{P}_t(dh^1dh^2dy):=\delta_{(z^1_t(y),z^2_t(y))}(dh^1dh^2)\nu^{\bar{X}_t,\mathcal{L}_{\bar{X}_t}}(dy).$$
Hence, it is easy to see that
 $$I(\varphi)\leq\tilde{I}(\varphi).$$
For the converse, for any given $\mathbf{P}\in \Xi_{\varphi}$,  we define
$$z_t(y):=(z^1_t(y),z^2_t(y)):=\int_{\mathcal{Z}_1\times\mathcal{Z}_2}h\eta(dh^1dh^2|y,t),$$
where we denote  $h:=(h^1,h^2)$,         $\eta(dh^1dh^2|y,t)$ is the conditional distribution of $\mathbf{P}$, so that $z\in\tilde{\Xi}_{\varphi}$.
 Then applying Jensen's inequality we get that
\begin{align*}
  \int_0^T\int_{\mathcal{Z}_1\times\mathcal{Z}_2\times\mathcal{Y}}|h|^2\mathbf{P}_t(dh^1dh^2dy)dt\geq & \int_0^T\int_{\mathcal{Y}}\left|\int_{\mathcal{Z}_1\times\mathcal{Z}_2}h\eta(dh^1dh^2|y,t)\right|^2\nu^{x,\mu}(dy)dt\\
  = & \int_0^T\int_{\mathcal{Y}}|z_t(y)|^2\nu^{x,\mu}(dy)dt.
\end{align*}
Then we can deduce that $$I(\varphi)\geq\tilde{I}(\varphi).$$ The claim follows.
\end{proof}

The following lemma gives an explicit representation of infimization problem (\ref{L0}), which plays an essential role in proving the Laplace principle upper bound.
\begin{lemma}
The control $h:=(h^1,h^2):[0,T]\times\mathcal{Y}\to\mathcal{Z}_1\times\mathcal{Z}_2$ defined by
\begin{eqnarray*}
\!\!\!\!\!\!\!\!&&
h^1_t(y):=\sigma^*(\bar{X}_t,\mathcal{L}_{\bar{X}_t})Q_2^{-1}(\bar{X}_t,\mathcal{L}_{\bar{X}_t})(\dot{\varphi}_t-\partial_x \bar{b}(\bar{X}_t,\mathcal{L}_{\bar{X}_t})\varphi_t),
\\
\!\!\!\!\!\!\!\!&&
h^2_t(y):=\sqrt{\gamma}(\partial_y\Phi_g)^*(\bar{X}_t,\mathcal{L}_{\bar{X}_t},y)Q_2^{-1}(\bar{X}_t,\mathcal{L}_{\bar{X}_t})(\dot{\varphi}_t-\partial_x \bar{b}(\bar{X}_t,\mathcal{L}_{\bar{X}_t})\varphi_t),
\end{eqnarray*}
attains the infimum in  (\ref{L0}), where
$$Q_2(\bar{X}_t,\mathcal{L}_{\bar{X}_t}):=\int_{\mathcal{Y}}\sigma\sigma^*(\bar{X}_t,\mathcal{L}_{\bar{X}_t})+\gamma(\partial_y\Phi_g)(\partial_y\Phi_g)^*(\bar{X}_t,\mathcal{L}_{\bar{X}_t},y)\nu^{\bar{X}_t,\mathcal{L}_{\bar{X}_t}}(dy).$$

Furthermore, the infimization problem (\ref{L0})  has the explicit solution
\begin{equation}\label{eqes}
 \tilde{I}(\varphi)=
 \frac{1}{2}\int_0^T|Q_2^{-1/2}(\bar{X}_t,\mathcal{L}_{\bar{X}_t})(\dot{\varphi}_t-\partial_x \bar{b}(\bar{X}_t,\mathcal{L}_{\bar{X}_t})\varphi_t)|^2dt.
\end{equation}
\end{lemma}

\begin{proof}
Note that for any  $z:=(z^1,z^2)\in\tilde{\Xi}_{\varphi}$, we have
\begin{align*}
\dot{\varphi}_t&=\int_{\mathcal{Y}}\Theta(\varphi_t,\bar{X}_t,\mathcal{L}_{\bar{X}_t},y,z^1(y),z^2(y))\nu^{\bar{X}_t,\mathcal{L}_{\bar{X}_t}}(dy) \\
   &=\partial_x \bar{b}(\bar{X}_t,\mathcal{L}_{\bar{X}_t})\varphi_t+\int_{\mathcal{Y}}\left[\sigma(\bar{X}_t,\mathcal{L}_{\bar{X}_t}) z^1(y)+\sqrt{\gamma}(\partial_y\Phi_g)(\bar{X}_t,\mathcal{L}_{\bar{X}_t},y)z^2(y)\right]\nu^{\bar{X}_t,\mathcal{L}_{\bar{X}_t}}(dy).
\end{align*}
Then applying the H\"{o}lder inequality for integrals of matrices (cf.~\cite[Lemma 5.1]{ds}) yields that for any  $z:=(z^1,z^2)\in\tilde{\Xi}_{\varphi}$,
$$\int_0^T\int_{\mathcal{Y}}|z(y)|^2\nu^{\bar{X}_t,\mathcal{L}_{\bar{X}_t}}(dy)dt\geq \int_0^T(\dot{\varphi}_t-\partial_x \bar{b}(\bar{X}_t,\mathcal{L}_{\bar{X}_t})\varphi_t)^*Q_2^{-1}(\bar{X}_t,\mathcal{L}_{\bar{X}_t})(\dot{\varphi}_t-\partial_x \bar{b}(\bar{X}_t,\mathcal{L}_{\bar{X}_t})\varphi_t)dt.$$
Furthermore, for any $t\in[0,T]$, if we take
\begin{eqnarray}
\!\!\!\!\!\!\!\!&&
h^1_t(y):=\sigma^*(\bar{X}_t,\mathcal{L}_{\bar{X}_t})Q_2^{-1}(\dot{\varphi}_t-\partial_x \bar{b}(\bar{X}_t,\mathcal{L}_{\bar{X}_t})\varphi_t),\label{eqh1}
\\
\!\!\!\!\!\!\!\!&&
h^2_t(y):=\sqrt{\gamma}(\partial_y\Phi_g)^*(\bar{X}_t,\mathcal{L}_{\bar{X}_t},y)Q_2^{-1}(\dot{\varphi}_t-\partial_x \bar{b}(\bar{X}_t,\mathcal{L}_{\bar{X}_t})\varphi_t),\label{eqh2}
\end{eqnarray}
then $h:=(h^1,h^2)\in\tilde{\Xi}_{\varphi}$ and
$$\int_0^T\int_{\mathcal{Y}}|h_t(y)|^2\nu^{\bar{X}_t,\mathcal{L}_{\bar{X}_t}}(dy)dt=\int_0^T(\dot{\varphi}_t-\partial_x \bar{b}(\bar{X}_t,\mathcal{L}_{\bar{X}_t})\varphi_t)^*Q_2^{-1}(\bar{X}_t,\mathcal{L}_{\bar{X}_t})(\dot{\varphi}_t-\partial_x \bar{b}(\bar{X}_t,\mathcal{L}_{\bar{X}_t})\varphi_t)dt,$$
which implies that (\ref{eqes}) holds and the infimum of (\ref{L0}) is achieved in $h=(h^1,h^2)$ defined by (\ref{eqh1}) and (\ref{eqh2}).
\end{proof}

%Moreover, continuity of $L_i^0(\bar{X},\mathcal{L}_{\bar{X}},\varphi,\dot{\varphi})$ in $\beta\in\mathbb{R}^d$ is clear whereas continuity with respect to $q$ follows from the fact that  $r(q)$ and $Q(q)$ are continuous with respect to $q$.

Now we have all the ingredients to prove the Laplace principle upper bound and hence to complete the proof of Theorem \ref{thj1}. {We intend to show that for all bounded, continuous functions $\Lambda$ mapping $C([0,T];\mathbb{R}^n )$ into $\mathbb{R}$
\begin{equation}\label{lpup}
\limsup_{\delta\rightarrow0}\left(-\frac{\delta}{\lambda^2(\delta)}\log\mathbb{E}\left[\exp\left\{-\frac{\lambda^2(\delta)}{\delta}\Lambda(Z^\delta)\right\}\right]\right) \leq\inf_{\varphi\in C([0,T];\mathbb{R}^n )}\Big[I(\varphi)+\Lambda(\varphi)\Big]=\inf_{\varphi\in C([0,T];\mathbb{R}^n )}\Big[\tilde{I}(\varphi)+\Lambda(\varphi)\Big].
\end{equation}
Note that for any  $\eta>0$, there exists $\psi\in C([0,T];\mathbb{R}^n )$ with $\psi_0=0$  such that
\begin{equation}\label{e3}
\tilde{I}(\psi)+\Lambda(\psi)\leq\inf_{\varphi\in C([0,T];\mathbb{R}^n )}\Big[\tilde{I}(\varphi)+\Lambda(\varphi)\Big]+\eta<\infty,
\end{equation}
and for each $(z^1,z^2)\in\tilde{\Xi}_{\psi}$,
\begin{equation}\label{hfx1}
\psi_t=\int_0^t\int_{\mathcal{Y}}\Theta(\psi_s,\bar{X}_s,\mathcal{L}_{\bar{X}_s},y,z^1_s(y),z^2_s(y))\nu^{\bar{X}_s,\mathcal{L}_{\bar{X}_s}}(dy)ds.
\end{equation}}
Since $\Lambda$ is bounded, this implies that $I(\psi)<\infty$ and hence $\psi$ is absolutely continuous by the definition of $I$.  For this given function $\psi$, we  define $\bar{h}_t(y):=(\bar{h}^1_t(y),\bar{h}^2_t(y))$, where
$$\bar{h}^1_t(y):=\sigma^*(\bar{X}_t,\mathcal{L}_{\bar{X}_t})Q_2^{-1}(\bar{X}_t,\mathcal{L}_{\bar{X}_t})(\dot{\psi}_t-\partial_x \bar{b}(\bar{X}_t,\mathcal{L}_{\bar{X}_t})\psi_t),$$
$$\bar{h}^2_t(y):=\sqrt{\gamma}(\partial_y\Phi_g)^*(\bar{X}_t,\mathcal{L}_{\bar{X}_t},y)Q_2^{-1}(\bar{X}_t,\mathcal{L}_{\bar{X}_t})(\dot{\psi}_t-\partial_x \bar{b}(\bar{X}_t,\mathcal{L}_{\bar{X}_t})\psi_t),$$
then we have $\bar{h}_{\cdot}(y)\in L^2([0,T];\mathcal{Z}_1\times\mathcal{Z}_2)$ uniformly in $y\in\mathcal{Y}$. Applying a standard mollification argument, without loss of generality,  we can  assume that
\begin{eqnarray}\label{eh2}
\bar{h}~\text{is Lipschitz continuous in}~ t\in[0,T].
\end{eqnarray}
More precisely, let $0\leq \chi\in C^{\infty}_0(\mathbb{R})$ with support contained in $\{r:|r|\leq 1\}$ such that $\int_{\mathbb{R}}\chi(r)dr=1$, and for any $k\geq 1$, let $\chi_k(r):=k\chi(kr)$ and define
$$\bar{h}^k_t(y):=\int_{\mathbb{R}}\bar{h}_r(y)\chi_k(t-r)dr.$$
By the property of convolutions (cf.~e.g.~\cite[(4.26)]{HW2019}), we know that for any $t_1,t_2\in[0,T]$,
$$|\bar{h}^k_{t_1}(y)-\bar{h}^k_{t_2}(y)|\leq c_k|t_1-t_2|,~y\in\mathcal{Y}$$
and
$$\|\bar{h}^k_{\cdot}(y)-\bar{h}_{\cdot}(y)\|_{L^2([0,T];\mathcal{Z}_1\times\mathcal{Z}_2)}\to 0,~k\to \infty,~~\text{uniformly in}~y\in\mathcal{Y}.$$
Moreover, from the conditions $(\mathbf{A1})$-$(\mathbf{A4})$ and (\ref{PHI3}), we can get that
\begin{eqnarray}\label{eh1}
\bar{h}~\text{is  Lipschitz continuous  and bounded in}~y\in\mathcal{Y},
\end{eqnarray}
{the detailed proof is postponed in Subsection \ref{app2} in Appendix.}

Hence, by  (\ref{eh2}) and (\ref{eh1}), we can also deduce that the same properties hold for the function $$\varpi(\cdot,\cdot):=|\bar{h}_{\cdot}(\cdot)|^2:[0,T]\times\mathcal{Y}\to\mathbb{R}.$$

Now we define a control in  feedback form by
$$\bar{h}^\delta_t:=\bar{h}_t(\bar{Y}_t^{\delta})=\big(\bar{h}^1_t(\bar{Y}_t^{\delta}),\bar{h}^2_t(\bar{Y}_t^{\delta})\big).$$
Here $\bar{Y}^{\delta}$ is the solution of the following auxiliary equation
$$
d\bar{Y}_{t}^{\delta}=\frac{1}{\varepsilon}f\left(\bar{X}_{t(\Delta)},\mathcal{L}_{\bar{X}_{t(\Delta)}},\bar{Y}_{t}^{\delta}\right)dt+\frac{1}{\sqrt{\varepsilon}}g\left(\bar{X}_{t(\Delta)},\mathcal{L}_{\bar{X}_{t(\Delta)}},\bar{Y}_{t}^{\delta}\right)dW_t^2,~\bar{Y}_{0}^{\delta}=y,
$$
where $t(\Delta):=[\frac{t}{\Delta}]\Delta$ and $[s]$ denotes the
integer part of $s$. By using Khasminskii's time discretization scheme,  we can show the following convergence
\begin{equation}\label{eqav1}
\lim_{\delta\to 0} \mathbb{E}\int_0^T\varpi(t,\bar{Y}_t^{\delta}) dt=\int_0^T\int_{\mathcal{Y}}\varpi(t,y) \nu^{\bar{X}_{t},\mathcal{L}_{\bar{X}_{t}}}(dy)dt,
\end{equation}
whose proof is postponed in Subsection \ref{secapp3} in Appendix. {Moreover, let ${\psi}\in C([0,T];\mathbb{R}^n )$ be the unique solution to the control problem (\ref{hfx1}) with $\bar{h}_t(y)$, we can deduce that
\begin{eqnarray}\label{e2}
Z^{\delta,\bar{h}^\delta}\Rightarrow\!\!\!\!\!\!\!\!&&~~~~{\psi},~~~~\text{in}~ C([0,T];\mathbb{R}^n),~\text{as}~\delta\to 0,
\end{eqnarray}
whose proof is postponed in Subsection  \ref{secapp4} in Appendix.}
%Indeed, in order to verify (\ref{e2}), the only  additional step is to show that
%\begin{eqnarray*}
%\!\!\!\!\!\!\!\!&&\lim_{\delta\to 0}\frac{\sqrt{\varepsilon}}{\sqrt{\delta}}\mathbb{E}\Bigg\{\sup_{t\in[0,T]}\Bigg|\int^t_0\Big(\partial_y \Phi_g(X_s^{\delta,\bar{h}^\delta} ,\mathcal{L}_{X^{\delta}_s},Y^{\delta,\bar{h}^\delta}_s)\bar{h}^2_s(\bar{Y}_s^{\delta})
%-\partial_y \Phi_g(X_s^{\delta,\bar{h}^\delta} ,\mathcal{L}_{X^{\delta}_s},Y^{\delta,\bar{h}^\delta}_s)\bar{h}^2_s(Y^{\delta,\bar{h}^\delta}_s) \Big)ds\Bigg|\Bigg\}=0,
%\nonumber\\
%\!\!\!\!\!\!\!\!&&\lim_{\delta\to 0} \mathbb{E}\Bigg\{\sup_{t\in[0,T]}\Bigg|\int_0^t\Big(\sigma\big(X_s^{\delta,\bar{h}^\delta} ,\mathcal{L}_{X^{\delta}_s}\big)\bar{h}^1_s(\bar{Y}_s^{\delta})-\sigma\big(X_s^{\delta,\bar{h}^\delta} ,\mathcal{L}_{X^{\delta}_s}\big)\bar{h}^1_s(Y^{\delta,\bar{h}^\delta}_s)\Big) ds\Bigg|\Bigg\}=0,
%\end{eqnarray*}
%which follow from  Lemma 4.6 in \cite{HLLS}.

Now  we can prove (\ref{lpup}). By (\ref{eqes}), (\ref{e3}), (\ref{eqav1}) and (\ref{e2}), we have
\begin{align}
 &\limsup_{\delta\rightarrow0}\left(-\frac{\delta}{\lambda^2(\delta)}\log\mathbb{E}\left[\exp\left\{-\frac{\lambda^2(\delta)}{\delta}\Lambda(Z^\delta)\right\}\right]\right)\nonumber\\ &=\limsup_{\delta\rightarrow0}\inf_{h\in \mathcal{A}}\mathbb{E}\left[\frac{1}{2}\int_0^T|h_s|^2ds+\Lambda(Z^{\delta,h})\right] \nonumber\\
 &\leq\limsup_{\delta\rightarrow0}\mathbb{E}\left[\frac{1}{2}\int_0^T|\bar{h}^\delta_s|^2ds+\Lambda(Z^{\delta,\bar{h}^\delta})\right]\nonumber\\
 &=\mathbb{E}\left[\frac{1}{2}\int_0^T\int_{\mathcal{Y}}|\bar{h}_s(y)|^2\nu^{\bar{X}_s,\mathcal{L}_{\bar{X}_s}}(dy)ds+\Lambda(\psi)\right]\nonumber\\
% &=I(\bar{Z})+\Lambda(\bar{Z})\nonumber\\
 &=I(\psi)+\Lambda(\psi)\nonumber\\
 &\leq\inf_{\varphi\in C([0,T];\mathbb{R}^n )}\big[I(\varphi)+\Lambda(\varphi)\big]+\eta.
\end{align}
Since $\eta$ is arbitrary, we complete the proof of the Laplace principle upper bound.   \hspace{\fill}$\Box$

\vspace{5mm}
{
\textbf{Conclusions and Future Work:}
In this paper,  we have derived the moderate deviation principle for McKean-Vlasov stochastic differential equations with a separation of fast and slow components and small noise in the slow component. Depending on the interaction of the fast scale with the intensity of the noise, the rate functions are different in two regimes. In particular, we show that it is strongly affected by the noise of the fast component in Regime 2, which is essentially different from the Regime 1. As a by-product, the explicit representation formulas of the rate functions in all of regimes are also given.  Our
strategies here are mainly based on the weak convergence approach developed by Budhiraja et al.~\cite{BDM} and the functional occupation measure approach developed by Dupuis and Spiliopoulos~\cite{ds}.  It is worth noting that {the above methods does not work in the case of Regime 3 anymore,  as such a scale relationship makes} the analysis very challenging and it is also an open issue even in the distribution independence case.

An interesting extension of this work would be to proving the large and moderate deviation principle in the
setting where the coefficients of fast and slow components depend also on the law of the fast process, i.e.  $\mathcal{L}_{Y^{\delta}_t}$.   Secondly,  as evidenced in \cite{ms}, one would also consider a more general system by adding a homogenization term $\frac{\sqrt{\delta}}{\sqrt{\varepsilon}}K(X^{\delta}_t,\mathcal{L}_{X^{\delta}_t},Y^{\delta}_t)$
in the drift of the slow equation, and it is sufficient to apply the weak convergence rate to derive the corresponding LDP and MDP. On the other hand, it is interesting to apply the results of the current paper to establish  the importance sampling scheme for the multi-scale McKean-Vlasov SDEs.
}

\section{Appendix}

 \subsection{The definition of Lions derivative on $\mathcal{P}_2$}\label{appendix1}
A matrix-valued function $u(x,y)$ defined on $\RR^n\times\RR^m$, we use $\partial_x u$ and $\partial_y u$ to denote the first order partial derivative of $u$ with respect to (w.r.t.)  $x$ and $y$ respectively, $\partial^2_{xx} u$, $\partial^2_{yy} u$ and $\partial^2_{xy} u$ to denote its second order derivatives of $u$.

From \cite[Section 6]{C}, for $u: \mathcal{P}_2\rightarrow \RR$ we denote by $U$ its ``extension" to $L^2(\Omega, \PP;\RR^n)$ defined by
$$
U(X)=u(\mathcal{L}_{X}),\quad X\in L^2(\Omega,\PP;\RR^n).
$$
We call $u$ is differentiable at $\mu\in\mathcal{P}_2$ if there exists $X\in L^2(\Omega,\PP;\RR^n)$ such that $\mathcal{L}_{X}=\mu$ and $U$ is Fr\'echet differentiable at $X$. Applying the Riesz theorem, the Fr\'echet derivative $DU(X)$, viewed as an element of $L^2(\Omega,\PP;\RR^n)$, can be represented as
\begin{equation}\label{esfre}
DU(X)=\partial_{\mu}u(\mathcal{L}_{X})(X),
\end{equation}
where $\partial_{\mu}u(\mathcal{L}_{X}):\RR^n\rightarrow \RR^n$, which is called Lions derivative of $u$ at $\mu= \mathcal{L}_{X}$. We note that the map $\partial_{\mu}u(\mathcal{L}_{X})$ depends only on $\mathcal{L}_X$, not on $X$.
In addition, $\partial_{\mu}u(\mu)\in L^2(\mu;\RR^n)$ for $\mu\in\mathcal{P}_2$. If $\partial_{\mu}u(\mu)(z):\RR^n\rightarrow \RR^n$ is differentiable at $z\in\RR^n$, we denote its derivative by $\partial_{z}\partial_{\mu}u(\mu)(z):\RR^n\rightarrow \RR^n\times\RR^n$.

\vspace{0.1cm}
We say that a matrix-valued function $u(\mu)=(u_{ij}(\mu))$ differentiable at $\mu\in\mathcal{P}_2$, if all its components are  differentiable at $\mu$, and set $$ \partial_{\mu}u(\mu)=(\partial_{\mu}u_{ij}(\mu)), ~~   \|\partial_{\mu}u(\mu)\|^2_{L^2(\mu)}=\sum_{i,j}\int_{\RR^n}|\partial_{\mu}u_{ij}(\mu)(z)|^2\mu(dz). $$
In addition, we call $\partial_{\mu}u(\mu)(z)$ differentiable at $z\in\RR^n$, if all its components are differentiable at $z$, and set $$\partial_{z}\partial_{\mu}u(\mu)(z)=(\partial_{z}\partial_{\mu}u_{ij}(\mu)(z)), ~~ \|\partial_{z}\partial_{\mu}u(\mu)\|^2_{L^2(\mu)}=\sum_{i,j}\int_{\RR^n}\|\partial_{z}\partial_{\mu}u_{ij}(\mu)(z)\|^2\mu(dz).$$

\vspace{0.1cm}
For the reader's convenience, we recall the following definitions.

 \begin{definition} For a  map $u(\cdot): \mathcal{P}_2 \to \RR$, we say $u\in C^{(1,1)}(\mathcal{P}_2; \RR)$, if this map is continuously differentiable at any $\mu\in\mathcal{P}_2$ and its derivative $\partial_{\mu}u(\mu)(z):\RR^n\rightarrow \RR^n$ is continuously differentiable at any $z\in\RR^n$. We say $u\in C^{(1,1)}_b(\mathcal{P}_2; \RR)$, if $u\in C^{(1,1)}(\mathcal{P}_2; \RR)$, moreover the derivatives $\partial_{\mu}u(\mu)(z)$ and $\partial_z\partial_{\mu}u(\mu)(z)$ are jointly continuous at any $(\mu,z)$, and uniformly bounded w.r.t.~$(\mu,z)$,~i.e., $\sup_{\mu\in\mathcal{P}_2,z\in\RR^n}|\partial_{\mu}u(\mu)(z)|<\infty$ and $\sup_{\mu\in\mathcal{P}_2,z\in\RR^n}\|\partial_z\partial_{\mu}u(\mu)(z)\|<\infty$. For a matrix-valued map $u(\cdot): \mathcal{P}_2 \to \RR^{l_1}\otimes \RR^{l_2}$, where $l_1,l_2\in \mathbb{N}_{+}$, we say $u\in C^{(1,1)}(\mathcal{P}_2;\RR^{l_1}\otimes \RR^{l_2})$ (resp. $C^{(1,1)}_b(\mathcal{P}_2;\RR^{l_1}\otimes \RR^{l_2})$) if all the components belong to $C^{(1,1)}(\mathcal{P}_2;\RR)$ (resp. $C^{(1,1)}_b(\mathcal{P}_2;\RR)$).
\end{definition}

\begin{definition} For a  map $u(\cdot): \RR^n \to \RR$, we say $u\in C^{2}_b(\RR^n; \RR)$, if the derivatives $\partial_{x} u(x)$,  $\partial^2_{xx} u(x)$ are bounded and continuous at any $x$. For a  map $u(\cdot,\cdot): \RR^n\times\RR^m \to \RR$, we say $u\in C^{2,2}(\RR^n\times\RR^m; \RR)$, if the partial derivatives $\partial_{x} u(x,y)$, $\partial_{y} u(x,y)$, $\partial^2_{xx} u(x,y)$, $\partial^2_{xy} u(x,y)$ and $\partial^2_{yy} u(x,y)$ exist at any $(x,y)$. We say $u\in C^{2,2}_b(\RR^n\times\RR^m; \RR)$, if $u\in C^{2,2}(\RR^n\times\RR^m; \RR)$ and the partial derivatives $\partial_{x} u(x,y)$, $\partial_{y} u(x,y)$, $\partial^2_{xx} u(x,y)$, $\partial^2_{xy} u(x,y)$ and $\partial^2_{yy} u(x,y)$ are jointly continuous at any $(x,y)$ and uniformly bounded w.r.t.  $(x,y)$. For a matrix-valued map $u(\cdot): \RR^n \to \RR^{l_1}\otimes \RR^{l_2}$, we say $u\in C^{2}_b(\RR^n;\RR^{l_1}\otimes \RR^{l_2})$ if all the components belong to $C^{2}_b(\RR^n;\RR)$. Similarly, we say $u\in C^{2,2}(\RR^n\times\RR^m;\RR^{l_1}\otimes \RR^{l_2})$ (resp. $C^{2,2}_b(\RR^n\times\RR^m;\RR^{l_1}\otimes \RR^{l_2})$) if all the components belong to $C^{2,2}(\RR^n\times\RR^m;\RR)$ (resp. $C^{2,2}_b(\RR^n\times\RR^m;\RR)$).
%The natation $C^{2,2}(\RR^n\times\RR^m;\RR^{l_1}\otimes \RR^{l_2})$ and $C^{2,2}_b(\RR^n\times\RR^m;\RR^{l_1}\otimes \RR^{l_2})$ can be defined similarly.

%For some $\gamma\in(0,1]$, we say $F\in C^{2,2+\gamma}_b(\RR^n\times\RR^m, \RR)$, if $F\in C^{2,2}_b(\RR^n\times\RR^m, \RR)$ and the all the second partial derivatives $\partial^2_{xx} F(x,y)$, $\partial^2_{xy} F(x,y)$ and $\partial^2_{yy} F(x,y)$ are $\gamma$ H\"{o}lder continuous with respective to $y$, i.e., there exists $C>0$ such that for any $u, v\in \{x,y\}$,
%\begin{eqnarray*}
%\sup_{x\in\RR^n}\|\partial^2_{uv} F(x, y_1)-\partial^2_{uv} F(x,y_2)\|\leq C|y_1-y_2|^{\gamma_2}.
%\end{eqnarray*}
\end{definition}

\begin{definition}\label{de1}
For a matrix-valued map $u(\cdot,\cdot): \RR^n\times\mathcal{P}_2  \to \RR^{l_1}\otimes \RR^{l_2}$, we say $u\in C^{2,(1,1)}_b(\RR^n\times\mathcal{P}_2;\RR^{l_1}\otimes \RR^{l_2})$ if $u(x,\cdot)\in C^{(1,1)}_b(\mathcal{P}_2;\RR^{l_1}\otimes \RR^{l_2})$ for any $x\in\RR^n$ and $u(\cdot,\mu)\in C^{2}_b(\RR^n;\RR^{l_1}\otimes \RR^{l_2})$ for any $\mu\in\mathcal{P}_2$, moreover, the derivatives $\partial_{x} u(x,\mu)$, $\partial^2_{xx} u(x,\mu)$,  $\partial_{\mu} u(x,\mu)(z)$, $\partial_{z}\partial_{\mu} u(x,\mu)(z)$ are jointly continuous at any $(x,\mu,z)$ and uniformly bounded w.r.t.~$(x,\mu,z)$.
\end{definition}

\begin{definition}\label{de2} For a matrix-valued map $u(\cdot,\cdot,\cdot): \RR^n\times\mathcal{P}_2 \times\RR^m \to \RR^{l_1}\otimes \RR^{l_2}$, we say $u\in C^{2,(1,1),2}(\RR^n\times\mathcal{P}_2\times\RR^m;\RR^{l_1}\otimes \RR^{l_2})$ if $u(x,\cdot,y)\in C^{(1,1)}(\mathcal{P}_2;\RR^{l_1}\otimes \RR^{l_2})$ for any $(x,y)\in\RR^n\times \RR^m$ and $u(\cdot,\mu,\cdot)\in C^{2,2}(\RR^n\times\RR^m;\RR^{l_1}\otimes \RR^{l_2})$ for any $\mu\in\mathcal{P}_2$. We say $u\in C^{2,(1,1),2}_b(\RR^n\times\mathcal{P}_2\times\RR^m;\RR^{l_1}\otimes \RR^{l_2})$, if $u\in C^{2,(1,1),2}(\RR^n\times\mathcal{P}_2\times\RR^m;\RR^{l_1}\otimes \RR^{l_2})$, moreover the partial derivatives $\partial_{x} u(x,\mu,y)$, $\partial_{y} u(x,\mu,y)$, $\partial^2_{xx} u(x,\mu,y)$, $\partial^2_{xy} u(x,\mu,y)$, $\partial^2_{yy} u(x,\mu,y)$, $\partial_{\mu}u(x,\mu,y)(z)$ and $\partial_z\partial_{\mu}u(x,\mu,y)(z)$ are uniformly bounded w.r.t.  $(x,\mu,y,z)$ and uniformly continuous on $\RR^n\times\mathcal{P}_2 \times\RR^m\times\RR^n$.
\end{definition}

\subsection{Proof of (\ref{ess16})}\label{appendix2}

%\begin{equation*}
%\left\{ \begin{aligned}
%dZ_t^{\delta,h^\delta}=&\frac{1}{\lambda(\delta)}\Big[b\big(\lambda(\delta)Z_t^{\delta,h^\delta} +\bar{X}_t ,\mathcal{L}_{X^{\delta}_t},Y^{\delta,h^\delta}_t\big)-\bar{b}\big(\lambda(\delta)Z_t^{\delta,h^\delta} +\bar{X}_t ,\mathcal{L}_{X^{\delta}_t}\big)\Big]dt\\
%&
%+\frac{1}{\lambda(\delta)}\Big[\bar{b}\big(\lambda(\delta)Z_t^{\delta,h^\delta} +\bar{X}_t ,\mathcal{L}_{X^{\delta}_t}\big)-\bar{b}(\bar{X}_t,\mathcal{L}_{X^{\delta}_t})\Big]dt\\
%&+\frac{1}{\lambda(\delta)}\Big[\bar{b}(\bar{X}_t,\mathcal{L}_{X^{\delta}_t})-\bar{b}(\bar{X}_t,\mathcal{L}_{\bar{X}_t})\Big]dt\\
%&+\frac{\sqrt{\delta}}{\lambda(\delta)}\sigma\big(\lambda(\delta)Z_t^{\delta,h^\delta} +\bar{X}_t ,\mathcal{L}_{X^{\delta}_t}\big)dW_t^1\\
%&+\sigma\big(\lambda(\delta)Z_t^{\delta,h^\delta} +\bar{X}_t ,\mathcal{L}_{X^{\delta}_t}\big)h^{1,\delta}_t dt,\\
%Z_0^{\delta,h^\delta}=0,&
%\end{aligned} \right.
%\end{equation*}
First, recall (\ref{esZ}) we can get
\begin{eqnarray}\label{esa0}
\!\!\!\!\!\!\!\!&&\mathbb{E}\Big[\sup_{t\in[0,T]}|Z_t^{\delta,h^\delta}|^2\Big]
\nonumber\\
\leq~~\!\!\!\!\!\!\!\!&& \frac{1}{\lambda(\delta)^2}\mathbb{E}\Big[\sup_{t\in[0,T]}\Big|\int_0^tb\big(\lambda(\delta)Z_s^{\delta,h^\delta} +\bar{X}_s ,\mathcal{L}_{X^{\delta}_s},Y^{\delta,h^\delta}_s\big)
\nonumber\\
\!\!\!\!\!\!\!\!&&
-\bar{b}\big(\lambda(\delta)Z_s^{\delta,h^\delta} +\bar{X}_s ,\mathcal{L}_{X^{\delta}_s}\big)ds\Big|^2\Big]
\nonumber\\
\!\!\!\!\!\!\!\!&&+\frac{1}{\lambda(\delta)^2}\mathbb{E}\Big[\sup_{t\in[0,T]}\Big|\int_0^t\bar{b}\big(\lambda(\delta)Z_s^{\delta,h^\delta} +\bar{X}_s ,\mathcal{L}_{X^{\delta}_s}\big)-\bar{b}(\bar{X}_s,\mathcal{L}_{X^{\delta}_s})ds\Big|^2\Big]
\nonumber\\
\!\!\!\!\!\!\!\!&&+\frac{1}{\lambda(\delta)^2}\mathbb{E}\Big[\sup_{t\in[0,T]}\Big|\int_0^t\bar{b}(\bar{X}_s,\mathcal{L}_{X^{\delta}_s})-\bar{b}(\bar{X}_s,\mathcal{L}_{\bar{X}_s})ds\Big|^2\Big]
\nonumber\\
\!\!\!\!\!\!\!\!&&+\frac{\delta}{\lambda(\delta)^2}\mathbb{E}\Big[\sup_{t\in[0,T]}\Big|\int_0^t\sigma\big(\lambda(\delta)Z_s^{\delta,h^\delta} +\bar{X}_s ,\mathcal{L}_{X^{\delta}_s}\big)dW_s^1\Big|^2\Big]
\nonumber\\
\!\!\!\!\!\!\!\!&&+\mathbb{E}\Big[\sup_{t\in[0,T]}\Big|\int_0^t\sigma\big(\lambda(\delta)Z_s^{\delta,h^\delta} +\bar{X}_s ,\mathcal{L}_{X^{\delta}_s}\big)h^{1,\delta}_s\Big|^2\Big]
\nonumber\\
=:~~\!\!\!\!\!\!\!\!&&\text{I}+\text{II}+\text{III}+\text{IV}+\text{V}.
\end{eqnarray}
For $\text{II}$, we use the Lipschitz continuity of $\bar{b}$ to get
\begin{eqnarray}\label{esa1}
\text{II}\leq \mathbb{E}\int_0^T|Z_t^{\delta,h^\delta}|^2dt.
\end{eqnarray}
For $\text{III}$, we recall the convergence order result (\ref{33}), thus
\begin{eqnarray}\label{esa2}
\text{III}\leq\frac{C_T}{\lambda(\delta)^2}\sup_{t\in[0,T]}\mathbb{W}_2(\mathcal{L}_{X^{\delta}_t},\mathcal{L}_{\bar{X}_t})^2\leq \frac{C_T(\varepsilon+\delta)}{\lambda(\delta)^2}\to 0,~\text{as}~\delta\to0.
\end{eqnarray}
For $\text{IV}$, by Burkholder-Davis-Gundy's inequality, we obtain
\begin{eqnarray}\label{esa3}
\text{IV}\leq\frac{C\delta}{\lambda(\delta)^2}\mathbb{E}\int_0^T\|\sigma\big(\lambda(\delta)Z_s^{\delta,h^\delta} +\bar{X}_s ,\mathcal{L}_{X^{\delta}_s}\big)\|^2ds
\leq\frac{C_T\delta}{\lambda(\delta)^2}.
\end{eqnarray}
For $\text{V}$, by H\"{o}lder's inequality, we have
\begin{eqnarray}\label{esa4}
\text{V}\leq\mathbb{E}\Big[\int_0^T\|\sigma\big(\lambda(\delta)Z_s^{\delta,h^\delta} +\bar{X}_s ,\mathcal{L}_{X^{\delta}_s}\big)\|^2ds\cdot\int_0^T|h^\delta_s|^2ds\Big]
\leq C_{M,T}.
\end{eqnarray}
Applying Gronwall's inequality in (\ref{esa0}), it leads to
\begin{eqnarray*}
\mathbb{E}\Big[\sup_{t\in[0,T]}|Z_t^{\delta,h^\delta}|^2\Big]\leq \Big(\text{I}+\frac{C_T(\varepsilon+\delta)}{\lambda(\delta)^2}+C_{M,T}\Big)e^{C_{M,T}\lambda(\delta)}.
\end{eqnarray*}
Recall the fact that $\lambda(\delta)\to 0,\frac{\delta}{\lambda(\delta)^2}\to 0$ as $\delta\to 0$. Thus for (\ref{ess16}), it is sufficient to prove
$$\text{I}\leq C_{M,T}(1+|x|^6+|y|^6).$$
To prove this estimate, we recall the Poisson equation (\ref{PE}) and get
\begin{eqnarray*}
\text{I}=~~\!\!\!\!\!\!\!\!&&\frac{1}{\lambda(\delta)^2}\mathbb{E}\Bigg[\sup_{t\in[0,T]}\Big|-\int^t_0 \mathbf{L}^{2}_{X_{s}^{\delta,h^\delta},\mathcal{L}_{X^{\delta}_{s}}}\Phi(X_{s}^{\delta,h^\delta},\mathcal{L}_{X^{\delta}_{s}},Y^{\delta,h^\delta}_{s})ds\Big|^2\Bigg]
\nonumber\\
=~~\!\!\!\!\!\!\!\!&&\frac{\varepsilon^2}{\lambda(\delta)^2}\mathbb{E}\Bigg[\sup_{t\in[0,T]}\Big|-\Phi(X_t^{\delta,h^\delta} ,\mathcal{L}_{X^{\delta}_t},Y^{\delta,h^\delta}_t)+\Phi(x,\delta_x,y)
\nonumber\\
\!\!\!\!\!\!\!\!&&+\int^t_0 \EE\left[b(X^{\delta}_s,\mathcal{L}_{ X^{\delta}_{s}}, Y^{\delta}_s)\partial_{\mu}\Phi(x,\mathcal{L}_{X^{\delta}_{s}},y)(X^{\delta}_s)\right]\Big|_{x=X_s^{\delta,h^\delta} ,y=Y^{\delta,h^\delta}_{s}}ds\nonumber\\
\!\!\!\!\!\!\!\!&&+\int^t_0 {\frac{\delta}{2}}\EE \text{Tr}\left[\sigma\sigma^{*}(X^{\delta}_s,\mathcal{L}_{ X^{\delta}_{s}})\partial_z\partial_{\mu}\Phi(x,\mathcal{L}_{X^{\delta}_{s}},y)(X^{\delta}_s)\right]\Big|_{x=X_s^{\delta,h^\delta} ,y=Y^{\delta,h^\delta}_{s}}ds\nonumber\\
\!\!\!\!\!\!\!\!&&+\int^t_0 \mathbf{L}^{1,\delta}_{\mathcal{L}_{X^{\delta}_{s}},Y^{\delta,h^\delta}_{s}}\Phi(X_s^{\delta,h^\delta},\mathcal{L}_{X^{\delta}_{s}},Y^{\delta,h^\delta}_{s})ds
+M^{1,\delta,h^\delta}_t\Big|^2\Bigg]
\nonumber\\
\!\!\!\!\!\!\!\!&&+\varepsilon^2\mathbb{E}\Bigg[\sup_{t\in[0,T]}\Big|\int^t_0 \partial_x \Phi(X_s^{\delta,h^\delta} ,\mathcal{L}_{X^{\delta}_s},Y^{\delta,h^\delta}_s)\cdot \sigma(X^{\delta,h^\delta}_s,\mathcal{L}_{X^{\delta}_s})h^{1,\delta}_s ds\Big|^2\Bigg]
\nonumber\\
\!\!\!\!\!\!\!\!&&+\frac{\varepsilon}{\delta}\mathbb{E}\Bigg[\sup_{t\in[0,T]}\Big|\int^t_0\partial_y \Phi(X_s^{\delta,h^\delta} ,\mathcal{L}_{X^{\delta}_s},Y^{\delta,h^\delta}_s)\cdot g(X^{\delta,h^\delta}_s,\mathcal{L}_{X^{\delta}_s},Y^{\delta,h^\delta}_s)h^{2,\delta}_s ds\Big|^2\Bigg]
\nonumber\\
\!\!\!\!\!\!\!\!&&
+\frac{\varepsilon}{\delta}\mathbb{E}\Big[\sup_{t\in[0,T]}|M^{2,\delta,h^\delta}_t|^2\Big]
\end{eqnarray*}
Collecting the calculations of Steps 2 and 3 in Proposition \ref{p5}, we can get that
\begin{eqnarray*}
\text{I}\leq C_{M,T}\Big(\frac{\varepsilon}{\lambda(\delta)^2}+\frac{\varepsilon}{\delta}\Big)(1+|x|^6+|y|^6).
\end{eqnarray*}
By the assumptions of Theorems \ref{t3} and \ref{thj1}, we can obtain the desired estimate. We complete the proof.

{\subsection{Proof of (\ref{eh1})}\label{app2}
Recall
$$Q_2(x,\mu)=\int_{\mathcal{Y}}\sigma\sigma^*(x,\mu)+\gamma(\partial_y\Phi_g)(\partial_y\Phi_g)^*(x,\mu,y)\nu^{x,\mu}(dy).$$
By the assumptions $(\mathbf{A1})$ and $(\mathbf{A4})$, we know that $\sigma\sigma^*$ is bounded and Lipschitz continuous. {By the regularity }of $\Phi$ and the assumption $(\mathbf{A3})$ on $g$, we also obtain the boundedness and Lipschitz continuity of $\partial_y\Phi_g$.
Thus $Q_2$ is bounded and  Lipschitz continuous by \cite[Section A.2]{RSX1}, which implies the boundedness of {inverse  matrix} $Q_2^{-1}$.

\vspace{1mm}
Note that we have used the mollification argument w.r.t.~the time variable $t$ in $\bar{h}$, thus we can deduce that
$\bar{h}~\text{is  bounded in}~y\in\mathcal{Y}.$ From the same reason, we can conclude the Lipschitz continuity in $y$ of $\bar{h}$. }

\subsection{Proof of (\ref{eqav1})}\label{secapp3}
Notice that
\begin{eqnarray}\label{15}
\!\!\!\!\!\!\!\!&&\int_0^T\varpi(t,\bar{Y}_t^{\delta}) dt-\int_0^T\int_{\mathcal{Y}}\varpi(t,y) \nu^{\bar{X}_{t},\mathcal{L}_{\bar{X}_{t}}}(dy)dt
\nonumber \\
=~~\!\!\!\!\!\!\!\!&&\int_0^T\varpi(t,\bar{Y}_t^{\delta})dt-\int_0^T\varpi(t(\Delta),\bar{Y}_t^{\delta}) dt
\nonumber \\
 \!\!\!\!\!\!\!\!&&+ \int_0^T\varpi(t(\Delta),\bar{Y}_t^{\delta}) dt-\int_0^T\int_{\mathcal{Y}}\varpi(t(\Delta),y) \nu^{\bar{X}_{t(\Delta)},\mathcal{L}_{\bar{X}_{t(\Delta)}}}(dy)dt
 \nonumber \\
 \!\!\!\!\!\!\!\!&& + \int_0^T\int_{\mathcal{Y}}\varpi(t(\Delta),y) \nu^{\bar{X}_{t(\Delta)},\mathcal{L}_{\bar{X}_{t(\Delta)}}}(dy)dt-\int_0^T\int_{\mathcal{Y}}\varpi(t,y) \nu^{\bar{X}_{t},\mathcal{L}_{\bar{X}_{t}}}(dy)dt
 \nonumber \\
=:~~\!\!\!\!\!\!\!\!&&\mathcal{I}_1(T)+\mathcal{I}_2(T)+\mathcal{I}_3(T).
\end{eqnarray}
First, by the Lipschitz continuity of $\varpi$ and \cite[(A.2)]{RSX1}, we can get
\begin{eqnarray}
\EE|\mathcal{I}_{1}(T)+\mathcal{I}_{3}(T)|^2
\leq
C_T\Delta^2.
\label{p6}
\end{eqnarray}
We now deal with the term $\mathcal{I}_2(t)$. Note that
\begin{eqnarray}
 |\mathcal{I}_2(T)|^2=~~\!\!\!\!\!\!\!\!&&\left|\sum_{k=0}^{[T/\Delta]-1} \int_{k\Delta} ^{(k+1)\Delta} \Big(\varpi(t(\Delta),\bar{Y}_t^{\delta})-\int_{\mathcal{Y}}\varpi(t(\Delta),y) \nu^{\bar{X}_{t(\Delta)},\mathcal{L}_{\bar{X}_{t(\Delta)}}}(dy)\Big) dt\right.\nonumber \\
 \!\!\!\!\!\!\!\!&& +\left.\int_{T(\Delta)} ^{T} \Big(\varpi(t(\Delta),\bar{Y}_t^{\delta})-\int_{\mathcal{Y}}\varpi(t(\Delta),y) \nu^{\bar{X}_{t(\Delta)},\mathcal{L}_{\bar{X}_{t(\Delta)}}}(dy)\Big) dt\right|^2\nonumber \\
 \leq~~\!\!\!\!\!\!\!\!&&\frac{C_T}{\Delta}\sum_{k=0}^{[T/\Delta]-1}\left|\int_{k\Delta} ^{(k+1)\Delta} \Big(\varpi(t(\Delta),\bar{Y}_t^{\delta})-\int_{\mathcal{Y}}\varpi(t(\Delta),y) \nu^{\bar{X}_{t(\Delta)},\mathcal{L}_{\bar{X}_{t(\Delta)}}}(dy)\Big)dt\right|^2\nonumber \\
 \!\!\!\!\!\!\!\!&& +2\left|\int_{T(\Delta)} ^{T} \Big(\varpi(t(\Delta),\bar{Y}_t^{\delta})-\int_{\mathcal{Y}}\varpi(t(\Delta),y) \nu^{\bar{X}_{t(\Delta)},\mathcal{L}_{\bar{X}_{t(\Delta)}}}(dy)\Big) dt\right|^2\nonumber \\
=:~~\!\!\!\!\!\!\!\!&&\mathcal{O}_{1}(t)+\mathcal{O}_{2}(t).  \label{p12}
\end{eqnarray}
By the boundedness of $\varpi$,  it follows that
\begin{eqnarray}
\EE\mathcal{O}_{2}(t)
 \leq C_T\Delta^2.\label{p8}
\end{eqnarray}
Repeating the same argument as in Proposition 4.2 in \cite{HLLS}, we can see that
\begin{eqnarray}
\EE\mathcal{O}_{1}(t)
\leq C_T\frac{\varepsilon}{\Delta}.\label{w3}
\end{eqnarray}
{By the assumptions $\frac{\lambda^2}{\Delta}\rightarrow0, \frac{\delta}{\lambda^2}\rightarrow0$ and $\frac{\varepsilon}{\delta}\rightarrow\gamma$, we have $\frac{\varepsilon}{\Delta}\rightarrow0$. Collecting estimates (\ref{15})-(\ref{w3}), we conclude that (\ref{eqav1}) holds.
  \hspace{\fill}$\Box$

\subsection{Proof of (\ref{e2})}\label{secapp4}
Recall that
\begin{eqnarray}
Z_t^{\delta,\bar{h}^\delta}=~~\!\!\!\!\!\!\!\!&&\frac{1}{\lambda(\delta)}\int_0^tb\big(\lambda(\delta)Z_s^{\delta,\bar{h}^\delta} +\bar{X}_s ,\mathcal{L}_{X^{\delta}_s},Y^{\delta,\bar{h}^\delta}_s\big)-\bar{b}\big(\lambda(\delta)Z_s^{\delta,\bar{h}^\delta} +\bar{X}_s ,\mathcal{L}_{X^{\delta}_s}\big)ds
\nonumber\\
\!\!\!\!\!\!\!\!&&
+\frac{1}{\lambda(\delta)}\int_0^t\bar{b}\big(\lambda(\delta)Z_s^{\delta,\bar{h}^\delta} +\bar{X}_s ,\mathcal{L}_{X^{\delta}_s}\big)-\bar{b}(\bar{X}_s,\mathcal{L}_{X^{\delta}_s})ds
\nonumber\\
\!\!\!\!\!\!\!\!&&+\frac{1}{\lambda(\delta)}\int_0^t\bar{b}(\bar{X}_s,\mathcal{L}_{X^{\delta}_s})-\bar{b}(\bar{X}_s,\mathcal{L}_{\bar{X}_s})ds
\nonumber\\
\!\!\!\!\!\!\!\!&&+\frac{\sqrt{\delta}}{\lambda(\delta)}\int_0^t\sigma\big(\lambda(\delta)Z_s^{\delta,\bar{h}^\delta} +\bar{X}_s ,\mathcal{L}_{X^{\delta}_s}\big)dW_s^1
\nonumber\\
\!\!\!\!\!\!\!\!&&+\int_0^t\sigma\big(\lambda(\delta)Z_s^{\delta,\bar{h}^\delta} +\bar{X}_s ,\mathcal{L}_{X^{\delta}_s}\big)\bar{h}^1_s(\bar{Y}_s^{\delta})ds,
\nonumber\\
=:~~\!\!\!\!\!\!\!\!&&\sum_{i=1}^5\mathcal{K}_i^{\delta}(t)\nonumber
\end{eqnarray}
and
\begin{align*}
{\psi}_t=&\int_0^t\partial_x\bar{b}(\bar{X}_s,\mathcal{L}_{\bar{X}_s})\cdot{\psi}_sds+\int_0^t\int_{\mathcal{Y}}\sigma(\bar{X}_s,\mathcal{L}_{\bar{X}_s})\bar{h}^1_s(y)\nu^{\bar{X}_s,\mathcal{L}_{\bar{X}_s}}(dy)ds \\
 & +\int_0^t\int_{\mathcal{Y}}\sqrt{\gamma}\partial_y\Phi_ g(\bar{X}_s,\mathcal{L}_{\bar{X}_s},y)\bar{h}^2_s(y)\nu^{\bar{X}_s,\mathcal{L}_{\bar{X}_s}}(dy)ds.
\end{align*}
It follows that
 \begin{align*}
Z_t^{\delta,h^\delta}-{\psi}_t=&~\mathcal{K}_1^{\delta}(t)-\int_0^t\int_{\mathcal{Y}}\sqrt{\gamma}\partial_y\Phi_ g(\bar{X}_s,\mathcal{L}_{\bar{X}_s},y)\bar{h}^2_s(y)\nu^{\bar{X}_s,\mathcal{L}_{\bar{X}_s}}(dy)ds\\
&+\mathcal{K}_2^{\delta}(t)-\int_0^t\partial_x\bar{b}(\bar{X}_s,\mathcal{L}_{\bar{X}_s})\cdot{\psi}_sds\\
&+\mathcal{K}_3^{\delta}(t)+\mathcal{K}_4^{\delta}(t)\\
&+\mathcal{K}_5^{\delta}(t)-\int_0^t\int_{\mathcal{Y}}\sigma(\bar{X}_s,\mathcal{L}_{\bar{X}_s})\bar{h}^1_s(y)\nu^{\bar{X}_s,\mathcal{L}_{\bar{X}_s}}(dy)ds.
\end{align*}

Notice that in light of (\ref{esa2}) and (\ref{esa3}), we know that the terms $\mathcal{K}_3^{\delta}(t)$ and $\mathcal{K}_4^{\delta}(t)$ vanish in probability in $C([0,T];\mathbb{R}^n)$, as $\delta\rightarrow0$. Moreover, using a same argument as  in (\ref{ess12}), we get
\begin{align}\label{hfxj1}
  & \mathbb{E} \sup_{t\in[0,T]} \Bigg|\mathcal{K}_2^{\delta}(t)-\int_0^t\partial_x\bar{b}(\bar{X}_s,\mathcal{L}_{\bar{X}_s})\cdot{\psi}_sds\Bigg|\nonumber\\
\leq&~ C_{T,M}(1+|x|^3+|y|^3)\int_0^T\int_0^1\Big(\mathbb{E}\|\partial_x\bar{b}\big(r\lambda(\delta)Z_s^{\delta,\bar{h}^\delta} +\bar{X}_s ,\mathcal{L}_{X^{\delta}_s}\big)-\partial_x\bar{b}(\bar{X}_s,\mathcal{L}_{\bar{X}_s})\|^2\Big)^{\frac{1}{2}}drds\nonumber\\
&+C\mathbb{E}\Bigg(\int_0^T|Z_s^{\delta,\bar{h}^\delta}-\psi_s|ds\Bigg).
\end{align}

 Thus, it suffices to prove the following convergence
\begin{equation}\label{hfx8}
 \lim_{\delta\rightarrow0}\mathbb{E}\sup_{t\in[0,T]}\Big|\mathcal{K}_1^{\delta}(t)-\int_0^t\int_{\mathcal{Y}}\sqrt{\gamma}\partial_y\Phi_ g(\bar{X}_s,\mathcal{L}_{\bar{X}_s},y)\bar{h}^2_s(y)\nu^{\bar{X}_s,\mathcal{L}_{\bar{X}_s}}(dy)ds\Big|=0
\end{equation}
and
\begin{equation}\label{hfx9}
\lim_{\delta\rightarrow0}\mathbb{E}\sup_{t\in[0,T]}\Big|\mathcal{K}_5^{\delta}(t)-\int_0^t\int_{\mathcal{Y}}\sigma(\bar{X}_s,\mathcal{L}_{\bar{X}_s})\bar{h}^1_s(y)\nu^{\bar{X}_s,\mathcal{L}_{\bar{X}_s}}(dy)ds\Big|=0.
\end{equation}
Note that collecting the calculations of Steps 2 and 3 in Proposition 4.1, (\ref{hfx8}) holds if we can prove
\begin{eqnarray}\label{hfx10}
\lim_{\delta\rightarrow0}~~~\!\!\!\!\!\!\!\!&&\mathbb{E}\sup_{t\in[0,T]}\Bigg|\frac{\sqrt{\varepsilon}}{\sqrt{\delta}}\int^t_0\partial_y \Phi_g(X_s^{\delta,\bar{h}^\delta} ,\mathcal{L}_{X^{\delta}_s},Y^{\delta,\bar{h}^\delta}_s)\bar{h}^2_s(\bar{Y}_s^{\delta}) ds
\nonumber\\
\!\!\!\!\!\!\!\!&&
-\int_0^t\int_{\mathcal{Y}}\sqrt{\gamma}\partial_y\Phi_ g(\bar{X}_s,\mathcal{L}_{\bar{X}_s},y)\bar{h}^2_s(y)\nu^{\bar{X}_s,\mathcal{L}_{\bar{X}_s}}(dy)ds\Bigg|=0.
\end{eqnarray}
On the one hand, it is clear that $\partial_y\Phi_g(x,\mu,y)$ is  Lipschitz continuous and bounded. Since $\bar{h}(y)$ is  Lipschitz continuous  and bounded in $y\in\mathcal{Y}$, we obtain  $\partial_y\Phi_g(x,\mu,y)\bar{h}(z)$ is also  Lipschitz continuous w.r.t.~$(x,\mu,y,z)$ and  $\sigma(x,\mu)\bar{h}(y)$ is locally  Lipschitz continuous. Finally, (\ref{hfx9}) and (\ref{hfx10}) follow from same argument as in the proof of \cite[(3.38)]{HLLS}.

Collecting the arguments above, by the continuity of $\partial_x\bar{b}$ in (\ref{hfxj1}) and  Gronwall's inequality, we  deduce that (\ref{e2}) holds. }
\hspace{\fill}$\Box$

\begin{funding}
W. Hong is supported by  NSFC (No.~12401177),  NSF of Jiangsu Province (No.~BK20241048).   S. Li is supported by NSFC (No.~12371147) and the PAPD of Jiangsu Higher Education Institutions.
\end{funding}

\vspace{5mm}
\noindent\textbf{Acknowledgements} { The authors would like to thank the referees for  very constructive suggestions
and valuable comments.}

\section*{Statements and Declarations}
\noindent\textbf{Data availability:} Data sharing not applicable to this article as no datasets were generated or analysed during the current study.

\noindent\textbf{Conflict of interest:}  On behalf of all authors, the corresponding author states that there is no conflict of interest.

\end{document}